\documentclass[a4paper,12pt]{article}
\usepackage{latexsym}		
\usepackage{amsmath}
\usepackage{amsthm}

\usepackage{graphicx}
\usepackage{txfonts}
\usepackage{bm}
\usepackage{color}
\usepackage{epic,eepic}

\newcommand{\RED}[1]{{\color{red}#1}} 
 \renewcommand{\RED}[1]{{#1}} 
\newcommand{\BLU}[1]{{\color{blue}#1}} 
 \renewcommand{\BLU}[1]{{#1}} 

\usepackage{geometry}
\newcommand{\qedJIAM}{} 
\newcommand{\finboxARX}{\finbox} 

\numberwithin{equation}{section}

\newtheorem{theorem}{Theorem}[section]
\newtheorem{proposition}{Proposition}[section]
\newtheorem{corollary}{Corollary}[section]
\newtheorem{remark}{Remark}[section]
\newtheorem{example}{Example}[section]

\newcommand{\memo}[1]{{\bf \small \RED{[MEMO:}} \BLU{#1} \ {\bf \small \RED{:end]} }}  
   \renewcommand{\memo}[1]{}           
\newcommand{\OMIT}[1]{{\bf [OMIT:} #1 \ {\bf --- end OMIT] }}  
   \renewcommand{\OMIT}[1]{}            

\newcommand{\RR}{{\mathbb{R}}}
\newcommand{\ZZ}{{\mathbb{Z}}}

\newcommand{\vecone}{{\bf 1}}
\newcommand{\veczero}{{\bf 0}}
\newcommand{\dom}{{\rm dom\,}}

\newcommand{\unitvec}[1]{\bm{1}\sp{#1}}

\newcommand{\argmin}{\arg \min}
\newcommand{\conv}{\Box\,}

\newcommand{\subgR}{\partial}
\newcommand{\Lnat}{{L$^{\natural}$}}
\newcommand{\Mnat}{{M$^{\natural}$}}
\newcommand{\LLnat}{{L$^{\natural}_{2}$}}
\newcommand{\MMnat}{{M$^{\natural}_{2}$}}
\newcommand{\LL}{{L$_{2}$}}
\newcommand{\MM}{{M$_{2}$}}

\newcommand{\kwd}[1]{{\em #1\/}}  
\newcommand{\finbox}{\hspace*{\fill}$\rule{0.2cm}{0.2cm}$}
\newcommand{\todaye}{\the\year/\the\month/\the\day}

\begin{document}

\title{
Recent Progress on Integrally Convex Functions
}

\author{
Kazuo Murota%
\thanks{
The Institute of Statistical Mathematics,
Tokyo 190-8562, Japan; 
and
Faculty of Economics and Business Administration,
Tokyo Metropolitan University, 
Tokyo 192-0397, Japan,
murota@tmu.ac.jp}
\ and 
Akihisa Tamura%
\thanks{Department of Mathematics, Keio University, 
Yokohama 223-8522, Japan,
aki-tamura@math.keio.ac.jp}
}

\date{November 2022 / February 2023}

\maketitle

\begin{abstract}
Integrally convex functions constitute a fundamental function class
in discrete convex analysis,
including M-convex functions, L-convex functions,
and many others.
This paper aims at a rather comprehensive survey of
recent results on integrally convex functions
with some new technical results.
Topics covered in this paper
include
characterizations of integral convex sets and functions,
operations on integral convex sets and functions,
optimality criteria for minimization with a proximity-scaling algorithm,
integral biconjugacy, and the discrete Fenchel duality.
While the theory of M-convex and L-convex functions
has been built upon fundamental results on matroids and submodular functions,
developing the theory of integrally convex functions
requires more general and basic tools such as 
the Fourier--Motzkin elimination.
\end{abstract}


{\bf Keywords}:
Discrete convex analysis,  
Integrally convex function, 
Box-total dual integrality,
Fenchel duality, 
Integral subgradient, 
Minimization.



\newpage



\section{Introduction}

Discrete convex analysis \cite{Mdca98,Mdcasiam} has found 
applications and connections to a wide variety of disciplines.
Early applications include those to 
combinatorial optimization, matrix theory, and economics,
as described in the book 
\cite{Mdcasiam} 
and a survey paper \cite{Mbonn09}.
More recently, there have been
active interactions with 
operations research (inventory theory)
\cite{CM21mnat,CL21dca,SCB14},
economics and game theory
\cite{Mdcaeco16,PLem17gs,ST15jorsj,Tam09book},
and algebra (Lorentzian polynomials, in particular) \cite{BH20lor}.
The reader is referred to
\cite{Mdcasiam}
for basic concepts and terminology in discrete convex analysis, and to 
\cite{Fuj05book,FT22twocnv,MM15dcprog,MM19multm,MM19projcnvl,MM22L2poly,Msurvop21,Mopernet21,Shi17L}
for recent theoretical and algorithmic developments.

Integrally convex functions,
due to Favati--Tardella \cite{FT90},
constitute a fundamental function class
in discrete convex analysis.
A function
$f: \ZZ^{n} \to \RR \cup \{ +\infty \}$
is called {\em integrally convex}
if its local convex extension 
$\tilde{f}: \RR^{n} \to \RR \cup \{ +\infty \}$  
is (globally) convex in the ordinary sense, where
$\tilde{f}$
is defined as the collection of convex extensions of $f$ in each 
unit hypercube 
$\{ x \in \RR\sp{n} \mid a_{i} \leq x_{i} \leq a_{i} + 1 \ (i=1,\ldots, n) \}$ 
with $a \in \ZZ^{n}$; see Section~\ref{SCicfndef} for more precise statements.
A subset of $\ZZ^{n}$ is called 
integrally convex if its indicator function
$\delta_{S}: \ZZ\sp{n} \to \{ 0, +\infty \}$
($\delta_{S}(x) =  0$ for $x \in S$ and
$= + \infty$ for $x \notin S$)
is an integrally convex function.
The concept of integral convexity is
used in formulating discrete fixed point theorems
\cite{Iim10,IMT05,IY09,Yan09fixpt}, 
and designing solution algorithms for discrete systems of nonlinear equations
\cite{LTY11nle,Yan08comp}.
In game theory 
the integral concavity of payoff functions 
guarantees the existence of a pure strategy equilibrium 
in finite symmetric games \cite{IW14}.

Integrally convex functions serve as
a common framework for discrete convex functions.
Indeed, separable convex,
{\rm L}-convex, \Lnat-convex, {\rm M}-convex,  
\Mnat-convex,  \LLnat-convex, and 
\MMnat-convex functions are known to be integrally convex \cite{Mdcasiam}.
Multimodular functions \cite{Haj85} 
are also integrally convex, 
as pointed out in \cite{Mdcaprimer07}.
Moreover, BS-convex and UJ-convex functions \cite{Fuj14bisubmdc}
are integrally convex.
Discrete midpoint convex functions 
\cite{MMTT20dmc}
and directed discrete midpoint convex functions 
\cite{TT21ddmc} are also integrally convex.
The relations among those discrete convexity concepts
are investigated in \cite{MM21inclinter,MS01rel}. 
Figure~\ref{FGdcsetclassTDI} 
is an overview of the inclusion relations 
among the most fundamental classes of discrete convex sets.
It is noted that the class of integrally convex sets contains 
the other classes of discrete convex sets.

\begin{figure}[b]
\centering
\includegraphics[height=50mm]{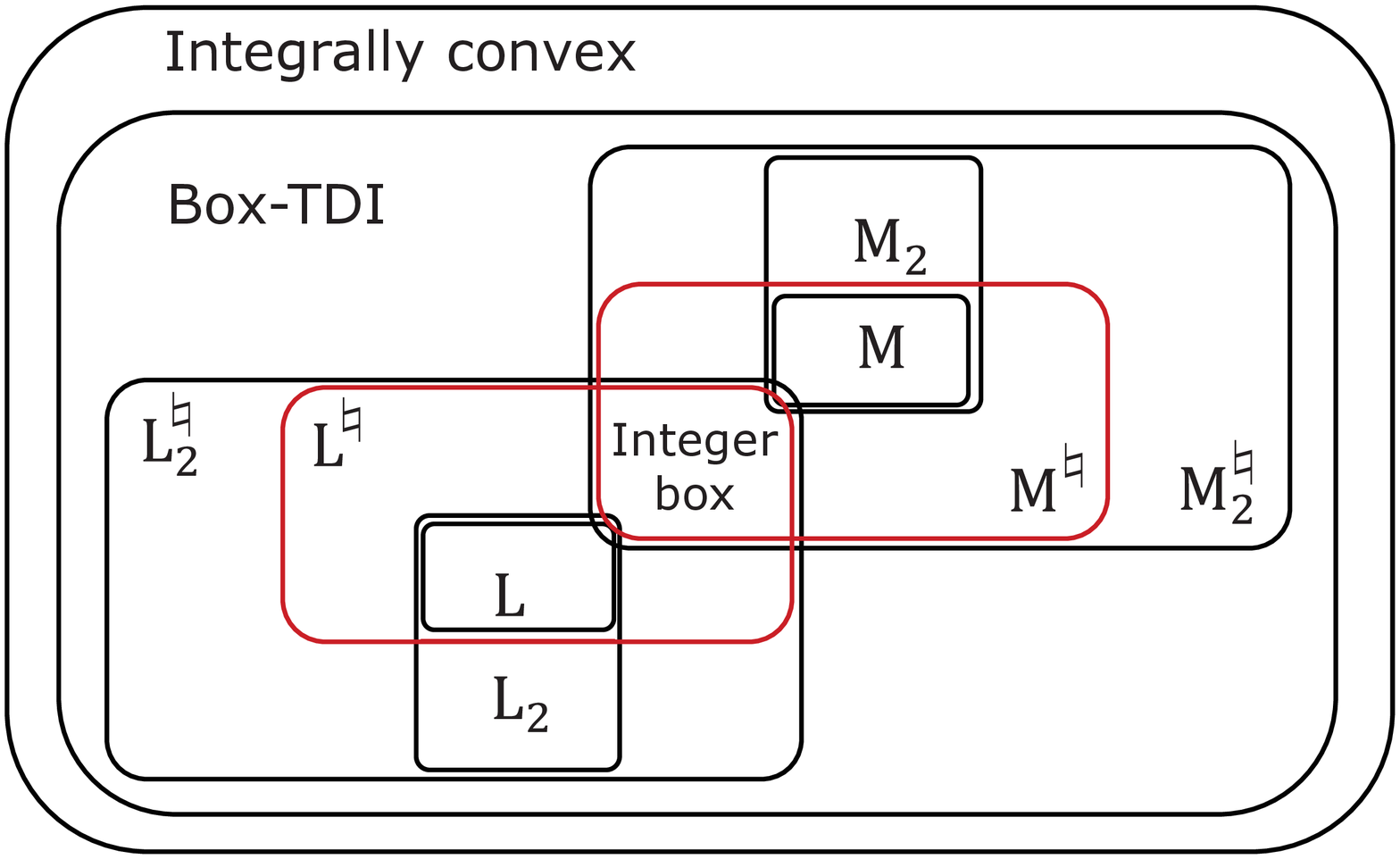}
\\
(\LLnat-convex $\cap$ \MMnat-convex $=$ 
\Lnat-convex $\cap$ \Mnat-convex $=$ Integer box)
\caption{Classes of discrete convex sets}
\label{FGdcsetclassTDI}
\end{figure}

In the last several years,
significant progress has been made 
in the theory of integrally convex functions.
A proximity theorem for integrally convex functions
is established in \cite{MMTT19proxIC} 
together with a proximity-scaling algorithm for minimization.
Fundamental operations for integrally convex functions
such as projection and convolution are investigated in \cite{MM19projcnvl,Msurvop21,Mopernet21}.
It is revealed that integer-valued integrally convex functions
enjoy integral biconjugacy \cite{MT20subgrIC},
and a discrete Fenchel-type min-max formula is established  
for a pair of integer-valued integrally convex and separable convex functions
\cite{MT22ICfenc}. 
The present paper aims at a rather comprehensive survey of
those recent results on integrally convex functions
with some new technical results.
While the theory of M-convex and L-convex functions
has been built upon fundamental results on matroids and submodular functions,
developing the theory of integrally convex functions
requires more general and basic tools such as 
the Fourier--Motzkin elimination.

This paper is organized as follows.
In Section~\ref{SCsetconcept} we review integrally convex sets,
with new observations on their polyhedral properties.
In Section~\ref{SCfnconcept} 
the concept of integrally convex functions is reviewed 
with emphasis on their characterizations.
Section~\ref{SCminiztn} deals with properties
related to minimization and minimizers, including a proximity-scaling algorithm.
Section~\ref{SCsubgrbiconj} is concerned with integral subgradients and biconjugacy,
and Section~\ref{SCfenc} with the discrete Fenchel duality.



\section{Integrally convex sets}
\label{SCsetconcept}

\subsection{Hole-free property}
\label{SCintroset}

Let $n$ be a positive integer and 
$N = \{ 1,2,\ldots, n  \}$.
For a subset $I$ of $N$, we denote by $\unitvec{I}$ 
the characteristic vector of $I$;
the $i$th component of $\unitvec{I}$ is equal to 1 or 0 according 
to whether $i \in I$ or not.
We use a short-hand notation $\unitvec{i}$ for $\unitvec{ \{ i \} }$,
which is the $i$th unit vector. 
The vector with all components equal to 1 is denoted by $\vecone$,
that is, $\vecone =(1,1,\ldots, 1) = \unitvec{N}$.

For two vectors 
$\alpha \in (\RR \cup \{ -\infty \})\sp{n}$ and 
$\beta \in (\RR \cup \{ +\infty \})\sp{n}$ with $\alpha \leq \beta$,
we define notation
$ [\alpha,\beta]_{\RR} = \{ x \in \RR\sp{n} \mid \alpha \leq x \leq \beta \}$,
which represents the set of real vectors between $\alpha$ and $\beta$. 
An {\em integral box} will mean
a set $B$ of real vectors represented as
$B =  [\alpha,\beta]_{\RR}$
for integer vectors $\alpha \in (\ZZ \cup \{ -\infty \})\sp{n}$ and 
$\beta \in (\ZZ \cup \{ +\infty \})\sp{n}$ with $\alpha \leq \beta$.
The set of integer vectors contained in an integral box
will be called a {\em box of integers}
or an {\em interval of integers}.
We use notation
$[\alpha,\beta]_{\ZZ} : = [\alpha,\beta]_{\RR} \cap \ZZ\sp{n} = 
\{ x \in \ZZ\sp{n} \mid \alpha \leq x \leq \beta \}$
for $\alpha \in (\ZZ \cup \{ -\infty \})\sp{n}$ and 
$\beta \in (\ZZ \cup \{ +\infty \})\sp{n}$ with $\alpha \leq \beta$.

For a subset $S$ of $\RR\sp{n}$, we denote its {\em convex hull} by 
$\overline{S}$, which is, by definition, the smallest convex set containing $S$.
As is well known, $\overline{S}$ coincides with the set of 
all convex combinations of (finitely many) elements of $S$.
We say that a set
$S \subseteq \ZZ\sp{n}$ is
\kwd{hole-free}\index{hole-free} if
\begin{equation} \label{convset1def}
  S = \overline{S} \cap \ZZ\sp{n} .
\end{equation}
Since the inclusion $S \subseteq \overline{S} \cap \ZZ\sp{n}$ is trivially true for any $S$,
the content of this condition lies in
\begin{equation} \label{convset2def}
  S \supseteq \overline{S} \cap \ZZ\sp{n} ,
\end{equation}
stating that the integer points contained in
the convex hull of $S$ all belong to $S$ itself.
A finite set of integer points is hole-free 
if and only if it is the set of integer points
in some integral polytope.

For a set $S \subseteq \RR\sp{n}$ we define its 
\kwd{indicator function}\index{indicator function}
$\delta_{S}: \RR\sp{n} \to \{ 0, +\infty \}$
by
\begin{equation} \label{indicatorZdef}
\delta_{S}(x)  =
   \left\{  \begin{array}{ll}
    0            &   (x \in S) ,     \\
   + \infty      &   (x \not\in S) . \\
                      \end{array}  \right.
\end{equation}

\begin{remark} \rm  \label{RMconvhull}
In a standard textbook \cite[Section A.1.3]{HL01},
the convex hull of 
a subset $S$ of the $n$-dimensional Euclidean space $\RR\sp{n}$
is denoted by $\mathrm{co}\, S$
and the {\em closed convex hull} 
by $\overline{\mathrm{co}}\, S$,
where
the closed convex hull of $S$
is defined to be the intersection of all closed convex set containing $S$.
It is known that 
$\overline{\mathrm{co}}\, S$ coincides with the (topological) closure of 
$\mathrm{co}\, S$, which is expressed as
$\overline{\mathrm{co}}\, S = \mathrm{cl}(\mathrm{co}\, S)$
with the use of notation $\mathrm{cl}$ 
for closure operation. 
Using our notation $\overline{S}$, we have
$\mathrm{co}\, S = \overline{S}$
and 
$\overline{\mathrm{co}}\, S = \mathrm{cl}(\overline{S})$. 
For a finite set $S$, we have
$\mathrm{co}\, S = \overline{\mathrm{co}}\, S$.
To see the difference of 
$\mathrm{co}\, S$ and $\overline{\mathrm{co}}\, S$
for an infinite set $S$,
consider 
$S = \{ (0,1) \} \cup \{ (k,0) \mid k \in \ZZ \}$.
The convex hull is 
$\mathrm{co}\, S = \overline{S} = \{ (0,1) \} \cup \{ (x_{1},x_{2}) \mid 0 \leq x_{2} < 1 \}$
and the closed convex hull is 
$\overline{\mathrm{co}}\, S = \mathrm{cl}(\overline{S}) 
= \{ (x_{1},x_{2}) \mid 0 \leq x_{2} \leq 1 \}$.
We have 
$\overline{S} \cap \ZZ\sp{2} =S$,
which shows that this set $S$ is hole-free,
while
$\mathrm{cl}(\overline{S})  \cap \ZZ\sp{2} \ne S$.
\finbox 
\end{remark}

\subsection{Definition of integrally convex sets}
\label{SCintconvsetdef}

For $x \in \RR^{n}$
the \kwd{integral neighborhood} of $x$
is defined by
\begin{equation} \label{Nxdef}
N(x) = \{ z \in \mathbb{Z}^{n} \mid | x_{i} - z_{i} | < 1 \ (i=1,2,\ldots,n)  \} .
\end{equation}
It is noted that 
strict inequality ``\,$<$\,'' is used in this definition
and 
$N(x)$ admits an alternative expression
\begin{equation}  \label{intneighbordeffloorceil}
N(x) = \{ z \in \ZZ\sp{n} \mid
\lfloor x_{i} \rfloor \leq  z_{i} \leq \lceil x_{i} \rceil  \ \ (i=1,2,\ldots, n) \} ,
\end{equation}
where, for $t \in \RR$ in general, 
$\left\lfloor  t  \right\rfloor$
denotes the largest integer not larger than $t$
(rounding-down to the nearest integer)
and 
$\left\lceil  t   \right\rceil$ 
is the smallest integer not smaller than $t$
(rounding-up to the nearest integer).
That is,
$N(x)$ consists of all integer vectors $z$ 
between
$\left\lfloor x \right\rfloor 
=( \left\lfloor x_{1} \right\rfloor ,\left\lfloor x_{2} \right\rfloor , 
  \ldots, \left\lfloor x_{n} \right\rfloor)$ 
and 
$\left\lceil x \right\rceil 
= ( \left\lceil x_{1} \right\rceil, \left\lceil x_{2} \right\rceil,
   \ldots, \left\lceil x_{n} \right\rceil)$.
See Fig.~\ref{FGneighbor} for $N(x)$ when $n=2$.

\begin{figure}\begin{center}
\includegraphics[width=0.5\textwidth,clip]{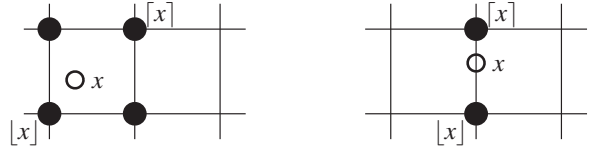}
\caption{Integral neighborhood $N(x)$ of $x$ \ \ 
{\rm (}$\bullet$: point of $N(x)${\rm )}.}
\label{FGneighbor}
\end{center}\end{figure}

For a set $S \subseteq \ZZ^{n}$
and $x \in \RR^{n}$
we call the convex hull of $S \cap N(x)$ 
the {\em local convex hull} of $S$ around $x$.
A nonempty set
$S \subseteq \ZZ^{n}$ is said to be 
\kwd{integrally convex}\index{integrally convex set}
if
the union of the local convex hulls $\overline{S \cap N(x)}$ over $x \in \RR^{n}$ 
is convex \cite{Mdcasiam}.
In other words, a set $S \subseteq \ZZ^{n}$ is called 
integrally convex if
\begin{equation}  \label{icsetdef0}
 \overline{S} = \bigcup_{x \in \RR\sp{n}} \overline{S \cap N(x)}.
\end{equation}
Since the inclusion 
\ $\overline{S} \supseteq \bigcup_{x} \overline{S \cap N(x)}$ \
is trivially true for any $S$,
the content of the condition \eqref{icsetdef0} lies in 
\begin{equation}  \label{icsetdef0ess}
 \overline{S} \subseteq \bigcup_{x \in \RR\sp{n}} \overline{S \cap N(x)}.
\end{equation}

\begin{example} \rm  \label{EXintcnvsetconcept}
The concept of integrally convex sets is illustrated by simple examples.
The six-point set 
in Fig.~\ref{FGintconvset}(a) is integrally convex.
The removal of the middle point 
(Fig.~\ref{FGintconvset}(b))
breaks integral convexity
(cf., Proposition~\ref{PRicvsetholefree}).
The four-point set $S = \{  \mbox{A, B, C, D} \}$
in Fig.~\ref{FGintconvset}(c) is not integrally convex,
since its convex hull $\overline{S}$, which is the triangle ACD,
does not coincide with the union of the local convex hulls
$\overline{S \cap N(x)}$,
which is the union of the line segment AB and the triangle BCD.
\finbox 
\end{example}

\begin{figure}\begin{center}
\includegraphics[width=0.78\textwidth,clip]{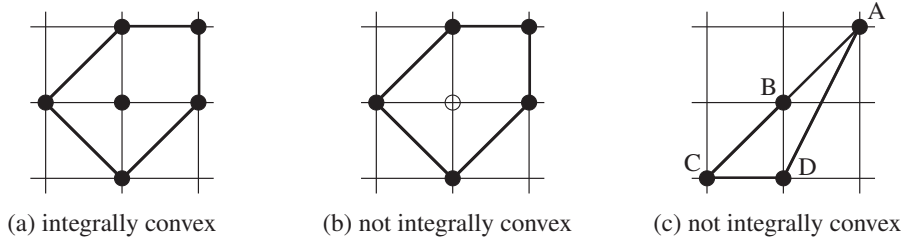}
\caption{Concept of integrally convex sets.}
\label{FGintconvset}
\end{center}\end{figure}

\begin{example} \rm  \label{EXintcnvsetObvEx}
Obviously,
every subset of $\{ 0, 1\}\sp{n}$ is integrally convex,
and every interval of integers $(\subseteq \ZZ\sp{n})$
is integrally convex.
\finbox 
\end{example}

Integral convexity can be defined by seemingly different conditions.
Here we mention the following two.
\begin{itemize}
\item
Every point $x$
in the convex hull of $S$ is contained in the convex hull of 
$S \cap N(x)$, i.e.,
\begin{equation}  \label{icsetdef1}
 x \in \overline{S} \ \Longrightarrow  x \in  \overline{S \cap N(x)} .
\end{equation}

\item
For each $x \in \RR\sp{n}$,
the intersection of the convex hulls of $S$ and $N(x)$
is equal to the convex hull of the intersection of $S$ and $N(x)$, i.e.,
\begin{equation}\label{icsetdef2}
 \overline{S} \cap  \overline{N(x)} =  \overline{S \cap N(x)}.
\end{equation}
Since the inclusion 
\ $\overline{S} \cap  \overline{N(x)} \supseteq  \overline{S \cap N(x)}$ \ 
is trivially true for any $S$ and $x$,
the content of this condition lies in 
\begin{equation}\label{icsetdef2ess}
 \overline{S} \cap  \overline{N(x)} \subseteq  \overline{S \cap N(x)}.
\end{equation}
\end{itemize}

The following proposition states the equivalence of the 
five conditions \eqref{icsetdef0} to \eqref{icsetdef2ess}.
Thus, any one of these conditions characterizes 
integral convexity of a set of integer points.

\begin{proposition}  \label{PRicvset}
For $S \subseteq \mathbb{Z}\sp{n}$, 
the five conditions \eqref{icsetdef0} to \eqref{icsetdef2ess} are all equivalent.
\end{proposition}
\begin{proof}
The proof is easy and straightforward, but we include it for completeness.
We already mentioned the equivalences 
[\eqref{icsetdef0}$\Leftrightarrow$\eqref{icsetdef0ess}]
and
[\eqref{icsetdef2}$\Leftrightarrow$\eqref{icsetdef2ess}].

[\eqref{icsetdef0ess}$\Rightarrow$\eqref{icsetdef1}]:
Take any
$x \in \overline{S}$.
By \eqref{icsetdef0ess} we obtain \ 
$x \in  \overline{S \cap N(y)}$ \ 
for some $y$.
This means that 
there exist integer points $x\sp{(1)}, x\sp{(2)}, \ldots, x\sp{(m)} \in S \cap N(y)$
and coefficients  $\lambda_{1}, \lambda_{2}, \ldots, \lambda_{m} > 0$ such that
$x = \sum_{k=1}\sp{m} \lambda_{k} \, x\sp{(k)}$
and $\sum_{k=1}\sp{m} \lambda_{k} = 1$.
Since all the points 
$x$, $x\sp{(1)}, x\sp{(2)}, \ldots, x\sp{(m)}$ 
lie between
$\left\lfloor y \right\rfloor$ and $\left\lceil y \right\rceil$,
we must have
$x\sp{(k)} \in N(x)$ for all $k$.
Therefore, $x \in  \overline{S \cap N(x)}$.

[\eqref{icsetdef1}$\Rightarrow$\eqref{icsetdef0ess}]:
Take any
$y \in \overline{S}$.
By \eqref{icsetdef1} we obtain
$y \in  \overline{S \cap N(y)}$,
whereas 
$\overline{S \cap N(y)} \subseteq \bigcup_{x \in \RR\sp{n}} \overline{S \cap N(x)}$.
Thus \eqref{icsetdef0ess} is shown.

[\eqref{icsetdef1}$\Rightarrow$\eqref{icsetdef2ess}]:
Take any
$y \in \overline{S} \cap  \overline{N(x)}$.
By $y \in \overline{S}$ and \eqref{icsetdef1} we obtain
$y \in  \overline{S \cap N(y)}$,
whereas 
$y \in \overline{N(x)}$
implies
$N(y) \subseteq N(x)$.
Hence
$y \in  \overline{S \cap N(x)}$.

[\eqref{icsetdef2ess}$\Rightarrow$\eqref{icsetdef1}]:
Take any
$x \in \overline{S}$.
Since 
$x \in \overline{N(x)}$, we have
$x \in \overline{S} \cap  \overline{N(x)}$.
Then
$x \in \overline{S \cap N(x)}$
by \eqref{icsetdef2ess}.
\qedJIAM
\end{proof}

As an application of Proposition~\ref{PRicvset} 
we give a formal proof to the statement that
an integrally convex set is hole-free.

\begin{proposition}  \label{PRicvsetholefree}
For an integrally convex set $S$, we have
$S = \overline{S} \cap \ZZ\sp{n}$.
\end{proposition}
\begin{proof}  
It suffices to show 
$\overline{S} \cap \ZZ\sp{n} \subseteq S$
in \eqref {convset2def}.
Take any $x \in \overline{S} \cap \ZZ\sp{n}$.
Since
$N(x) = \{ x \}$, 
\eqref{icsetdef1}
shows $x \in  \overline{S \cap  \{ x \}}$, which implies $x \in S$.
\qedJIAM
\end{proof}

The following (new) characterization of an integrally convex set 
is often useful,
which is used indeed in the proof of Theorem~\ref{THicchardmc} in Section~\ref{SCcharICfn}.
While the condition \eqref{icsetdef1} refers to all real vectors $x$,
the condition \eqref{intcnvsetdist234} below restricts $x$ to 
the midpoints of two vectors in $S$.

\begin{theorem} \label{THicSetmidpt}
A nonempty set $S \subseteq \ZZ\sp{n}$
is integrally convex if and only if
\begin{equation}  \label{intcnvsetdist234}
\frac{y + y'}{2} \ \in \  \overline{S \cap N \bigg(\frac{y + y'}{2} \bigg)}
\end{equation}
for every $y, y' \in S$ with $\| y - y' \|_{\infty} \geq 2$.
\end{theorem}
\begin{proof}
The only-if-part is obvious from \eqref{icsetdef1}, since
$(y + y')/2 \in \overline{S}$.
To prove the if-part, let
$x \in \overline{S}$.
This implies the existence of $y\sp{(1)},y\sp{(2)},\ldots, y\sp{(m)} \in S$ such that
\begin{equation} \label{xyifxfyi-2}
 x =  \sum_{i=1}\sp{m} \lambda_{i} y\sp{(i)} ,  
\end{equation}
where
$\sum_{i=1}\sp{m}  \lambda_{i} = 1$ and 
$\lambda_{i} > 0 \ (i=1,2,\ldots, m)$.
In the following we modify the generating points 
$y\sp{(1)},y\sp{(2)},\ldots, y\sp{(m)}$
repeatedly and eventually arrive at 
an expression of the form \eqref{xyifxfyi-2}
with the additional condition that
$y\sp{(i)} \in N(x)$
for all $i$,
showing
$x \in \overline{S \cap N(x)}$.
Then the integral convexity of $S$
is established by \eqref{icsetdef1}.

For each $j=n,n-1,\ldots, 1$, we look at the 
$j$-th component of
$y\sp{(1)},y\sp{(2)}, \ldots,  \allowbreak  y\sp{(m)}$.
Let $j=n$ and define
\begin{equation} \label{alphabetaIminImax}
\alpha_{n} := \min_{i} y\sp{(i)}_{n},
\ 
\beta_{n} := \max_{i} y\sp{(i)}_{n},
\
I_{\min} := \{ i \mid y\sp{(i)}_{n} = \alpha_{n} \},
\
I_{\max} := \{ i \mid y\sp{(i)}_{n} = \beta_{n} \} .
\end{equation}
If $\beta_{n} - \alpha_{n} \leq 1$, we are done with $j=n$.
Suppose that 
$\beta_{n} - \alpha_{n} \geq 2$.
By translation and reversal of the $n$-th coordinate,
we may assume $0 \leq  x_{n} \leq  1$,
$\alpha_{n} \leq 0$, and
$\beta_{n} \geq 2$.
By renumbering the generators we may assume
$1 \in  I_{\min}$
and
$2 \in  I_{\max}$,
i.e., $y\sp{(1)}_{n} = \alpha_{n}$ and $y\sp{(2)}_{n} = \beta_{n}$.
We have
$\| y\sp{(1)} - y\sp{(2)} \|_{\infty} \geq 2$. 

By \eqref{intcnvsetdist234} for $(y\sp{(1)}, y\sp{(2)})$ 
we have
\begin{equation} \label{y1y2zk-2}
\frac{y\sp{(1)} + y\sp{(2)}}{2} =  \sum_{k=1}\sp{l}  \mu_{k} z\sp{(k)},
\quad
z\sp{(k)} \in S \cap N \bigg( \frac{y\sp{(1)} + y\sp{(2)}}{2} \bigg)
\quad
(k=1,2,\ldots, l)
\end{equation}
with 
$\mu_{k} > 0$ \  $(k=1,2,\ldots, l)$ and
$\sum_{k=1}\sp{l}  \mu_{k} = 1$.
With notation $\lambda = \min(\lambda_{1}, \lambda_{2})$,
it follows from 
\eqref{xyifxfyi-2}
and
\eqref{y1y2zk-2} that
\[
 x =  (\lambda_{1} - \lambda ) y\sp{(1)} + (\lambda_{2}-\lambda) y\sp{(2)}
 + 2 \lambda \sum_{k=1}\sp{l} \mu_{k} z\sp{(k)}
+  \sum_{i=3}\sp{m} \lambda_{i} y\sp{(i)} ,
\]
which is another representation of the form \eqref{xyifxfyi-2}.

With reference to this new representation
we define
$\hat \alpha_{n}$, $\hat \beta_{n}$,
$\hat{I}_{\min}$, and 
$\hat{I}_{\max}$, 
as in \eqref{alphabetaIminImax}.
Since $\beta_{n} - \alpha_{n} \geq 2$, we have 
\[
 \alpha_{n} + 1 \leq (y\sp{(1)}_{n} + y\sp{(2)}_{n})/2  \leq \beta_{n} - 1  ,
\]
which implies $\alpha_{n} + 1 \leq z\sp{(k)}_{n} \leq \beta_{n} - 1$ for all $k$.
Hence, 
$\alpha_{n} \leq \hat \alpha_{n}$ and
$\hat \beta_{n} \leq \beta_{n}$.
Moreover, if 
$(\hat \alpha_{n}, \hat \beta_{n})=(\alpha_{n},\beta_{n})$,
then 
$|\hat{I}_{\min}| +  |\hat{I}_{\max}|  \leq |I_{\min}| +  |I_{\max}| - 1$.
Therefore, by repeating the above process with $j=n$, 
 we eventually arrive at 
a representation of the form of \eqref{xyifxfyi-2} 
with $\beta_{n} - \alpha_{n} \leq 1$.

We next apply the above procedure for the $(n-1)$-st component.
What is crucial here is that 
the condition $\beta_{n} - \alpha_{n} \leq 1$
is maintained in the modification 
of the generators
via \eqref{y1y2zk-2}
for the $(n-1)$-st component.
Indeed, 
for each $k$, the inequality
$\alpha_{n} \leq z\sp{(k)}_{n} \leq \beta_{n}$ 
follows from 
$\alpha_{n} \leq (y\sp{(1)}_{n} +y\sp{(2)}_{n})/2 \leq \beta_{n}$
and
$z\sp{(k)} \in N( (y\sp{(1)} + y\sp{(2)})/2 )$.
Therefore, we can obtain 
a representation of the form of \eqref{xyifxfyi-2} 
with 
$\beta_{n} - \alpha_{n} \leq 1$ and $\beta_{n-1} - \alpha_{n-1} \leq 1$,
where
$\alpha_{n-1} = \min_{i} y\sp{(i)}_{n-1}$ and $\beta_{n-1} = \max_{i} y\sp{(i)}_{n-1}$.

Then we continue the above process for $j=n-2,n-3, \ldots,1$,
to finally obtain a representation of the form of \eqref{xyifxfyi-2}
with $|y\sp{(i)}_{j} - y\sp{(i')}_{j}| \leq 1$
for all $i, i'$ and $j=1,2,\dots,n$.
This means, in particular, that 
$y\sp{(i)} \in S \cap N(x)$
for all $i$.
\qedJIAM
\end{proof}

\subsection{Polyhedral aspects}
\label{SCpolyh}

A subset of $\RR\sp{n}$ is called a 
\kwd{polyhedron}\index{polyhedron}
if it is described by a finite number of linear inequalities.
A polyhedron is said to be
\kwd{rational}\index{rational polyhedron}\index{polyhedron!rational}
if it is described by a finite number of linear inequalities with
rational coefficients.
A 
polyhedron is called an
\kwd{integer polyhedron}\index{integer polyhedron}%
\index{polyhedron!integer}\index{integrality!polyhedron}
if $P=\overline{P \cap \ZZ\sp{n}}$, i.e., if
it coincides with the convex hull
of the integer points contained in it,
or equivalently, if
$P$ is rational and each face of $P$ contains an integer vector. 
See \cite{Sch86,Sch03} for terminology about polyhedra.
For two vectors $a, x  \in \RR\sp{n}$, we use notation
$\langle a, x \rangle = a\sp{\top} x = \sum_{i=1}\sp{n} a_{i} x_{i}$.

The convex hull of an integrally convex set  
is an integer polyhedron
(see \cite[Sec.~4.1]{MT20subgrIC} for a rigorous proof).
However, not much is known about the 
inequality system $Ax \leq  b$
to describe an integrally convex set.
This is not surprising because
every subset of $\{ 0, 1\}\sp{n}$ is integrally convex 
(as noted in Example~\ref{EXintcnvsetObvEx}),
and most of the NP-hard combinatorial optimization problems 
can be formulated on $\{ 0, 1\}\sp{n}$ polytopes.

When $n = 2$, the following fact is easy to see.

\begin{proposition}[\cite{MMTT19proxIC}]  \label{PRintcnvsetdim2}
A set $S \subseteq \mathbb{Z}^{2}$ is integrally convex
if and only if it can be represented as
$S = \{ (x_{1},x_{2}) \in \mathbb{Z}^{2} \mid 
a^{i} x_{1} +  b^{i} x_{2} \leq c^{i} \ (i=1,2,\ldots,m) \}$ 
for some $a^{i}, b^{i} \in \{ -1,0,1 \}$ and $c^{i} \in \ZZ$
$(i=1,2,\ldots,m)$.
\finboxARX
\end{proposition}

Polyhedral descriptions 
are known for major subclasses of integrally convex sets.
The present knowledge 
for various kinds of discrete convex sets
is summarized in Table~\ref{TBpolydesc},
which shows the possible forms of the vector $a$ 
for an inequality $a\sp{\top} x  \leq b$
to describe the convex hull $\overline{S}$ of a discrete convex set $S$.  
It should be clear that each vector $a$ corresponds (essentially) 
to the normal vector of a face of $\overline{S}$.
Since an \MM-convex (resp., \MMnat-convex) set is,
by definition,
the intersection of two 
M-convex (resp., \Mnat-convex) sets \cite{Mdcasiam},
the polyhedral description of an \MM-convex (resp., \MMnat-convex) set
is obtained immediately as the union of the 
inequality systems for the constituent M-convex (resp., \Mnat-convex) sets.

\begin{table}
\begin{center}
\caption{Polyhedral descriptions of discrete convex sets \cite{MM22L2poly}.}
\label{TBpolydesc}
\medskip
\renewcommand{\arraystretch}{1.1}%
\begin{tabular}{l|lc}
  & Vector $a$ for $a\sp{\top} x  \leq b$  &   Ref. 
\\ \hline
 Box (interval) & $\pm \unitvec{i}$    & obvious 
\\ 
L-convex set  & $\unitvec{j} - \unitvec{i}$   & \cite[Sec.5.3]{Mdcasiam}
\\ 
\Lnat-convex set  & $\unitvec{j} - \unitvec{i}$, \  \ $\pm \unitvec{i}$  & \cite[Sec.5.5]{Mdcasiam}   
\\ 
\LL-convex set & $\unitvec{J} - \unitvec{I}$  \  ($|I|=|J|$)  & \cite{MM22L2poly}
\\ 
\LLnat-convex set & $\unitvec{J} - \unitvec{I}$  \ ($|I|-|J| \in \{ -1,0,1 \}$) & \cite{MM22L2poly} 
\\ 
M-convex set & $\unitvec{I}$, \ \ $- \unitvec{N} (= -\vecone)$  & \cite[Sec.4.4]{Mdcasiam}  
\\
\Mnat-convex set  & $\pm \unitvec{I}$    &  \cite[Sec.4.7]{Mdcasiam}
\\ 
\MM-convex set & $\unitvec{I}$, \ \ $- \unitvec{N} (= -\vecone)$   & by M-convex
\\
\MMnat-convex set  & $\pm \unitvec{I}$     &  by \Mnat-convex
\\ 
Multimodular set & $\pm \unitvec{I}$ \  ($I$: consecutive)   & \cite{MM21inclinter} 
\\ \hline
\multicolumn{3}{l}{%
$\unitvec{I}$: characteristic vector of $I \subseteq N$; \ 
$\unitvec{i}$: $i$-th unit vector $(= \unitvec{ \{ i \} })$}
\end{tabular}
\renewcommand{\arraystretch}{1.0}%
\end{center}
\end{table}

In all cases listed in Table~\ref{TBpolydesc},
we have $a \in \{ -1,0,+1 \}\sp{n}$, that is,
every component of $a$ belongs to $\{ -1,0,+1 \}$.
However, this is not the case with a general integrally convex set.

\begin{example}\rm  \label{EXnonboxTDI1}
Let
$S = \{  (1,1,0,0),  \  (0,1,1,0), \  (1,0,1,0), \  (0,0,0,1) \}$,
which is obviously an integrally convex set
since $S \subseteq \{ 0, 1\}\sp{4}$. 
Because all these points lie on the hyperplane
$x_{1} + x_{2} + x_{3} + 2 x_{4} = 2$,
we need a vector $a= \pm (1,1,1,2)$ 
to describe the convex hull $\overline{S}$.
\finbox
\end{example}

The following example illustrates 
a use of the results in Table~\ref{TBpolydesc}.

\begin{example} \rm \label{EXicex4}
Consider 
$S = \{  x \in \ZZ\sp{4} \mid x_{1} + x_{2} = x_{3} + x_{4} \}$.
This set is described by two inequalities 
of the form of $a\sp{\top} x  \leq 0$
with $a = (1,1,-1,-1), (-1,-1,1,1)$.
This shows that $S$ is not \Lnat-convex,
because
we must have
$a=\unitvec{i} - \unitvec{j}$
or $\pm \unitvec{i}$
 for an \Lnat-convex set
(see Table~\ref{TBpolydesc}).
\Mnat-convexity of $S$ is also denied 
because $a=\pm \unitvec{I}$  
for an \Mnat-convex set.
The set $S$ is, in fact, an \LL-convex set
with
$a=\unitvec{J} - \unitvec{I}$ 
for $(I,J)=(\{ 1,2 \},\{ 3,4 \})$ 
and $(I,J)=(\{ 3,4 \},\{ 1,2 \})$. 
\finbox
\end{example}

Each face of the convex hull $\overline{S}$ 
of an integrally convex set $S$ 
induces an integrally convex set.

\begin{proposition}    \label{PRfaceICset}
Let $S \subseteq \ZZ\sp{n}$ be an integrally convex set.
For any face $F$ of $\overline{S}$,
$F \cap \ZZ\sp{n}$ is integrally convex.
\end{proposition}
\begin{proof}
Consider inequality descriptions of $\overline{S}$ and $F$, say,
$\overline{S} = \{ x \in \RR\sp{n} 
\mid \langle a\sp{(i)}, x \rangle \leq b\sp{(i)} \ (i \in I) \}$ 
and
$F = \{ x \in \RR\sp{n} 
\mid \langle a\sp{(j)}, x \rangle = b\sp{(j)}  \ (j \in J), \ 
 \langle a\sp{(i)}, x \rangle \leq b\sp{(i)}  \ (i \in I \setminus J) \}$ 
with some index sets $I \supseteq J$.
Since $\overline{S}$ is an integer polyhedron, 
we may assume $\veczero \in F$, 
which implies 
$b\sp{(j)} = 0$ for $j \in J$.
Take any $x \in F$.
Since
$x \in \overline{S}$ and $S$ is integrally convex,
there exist $y\sp{(1)}, y\sp{(2)}, \ldots, y\sp{(m)} \in S \cap N(x)$
and coefficients  $\lambda_{1}, \lambda_{2}, \ldots, \lambda_{m} > 0$ such that
$x = \sum_{k=1}\sp{m} \lambda_{k} \, y\sp{(k)}$
and $\sum_{k=1}\sp{m} \lambda_{k} = 1$.
For each  $j \in J$, we have
$0 = \langle a\sp{(j)}, x \rangle 
= \langle a\sp{(j)}, \sum_{k=1}\sp{m} \lambda_{k} \, y\sp{(k)} \rangle 
= \sum_{k=1}\sp{m} \lambda_{k} \langle a\sp{(j)}, y\sp{(k)} \rangle$,
where 
$\langle a\sp{(j)}, y\sp{(k)} \rangle \leq 0$ 
since $y\sp{(k)} \in S$.
Therefore 
$\langle a\sp{(j)}, y\sp{(k)} \rangle = 0$
for all $k$, which implies
$y\sp{(k)} \in F$,   
and hence
$y\sp{(k)} \in (F \cap \ZZ\sp{n}) \cap N(x)$ for all $k$. 
\qedJIAM
\end{proof}

The edges of $\overline{S}$ have a remarkable 
property that the direction of an edge is 
given by a $\{ -1,0,+1 \}$-vector.
This property follows from a basic fact 
that  every edge (line) of $\overline{S}$ contains 
a pair of lattice points in a translated unit hypercube, 
whose difference is a $\{ -1,0,+1 \}$-vector.
In this connection, the following fact is known.

\begin{proposition}[{\cite[Proposition~1]{MT20subgrIC}}]    \label{PRedgeICset}
Let $S \subseteq \ZZ\sp{n}$ be an integrally convex set.
For any face $F$ of $\overline{S}$,
the smallest affine subspace containing $F$ 
is given as 
\ $\{ x + \sum_{k=1}\sp{h} c_{k} d^{(k)} \mid c_{1}, c_{2}, \ldots, c_{h} \in \RR \}$ \ 
for any point $x$ in $F$ and some direction vectors 
$d^{(k)} \in \{ -1,0,+1 \}\sp{n}$ $(k=1,2,\ldots, h)$.
\finboxARX
\end{proposition}

\begin{remark}[\cite{MT20subgrIC}] \rm \label{RMedgedir}
The property mentioned in Proposition~\ref{PRedgeICset} 
does not characterize 
integral convexity of a set.
For example, let
$S = \{  (0,0,0), \allowbreak  (1,0,1),  \allowbreak (1,1,-1), (2,1,0) \}$.
The convex hull $\overline{S}$ is a parallelogram with edge directions
$(1,0,1)$ and $(1,1,-1)$,
and hence is an integer polyhedron
with the property that
the smallest affine subspace containing each face is spanned by $\{ -1,0,+1 \}$-vectors.
However, $S$ is not integrally convex, since
\eqref{icsetdef1} is violated by
$x = [(1,0,1) + (1,1,-1) ]/2 = (1,1/2,0) \in \overline{S}$,  
for which $N(x) = \{ (1,0,0), (1,1,0) \}$ and $S \cap N(x) = \emptyset$.
\finbox
\end{remark}

The following result is concerned with a direction of infinity
in the discrete setting,
to be used in the proof of Proposition~\ref{PRinfdirICfn} 
in Section~\ref{SCtechfn}.

\begin{proposition}    \label{PRinfdirICset}
Let $S \subseteq \ZZ\sp{n}$ be an integrally convex set, 
$y \in S$, and $d \in \ZZ\sp{n}$.
If $y + kd \in S$ 
for all integers $k \geq 1$, 
then for any $z \in S$, 
we have $z + kd \in S$ for all integers $k \geq 1$.
\end{proposition}
\begin{proof}
It suffices to prove that, for any $x \in S$, 
we have $x + d \in S$.
For an integer $k \geq 1$
(to be specified later),
consider
$u = \frac{k-1}{k} x + \frac{1}{k}(y + kd) = x+d + \frac{1}{k}(y-x)$,
which is a convex combination of $x$ and $y + kd$. 
By taking  $k$ large enough, we can assume
$ \| u - (x+d) \|_{2} = \| y - x \|_{2}/k < 1 /  \sqrt{n}  $,
which implies that $x+d \in N(u)$ and 
$u \notin \overline{N(u) \setminus \{ x+d  \}}$.
On the other hand, we have
$u \in \overline{S \cap N(u)}$,
since $u \in \overline{S}$ and $S$ is integrally convex.
Therefore, we must have $x+d \in S$.
\qedJIAM
\end{proof}

\subsubsection*{Box-integer and box-TDI polyhedra}

A polyhedron $P \subseteq \RR\sp{n}$ is called 
\kwd{box-integer}\index{box-integer} 
if $P \cap [ l , u ]_{\RR}$
($=P \cap \{ x \in \RR\sp{n} \mid l   \leq x \leq  u \}$)
is an integer polyhedron for each choice of integer vectors $l , u \in \ZZ\sp{n}$
with $l  \leq u$
(\cite[Section~5.15]{Sch03}).
This concept is closely related (or essentially equivalent) to 
that of integrally convex sets, as follows.

\begin{proposition}[\cite{Mopernet21}] \label{PRintpolyIC} 
If a set $S \subseteq \ZZ^{n}$ is integrally convex, then its convex hull
$\overline{S}$ is a box-integer polyhedron.
Conversely, if $P$ is a box-integer polyhedron, then 
$S = P \cap \ZZ\sp{n}$ is an integrally convex set.
\finboxARX
\end{proposition}

It follows from 
Proposition~\ref{PRintpolyIC}
that
a set $S$ of integer points
is integrally convex if and only if 
it is hole-free
and its convex hull $\overline{S}$ is a box-integer polyhedron.

The concept of (box-)total dual integrality
has long played a major role in combinatorial optimization
\cite{Cook83,Cook86,EG77,EG84tdi,Sch86,Sch03}.
A linear inequality system $Ax \leq  b$
is said to be 
\kwd{totally dual integral}\index{totally dual integral} 
(\kwd{TDI}\index{TDI}) 
if the entries of $A$ and $b$ are rational numbers and
the minimum in the linear programming duality equation
\begin{equation*}  
  \max\{w\sp{\top} x \mid  Ax \leq b \} \ 
  = \ \min \{y\sp{\top} b \mid  y\sp{\top} A=w\sp{\top}, \ y\geq 0 \}
\end{equation*}
has an integral optimal solution $y$ for every integral vector $w$ 
such that the minimum is finite.  
A linear inequality system $Ax \leq  b$
is said to be {\em box-totally dual integral}  ({\em box-TDI}) 
if the system $[ Ax \leq  b, d \leq x\leq c ]$ 
is TDI for each choice of rational (finite-valued) vectors $c$ and $d$.  
It is known 
\cite[Theorem~5.35]{Sch03} that
a system $A x \leq  b$ is box-TDI
if the matrix $A$ is totally unimodular.

A polyhedron is called a 
\kwd{box-TDI polyhedron}\index{box-TDI polyhedron}
if it can be described by a box-TDI system.  
It was pointed out in \cite{Cook86}
that every TDI system describing a box-TDI polyhedron 
is a box-TDI system,
which fact indicates that box-TDI is a property of a polyhedron.
In this connection it is worth mentioning that 
every rational polyhedron can be described by a TDI system,
showing that TDI is a property of a system of inequalities
and not of a polyhedron.

An integral box-TDI polyhedron is box-integer \cite[(5.82), p.~83]{Sch03}.
Although the converse is not true (see Example~\ref{EXnonboxTDI2} below),
it is possible to characterize a box-TDI polyhedron 
in terms of box-integrality of its dilations.
For a positive integer $\alpha$, the 
\kwd{$\alpha$-dilation}\index{dilation} 
of a polyhedron 
$P=\{x \mid A x \leq b \}$
means the polyhedron  
$\alpha P :=\{x \mid A x \leq \alpha b \} = \{x  \mid \frac{1}{\alpha}x \in P \}$.

\begin{proposition} 
[{\cite[Theorem~2 \& Prop.~2]{CGR21}}] 
\label{PRdilat} 
An integer polyhedron $P$ is box-TDI if and only if 
the $\alpha$-dilation $\alpha P$ is box-integer for any positive integer $\alpha$.
\finboxARX
\end{proposition}

\begin{example}\rm  \label{EXnonboxTDI2}
Here is an example of 
a $\{ 0,1 \}$-polyhedron that is not box-TDI.
Let
$P$  $(\subseteq \RR\sp{4})$ be the convex hull of 
$S = \{  (1,1,0,0),  \  (0,1,1,0), \  (1,0,1,0),  \allowbreak  (0,0,0,1) \}$
(considered in Example~\ref{EXnonboxTDI1}).
Since $P$ is a $\{ 0,1 \}$-polyhedron, it is obviously box-integer. 
However, 
the 2-dilation $2P$ is not box-integer.
To see this we note that 
\[
(1,1,1,1/2) =  \frac{1}{4} \big( (2,2,0,0) + (0,2,2,0) + (2,0,2,0) + (0,0,0,2) \big) \in 2P ,
\]
whereas
$(1,1,1,1/2) \in  [ l , u ]_{\RR}$
for $l =(1,1,1,0)$ and $u =(1,1,1,1)$,
and $l \notin 2P$ and $u \notin 2P$.
We can easily show that
$(2P) \cap [ l , u ]_{\RR}$
consists of 
$(1,1,1,1/2)$ only.
Thus $2P$ is not box-integer,
which implies, by Proposition~\ref{PRdilat}, that $P$ is not box-TDI.
\finbox
\end{example}

Following \cite{FM21boxTDI}
we call the set of integral elements of an integral box-TDI polyhedron 
a {\em discrete box-TDI set}, or just a {\em box-TDI set}.
A box-TDI set is an integrally convex set,
but the converse is not true
(Example~\ref{EXnonboxTDI2}).
That is, box-TDI sets form a proper subclass of integrally convex sets.
On the other hand,
the major classes of discrete convex sets
considered in discrete convex analysis 
are known to be box-TDI as follows.

\begin{proposition}[\cite{MM22L2poly}] \label{PRL2boxTDI} 
An \LLnat-convex set is a box-TDI set.
\finboxARX
\end{proposition}

\begin{proposition} \label{PRM2boxTDI} 
An \MMnat-convex set is a box-TDI set.
\finboxARX
\end{proposition}

Proposition~\ref{PRL2boxTDI} for \LLnat-convex sets
is established recently in \cite{MM22L2poly}
and Proposition~\ref{PRM2boxTDI} 
for \MMnat-convex sets is a reformulation
of the fundamental fact about polymatroid intersection \cite{Sch03} 
in the language of discrete convex analysis.
Theses propositions imply, in particular, that 
\LL-, \Lnat-,  L-, \MM-, \Mnat-,  M-convex sets are all box-TDI.

The following are examples of 
a box-TDI set $S$ that is neither \LLnat-convex nor \MMnat-convex. 
The former consists of $\{ 0,1 \}$-vectors and the latter arises from a cone.

\begin{example}\rm  \label{EXnonL2M2boxTDI}
Consider $S = \{  (0,0,0),  \ (1,1,0), \  (1,0,1), \  (0,1,1) \}$.
This set is described by four inequalities 
\[
x_{1} + x_{2} + x_{3} \leq 2,
\ 
x_{1} - x_{2} - x_{3} \leq 0,
\  
- x_{1} + x_{2} - x_{3} \leq 0,
\  
- x_{1} - x_{2} + x_{3} \leq 0.
\]
The first inequality,
of the form of $a\sp{\top} x  \leq b$
with $a =(1,1,1)$,
denies \LLnat-convexity of $S$,
because we must have
$a=\unitvec{J} - \unitvec{I}$  
with $|I|-|J| \in \{ -1,0,1 \}$ for an \LLnat-convex set
(see Table~\ref{TBpolydesc}).
In the second inequality we have $a =(1,-1,-1)$,
which denies \MMnat-convexity of $S$,
because $a=\pm \unitvec{I}$ for an \MMnat-convex set.
The set $S$ is box-TDI, that is, its convex hull $\overline{S}$ 
is a box-TDI polyhedron,
which we can verify on the basis of
Proposition~\ref{PRdilat}.
\finbox
\end{example}

\begin{example}\rm  \label{EXnonL2M2boxTDIcone}
The set $S = \{ x  \in \ZZ\sp{2} \mid 
x_{1} + x_{2}  \leq 0,
x_{1} - x_{2}  \leq 0
 \}$
is neither \LLnat-convex nor \MMnat-convex,
whereas
it is box-TDI since
the convex hull $\overline{S}$ is a box-TDI polyhedron by Proposition~\ref{PRdilat}.
\finbox
\end{example}

We can summarize the above argument as
{\small
\[
 \{ \mbox{\LLnat-convex sets} \} 
  \cup \{ \mbox{\MMnat-convex sets} \} 
\subsetneqq \{ \mbox{box-TDI sets} \} 
\subsetneqq \{ \mbox{integrally convex sets} \} ,
\]
}
where 
Examples \ref{EXnonboxTDI2}, \ref{EXnonL2M2boxTDI}, and \ref{EXnonL2M2boxTDIcone}
demonstrate the strict inclusions ($\subsetneqq $).
See Fig.~\ref{FGdcsetclassTDI}.


\subsection{Basic operations}
\label{SCoperset}

In this section we show how integral convexity of a set behaves under basic operations.
Let $S$ be a subset of $\ZZ\sp{n}$, i.e.,  $S \subseteq \ZZ\sp{n}$.

\paragraph{Origin shift:}

For an integer vector $b \in \ZZ\sp{n}$, the
{\em origin shift}
of $S$ by $b$ means a set $T$
defined by
$T  = \{  x - b \mid x \in S \}$.
The origin shift of 
an integrally convex set
is an integrally convex set.

\paragraph{Inversion of coordinates:}

The {\em independent coordinate inversion} of $S$ 
means a set $T$ defined by
\[
T  =
\{ (\tau_{1} x_{1}, \tau_{2} x_{2}, \ldots,  \tau_{n}x_{n}) 
          \mid (x_{1},x_{2}, \ldots, x_{n}) \in S \}
\]
with an arbitrary choice of $\tau_{i} \in \{ +1, -1 \}$ $(i=1,2,\ldots,n)$.
The independent coordinate inversion of an integrally convex set is an integrally convex set.
This is a nice property of integral convexity, not shared by
\Lnat-, \LLnat-, \Mnat, or \MMnat-convexity.

\paragraph{Permutation of coordinates:}

For a permutation $\sigma$ of $(1,2,\ldots,n)$,
the {\em permutation} of $S$ by $\sigma$
means a set $T$ 
defined by
\begin{equation*}
T =   \{ (y_{1},y_{2}, \ldots, y_{n})   
 \mid (y_{\sigma(1)}, y_{\sigma(2)}, \ldots, y_{\sigma(n)}) \in S \}.
\end{equation*}
The permutation of an integrally convex set is an integrally convex set.

\begin{remark} \rm  \label{RMunimodIC}
Integral convexity is not preserved under a transformation
by a (totally) unimodular matrix.
For example, 
$S=\{ (0,0), (1,0), (1,1) \}$
is integrally convex
and 
$A =$
{\small
$\left[ 
\begin{array}{cc}
    1     &  1     \\
    0     &  1 \\
\end{array}  \right]
$}
is totally unimodular.
However,  
$\{ Ax \mid x \in S \} = \{ (0,0), (1,0), (2,1) \}$
is not integrally convex.
\finbox
\end{remark}

\paragraph{Scaling:}
For a positive integer $\alpha$,
the {\em scaling} of $S$ by $\alpha$
means a set $T$ 
defined by
\begin{equation} \label{scalesetdef}
T = \{ (y_{1},y_{2}, \ldots, y_{n}) \in \ZZ\sp{n}
 \mid (\alpha y_{1}, \alpha y_{2}, \ldots, \alpha y_{n}) \in S \}.
\end{equation}
Note that the same scaling factor $\alpha$ is used for all coordinates.
If $\alpha = 2$, for example, this operation amounts to considering the
set of even points contained in $S$.
The scaling of an integrally convex set
is not necessarily integrally convex
(Example~\ref{EXscICsetNG422} below).
However, when $n = 2$, 
integral convexity admits the scaling operation.
That is, if $S \subseteq \ZZ^{2}$
is integrally convex, then 
$T = \{ y \in \ZZ^{2} \mid \alpha y \in S \}$
is integrally convex
(\cite[Proposition~3.1]{MMTT19proxIC}).

\begin{figure}\begin{center}
\includegraphics[height=45mm]{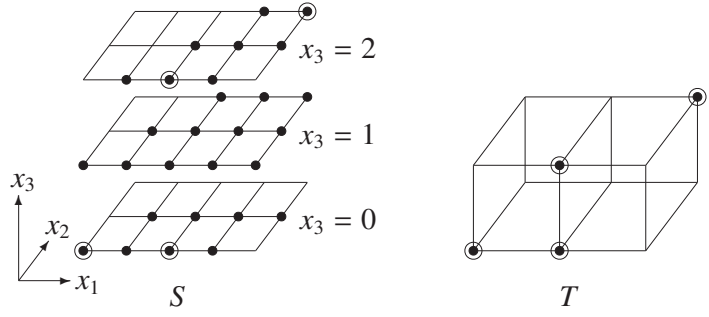}
\caption{An integrally convex set $S$ and its scaled set $T$ 
(Example~\ref{EXscICsetNG422}) \cite{MMTT19proxIC}.}
\label{FGicsetsc}
\end{center}\end{figure}

\begin{example}[{\cite[Example 3.1]{MMTT19proxIC}}] \rm \label{EXscICsetNG422}
This example shows that integral convexity is not preserved under scaling.
Let $S$ be a subset of $\ZZ^{3}$ defined  by
\begin{align*}
S  = & 
\{ (x_{1},x_{2},0) \mid 
0 \leq x_{2} \leq 1,  \  0 \leq x_{1} - x_{2} \leq 3 
\}  
\\ & \cup
\{ (x_{1},x_{2},1) \mid 
0 \leq x_{2} \leq 2, \  x_{2} \leq x_{1} \leq 4
\}  
\\ & \cup
\{ (x_{1},x_{2},2) \mid 
0 \leq x_{2} \leq 2, \ 1 \leq x_{1} - x_{2} \leq 3, \  x_{1} \leq 4
\}  ,
\end{align*}
which is an integrally convex set (Fig.~\ref{FGicsetsc}, left).
With the scaling factor $\alpha=2$,
however, 
the scaled set
$ T = \{ y \in \ZZ\sp{3} \mid 2 y \in S \} 
= \{ (0,0,0), (1,0,0), (1,0,1),  \allowbreak  (2,1,1) \}$
is not integrally convex  
(Fig.~\ref{FGicsetsc}, right).
\finbox
\end{example}

\paragraph{Dilation:} 
The dilation operation for a polyhedron
(described in Section~\ref{SCpolyh})
is another kind of scaling operation.
An adaptation of this operation 
to a hole-free discrete set $S \subseteq \ZZ\sp{n}$,
we may call the set 
$T' = (\alpha \overline{S}) \cap \ZZ\sp{n}$
the 
\kwd{$\alpha$-dilation}
of $S$, where $\alpha$ is a positive integer.
Note that 
the scaling in \eqref{scalesetdef} can be expressed as 
$T = (\frac{1}{\alpha} \overline{S}) \cap \ZZ\sp{n}$
when $S$ is hole-free.

The dilation operation does not always preserve integral convexity.
Indeed,
Example~\ref{EXnonboxTDI2} shows that
the $2$-dilation of an integrally convex set is not 
necessarily integrally convex.

\begin{remark} \rm  \label{RMscaldilset}
Failure of dilation operation 
is rather exceptional for discrete convex sets.
Indeed, all kinds of discrete convexity
(box, L-, \Lnat-, \LL-, \LLnat-, M-, \Mnat-, \MM-, \MMnat-convexity, and multimodularity)
listed in Table~\ref{TBpolydesc}
are preserved under the dilation operation.
In contrast, the scaling operation in \eqref{scalesetdef} 
preserves L-convexity and its relatives 
(box, L-, \Lnat-, \LL-, \LLnat-convexity, and multimodularity),
and not M-convexity and its relatives 
(M-, \Mnat-, \MM-, \MMnat-convexity).
\finbox
\end{remark}

\paragraph{Restriction:}

For a set $S \subseteq \ZZ\sp{N}$
and a subset $U$ of the index set $N = \{ 1,2,\ldots, n \}$,
the {\em restriction} 
of $S$ to $U$ is a subset $T$ of $\ZZ\sp{U}$ defined by
\begin{equation*}  
 T =  \{ y \in \ZZ\sp{U} \mid  (y,\veczero_{N \setminus U}) \in S \}, 
\end{equation*}
where $\veczero_{N \setminus U}$ denotes the zero vector 
in $\ZZ\sp{N \setminus U}$.
The notation 
$(y,\veczero_{N \setminus U})$ means the vector in $\ZZ\sp{N}$
whose $i$th component is equal to $y_{i}$ for $i \in U$
and to $0$ for $i \in N \setminus U$.
The restriction 
of an integrally convex set is integrally convex 
(if the resulting set is nonempty).

\paragraph{Projection:}

For a set $S \subseteq \ZZ\sp{N}$
and a subset $U$ of the index set $N = \{ 1,2,\ldots, n \}$,
the
{\em projection}
of $S$ to $U$ is a subset $T$ of $\ZZ\sp{U}$ defined by
\begin{equation}   \label{projsetdef}
 T =  \{ y \in \ZZ\sp{U} \mid  (y,z) \in S \mbox{ for some $z \in \ZZ\sp{N \setminus U}$} \}, 
\end{equation}
where the notation $(y,z)$ means the vector in $\ZZ \sp{N}$
whose $i$th component is equal to $y_{i}$ for $i \in U$
and to $z_{i}$ for $i \in N \setminus U$.
The projection of an integrally convex set
is integrally convex
(\cite[Theorem~3.1]{MM19projcnvl}).

\paragraph{Splitting:}

Suppose that we are given a family 
$\{ U_{1},  U_{2}, \dots , U_{n} \}$
of  disjoint nonempty sets
indexed by $N = \{ 1, 2, \dots , n\}$. 
Let $m_{i}= |U_{i}|$ for $i=1,2,\ldots, n$ and define 
$m= \sum_{i=1}\sp{n} m_{i}$, where $m \geq n$.
For each $i \in N$
we define an $m_{i}$-dimensional vector
$y_{[i]} = ( y_{j} \mid  j \in U_{i} )$
and express $y \in \ZZ\sp{m}$ as
$y = (y_{[1]}, y_{[2]}, \dots , y_{[n]})$. 
For a set $S \subseteq \ZZ\sp{n}$, 
the subset of $\ZZ\sp{m}$ defined by 
\begin{equation*} 
 T = \{ (y_{[1]}, y_{[2]}, \dots , y_{[n]})  \in  \ZZ\sp{m} \mid
  y_{[i]} \in \ZZ\sp{m_{i}}, \   y_{[i]}(U_{i}) = x_{i}  \ \ (i \in N) , \ x  \in S   \}
\end{equation*}
is called the {\em splitting} of $S$ by
$\{ U_{1},  U_{2}, \dots , U_{n} \}$,
where $y_{[i]}(U_{i}) = \sum\{  y_{j} \mid j \in U_{i} \}$.
For example,
$T = \{ (y_{1}, y_{2}, y_{3}) \in  \ZZ\sp{3} \mid (y_{1}, y_{2}+ y_{3}) \in S  \}$
is a splitting of $S \subseteq \ZZ\sp{2}$
for $U_{1} = \{ 1 \}$ and $U_{2} = \{ 2,3 \}$,
where $n=2$ and $m=3$.
The splitting of an integrally convex set is integrally convex
(\cite[Proposition~3.4]{Mopernet21}).

\paragraph{Aggregation:}

Let 
$\mathcal{P} =  \{ N_{1},  N_{2}, \dots , N_{m} \}$
 be a  partition of $N = \{ 1,2, \ldots, n \}$ into disjoint nonempty subsets:
$N = N_{1} \cup N_{2} \cup \dots \cup N_{m}$
and
$N_{i} \cap N_{j} = \emptyset$ for $i \not= j$.
For a set $S \subseteq \ZZ\sp{N}$ 
the subset of $\ZZ\sp{m}$,
where $m \leq n$,
 defined by 
\begin{equation*}  
 T = \{ (y_{1}, y_{2}, \dots , y_{m}) \in  \ZZ\sp{m} \mid
     y_{j} = x(N_{j}) \ (j=1,2,\ldots,m), \ x \in S \}
\end{equation*}
is called the {\em aggregation} of $S$ by $\mathcal{P}$.
For example,
$T = \{ (y_{1}, y_{2}) \in  \ZZ\sp{2} \mid 
y_{1} = x_{1},
y_{2} = x_{2} + x_{3}
\mbox{ for some } (x_{1}, x_{2}, x_{3}) \in S  \}$
is an aggregation of $S \subseteq \ZZ\sp{3}$
for $N_{1} = \{ 1 \}$ and $N_{2} = \{ 2,3 \}$,
where $n=3$ and $m=2$.
The aggregation of an integrally convex set
is not necessarily integrally convex.

\begin{example}[{\cite[Example~3.4]{Mopernet21}}] \rm \label{EXicvsetaggr}
Set
$S = \{ (0,0,1,0), (0,0,0,1), (1,1,1,0),   \allowbreak  (1,1,0,1) \}$
is an integrally convex set.
For the partition of $N = \{ 1,2,3,4 \}$
into $N_{1} = \{ 1,3 \}$ and $N_{2} = \{ 2,4 \}$,
the aggregation of $S$
by $\{ N_{1}, N_{2} \}$ 
is given by
$T = \{ (1,0), (0,1), (2,1), (1,2) \}$,
which is not integrally convex.
\finbox
\end{example}

\paragraph{Intersection:}

The intersection $S_{1} \cap S_{2}$ 
of integrally convex sets
$S_{1}$, $S_{2} \subseteq \ZZ\sp{n}$
is not necessarily integrally convex
(Example~\ref{EXicsetinter} below).
However, it is obviously true
(almost from definition) 
that the intersection of an integrally convex set
with a box of integers is integrally convex.

\begin{example}[{\cite[Example 4.4]{MS01rel}}] \rm \label{EXicsetinter}
The intersection of two integrally convex sets
is not necessarily integrally convex. 
Let
$S_{1}  =  \{(0, 0, 0), \allowbreak (0, 1, 1), (1, 1, 0),  \allowbreak (1, 2, 1)\}$
and 
$S_{2}  =  \{(0, 0, 0), \allowbreak (0, 1, 0), (1, 1, 1), \allowbreak (1, 2, 1)\}$,
for which
$S_{1} \cap S_{2}= \{(0, 0, 0), (1, 2, 1)\}$.
The sets $S_{1}$ and $S_{2}$ are integrally convex,
whereas  $S_{1} \cap S_{2}$ is not.
\finbox
\end{example}

\paragraph{Minkowski sum:}

The {\em Minkowski sum}
of two sets $S_{1}$, $S_{2} \subseteq \ZZ\sp{n}$ 
means the subset of $\ZZ\sp{n}$ 
defined by
\begin{equation} \label{minkowsumZdef}
S_{1}+S_{2} = 
\{ x + y \mid x \in S_{1}, \  y \in S_{2} \} .
\end{equation}
The Minkowski sum of integrally convex sets
is not necessarily integrally convex
(Example~\ref{EXicdim2sumhole} below).
However, 
the Minkowski sum of an integrally convex set
with a box of integers is integrally convex
(\cite[Theorem~4.1]{MM19projcnvl}).

\begin{example}[{\cite[Example 3.15]{Mdcasiam}}] \rm \label{EXicdim2sumhole}
The Minkowski sum of
$S_{1} = \{ (0,0), (1,1) \}$
and
$S_{2} = \{ (1,0), (0,1) \}$
is equal to 
$S_{1}+S_{2} = \{ (1,0), (0,1), (2,1), (1,2) \}$,
which has a ``hole'' at $(1,1)$, i.e.,
$(1,1) \in \overline{S_{1}+S_{2}}$ and
$(1,1) \not\in S_{1}+S_{2}$.
\finbox
\end{example}

\begin{figure}\begin{center}
\includegraphics[height=33mm]{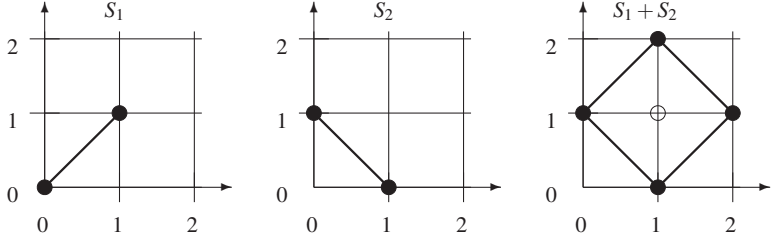}
\caption{Failure of convexity in Minkowski sum \eqref{convminkowsumG}.}
\label{FGminkowhole}
\end{center}\end{figure}

\begin{remark} \rm  \label{RMconvminkowsumG}
The Minkowski sum is often a source of difficulty in a discrete setting,
because 
\begin{equation} \label{convminkowsumG}
   S_{1}+S_{2}  = ( \overline{S_{1}+ S_{2}}) \cap \ZZ\sp{n} 
\end{equation}
is not always true
(Example~\ref{EXicdim2sumhole}).
In other words, the equality \eqref{convminkowsumG}, 
if true, captures a certain essence
of discrete convexity.
The property
\eqref{convminkowsumG} 
is called 
``convexity in Minkowski sum''
in \cite[Section 3.3]{Mdcasiam}.
We sometimes call \eqref{minkowsumZdef} the
{\em discrete (or integral) Minkowski sum}
of $S_{1}$ and $S_{2}$ to emphasize discreteness.
\finbox
\end{remark}

\begin{remark} \rm  \label{RMmlsetMinkow}
The Minkowski sum plays a central role in discrete convex analysis.
The Minkowski sum of two (or more) M$\sp{\natural}$-convex sets
is M$\sp{\natural}$-convex.
The Minkowski sum of two L$\sp{\natural}$-convex sets
is not necessarily L$\sp{\natural}$-convex,
but it is integrally convex.
The Minkowski sum of three \Lnat-convex sets
is no longer integrally convex.
For example (\cite[Example 4.12]{MS01rel}),
$S_{1} = \{(0, 0, 0), (1, 1, 0)\}$, $S_{2} = \{(0, 0, 0), (0, 1, 1)\}$, and
$S_{3} = \{(0, 0, 0), (1, 0, 1)\}$
are \Lnat-convex sets, and their Minkowski sum 
$S = S_{1} + S_{2} + S_{3}$
is given as 
\[
S = \{(0,0,0),(0,1,1),(1,1,0),(1,0,1),(2,1,1),(1,1,2),(1,2,1),(2,2,2)\},
\]
which is not integrally convex, since
$(1,1,1) \in \overline{S}$ and
$(1,1,1) \not\in S$.
\finbox
\end{remark}

The following theorem is a discrete analogue of a well-known
decomposition of a polyhedron into a bounded part and 
a conic part (recession cone or characteristic cone)
\cite[Theorem 8.5]{Sch86}.
An integrally convex set is called {\em conic} if its convex hull is a cone.

\begin{theorem}[\cite{MT22ICdecmkw}]  \label{THicsetDecZ}
Every integrally convex set $S$
can be represented as a (discrete) Minkowski sum of
a bounded integrally convex set $Q$
and a conic integrally convex set $C$,
that is, $S = Q + C$. 
\finboxARX
\end{theorem}



\section{Integrally convex functions}
\label{SCfnconcept}

\subsection{Convex extension}
\label{SCfnconvext}

For a function 
$g: \RR\sp{n} \to \RR \cup \{ -\infty, +\infty \}$ 
in general, 
$\dom g := \{ x \in \RR\sp{n} \mid -\infty < g(x) < +\infty \}$
is called the {\em effective domain} of $g$.
In this section we always assume that
$f: \ZZ^{n} \to \RR \cup \{ +\infty  \}$
and $\dom f \ne \emptyset$, that is,
$f$ is a function defined on $\ZZ^{n}$
taking values in 
$\RR \cup \{ +\infty  \}$ and 
$\dom f = \{ x \in \ZZ\sp{n} \mid  f(x) < +\infty \}$
is nonempty.

We say that $f$ is
\kwd{convex-extensible}
if there exists a convex function 
$g: \RR^{n} \to \RR \cup \{ +\infty  \}$
satisfying 
$g(x) = f(x)$ for all $x \in \ZZ^{n}$.
When $n=1$, 
$f: \ZZ \to \RR \cup \{ +\infty  \}$
is convex-extensible
if and only if 
$\dom f$ is an interval of integers and
$f(k-1) + f(k+1) \geq 2 f(k)$
for all $k \in \ZZ$.
In this case, a convex extension of $f$
is given by
the piecewise-linear function
whose graph consists of line segments connecting
$(k,f(k))$ and $(k+1,f(k+1))$ 
for all $k \in \ZZ$.

We say that a function
$g: \RR^{n} \to \RR \cup \{ +\infty  \}$
\kwd{minorizes} $f$ if
$g(x) \leq f(x)$ for all $x \in \ZZ^{n}$.
In this section we always assume that 
$f$ is minorized by some affine function
$g(x) = \langle p, x \rangle + \alpha$,
where $p \in \RR\sp{n}$, $\alpha \in \RR$, and
$\langle p, x \rangle := \sum_{i=1}\sp{n} p_{i} x_{i}$
denotes the inner product
(or duality pairing, to be more precise)
 of $p=(p_{1}, p_{2}, \ldots, p_{n})$ and 
$x=(x_{1}, x_{2}, \allowbreak \ldots, \allowbreak  x_{n})$.
Note that every convex-extensible function is minorized 
by an affine function.

The {\em convexification}
of $f$,
to be denoted by $\check{f}$,
is defined as
\begin{equation} \label{fnconvenvInf}
 \check{f}(x)  
:=  \inf_{\lambda}\{ \sum_{y} \lambda_{y} f(y) \mid
      \sum_{y} \lambda_{y} y = x,  
  (\lambda_{y})  \in \Lambda \}
\quad (x \in \RR^{n}) ,
\end{equation}	
where $\Lambda$ denotes the set of coefficients 
for convex combinations indexed by $y\in \ZZ\sp{n}$:
\begin{equation*} 
  \Lambda = \{ (\lambda_{y} \mid y \in \ZZ\sp{n} ) \mid 
      \sum_{y} \lambda_{y} = 1, 
      \lambda_{y} \geq 0 \ \mbox{for all $y$}, 
      \lambda_{y} > 0 \ \mbox{for finitely many $y$}  
 \} .
\end{equation*} 
It is known \cite[Section B.2.5]{HL01}
that $\check{f}$ is a convex function
and that $\check{f}$ coincides with the pointwise supremum of all convex functions 
minorizing $f$,
that is,
\begin{equation*} 
 \check{f}(x) = 
  \sup \{ g(x) \mid 
 \mbox{$g$ is convex, $g(y) \leq f(y)$ for all $y \in \ZZ\sp{n}$}
 \}
\quad (x \in \RR^{n}) .
\end{equation*}	
Therefore, $f$ is convex-extensible 
if and only if
$\check{f}(x) = f(x)$ for all $x \in \ZZ^{n}$.

The {\em convex envelope}
 of $f$,
to be denoted by $\overline{f}$,
is defined as
the pointwise supremum of all affine functions 
minorizing $f$, that is,
\begin{equation} \label{fnconvenvClSupAff}
 \overline{f}(x)  :=
 \sup_{p,\alpha}\{ 
\langle p, x \rangle + \alpha \mid 
\langle p, y \rangle + \alpha \leq f(y) \ 
(\forall y \in \ZZ\sp{n}) \}
\quad (x \in \RR\sp{n}).
\end{equation}
This function
$\overline{f}$ is a closed convex function
and coincides with the pointwise supremum of all closed convex functions
minorizing $f$,
that is,
\begin{equation*} 
 \overline{f}(x) = 
  \sup \{ g(x) \mid 
 \mbox{$g$ is closed convex, $g(y) \leq f(y)$ for all $y \in \ZZ\sp{n}$} \}
\quad (x \in \RR^{n}) .
\end{equation*}	
In this paper we often refer to the condition
\begin{equation} \label{fnconvextcl}
  f(x) =  \overline{f}(x) 
\quad (x \in \ZZ^{n}) 
\qquad (\mbox{i.e., \  $f = \overline{f}\,|_{\ZZ\sp{n}}$})
\end{equation}	
as the convex-extensibility of $f$,
although this condition is slightly stronger than the condition 
$f = \check{f}\,|_{\ZZ\sp{n}}$
mentioned above.
Accordingly,
we often refer to $\overline{f}$ as the {\em convex extension}
of $f$ if \eqref{fnconvextcl} is the case.

\begin{example} \rm \label{EXicquadext}
The quadratic function
$f(x)=x\sp{2}$ defined for $x \in \ZZ$
is convex-extensible,
where 
$g(x)=x\sp{2}$ $(x \in \RR)$
is an obvious convex extension of $f$.
The convex envelope $\overline{f}$ 
in \eqref{fnconvenvClSupAff}
is a piecewise-linear function given by 
\begin{equation*} 
 \overline{f}(x) = (2k+1) |x| - k (k+1) 
\quad \mbox{with \  $k = \lfloor |x| \rfloor$}
\qquad (x \in \RR).
\end{equation*}
It is noted that 
$\overline{f}(x) = x\sp{2}$ for integers $x$ and
$\overline{f}(x) > x\sp{2}$ for non-integral $x$;
for example, $\overline{f}(1/2) = 1/2 > 1/4$.
The convexification $\check{f}$ in \eqref{fnconvenvInf}
coincides with $\overline{f}$.
\finbox
\end{example}

\begin{remark} \rm  \label{RMconvhullfn}
In a standard textbook \cite[Section B.2.5]{HL01},
\eqref{fnconvenvInf}
is called the {\em convex hull} of $f$
and denoted by $\mathrm{co} f$,
whereas 
\eqref{fnconvenvClSupAff}
is called the {\em closed convex hull} of $f$ and
denoted by $\overline{\mathrm{co}}\, f$ or $\mathrm{cl}(\mathrm{co} f)$.
Using our notation we have
$\mathrm{co} f = \check{f}$ and $\overline{\mathrm{co}}\, f = \overline{f}$.
\finbox 
\end{remark}

\begin{remark} \rm \label{RMcofncoclfn}
We have
$\overline{f}(x) \leq \check{f}(x)$
for all $x \in \RR\sp{n}$,
and the equality may fail in general
 (Example \ref{EXdim2Fext} below).
However,
when $\dom f$ is bounded,
we have 
$\overline{f}(x) = \check{f}(x)$
for all $x \in \RR\sp{n}$.
The proofs are as follows.
In \eqref{fnconvenvClSupAff} we have
$\langle p, y \rangle + \alpha \leq f(y)$
for each $y$.
By using $\lambda \in \Lambda$
satisfying $\sum_{y} \lambda_{y} y = x$,  we obtain
\[
\langle p, x \rangle + \alpha 
= \sum_{y} \lambda_{y} ( \langle p, y \rangle + \alpha )
\leq \sum_{y} \lambda_{y} f(y) ,
\]
from which 
$\overline{f}(x) =  \sup_{(p, \alpha)} \{ \langle p, x \rangle + \alpha \}  
\leq 
  \inf_{\lambda}\{ \sum_{y} \lambda_{y} f(y) \} = \check{f}(x)$.
When $\dom f$ is bounded, $\dom f$ is a finite set.
For each $x \in \RR\sp{n}$, consider a pair of 
(mutually dual) linear programs:
\begin{center}
\begin{tabular}{llll}
(P) &
Maximize &  $\langle p, x \rangle + \alpha$ \\
& subject to & 
 $\langle p, y \rangle + \alpha \leq f(y) \quad (y \in \dom f)$,
\\[3mm]
(D) &
Minimize &
$\displaystyle \sum_{y \in \dom f} \lambda_{y} f(y)$ \\
& subject to &
$\displaystyle \sum_{y \in \dom f} \lambda_{y} y = x$,
 $\displaystyle \sum_{y \in \dom f} \lambda_{y} = 1$, 
 $\lambda_{y} \geq 0 \quad  (y \in \dom f)$,
\end{tabular}
\end{center}
where $(p, \alpha) \in \RR\sp{n} \times \RR$ and $(\lambda_{y} \mid y \in \dom f)$ 
are the variables of (P) and (D), respectively.
The optimal values of (P) and (D) are equal to 
$\overline{f}(x)$ and $\check{f}(x)$, respectively.
Problem (P) is feasible
(e.g., take $p=0$ and a sufficiently small $\alpha$).
By LP duality, 
(P) and (D) have the same (finite or infinite) optimal values,
that is,
$\overline{f}(x) = \check{f}(x)$.
Note that (D) is feasible if and only if 
$x \in \overline{\dom f}$, 
in which case 
the optimal values are finite.
\finbox
\end{remark}

\begin{example} \rm \label{EXdim2Fext}
Let 
$f: \ZZ\sp{2} \to \RR \cup \{ +\infty \}$
be the indicator function $\delta_{S}$
of the set $S$ 
considered in Remark~\ref{RMconvhull}.
For $x= (x_{1},1)$ with $x_{1} \ne 0$,
we have $x \in \mathrm{cl}(\overline{S}) \setminus \overline{S}$.
Hence 
$0= \overline{f}(x)< \check{f}(x) = +\infty$.
\finbox
\end{example}

\begin{remark} \rm \label{RMholeindfn}
For a set $S \subseteq \ZZ\sp{n}$,
the convexification
of the indicator function $\delta_{S}$ 
coincides with the indicator function 
of its convex hull $\overline{S}$,
that is,
$\check{\delta}_{S} = \delta_{\overline{S}}$.
A set $S \subseteq \ZZ\sp{n}$ 
is hole-free if and only if 
the indicator function $\delta_{S}$ is convex-extensible.
\finbox
\end{remark}

\subsection{Definition of integrally convex functions}
\label{SCicfndef}

Recall the notation $N(x)$
for the integral neighborhood of  $x \in \RR^{n}$
(cf., \eqref{Nxdef}, Fig.~\ref{FGneighbor}).
For a function
$f: \mathbb{Z}^{n} \to \mathbb{R} \cup \{ +\infty  \}$,
the \kwd{local convex extension} 
$\tilde{f}: \RR^{n} \to \RR \cup \{ +\infty \}$
of $f$ is defined 
as the union of all 
convex extensions (convexifications)
 of $f$ on $N(x)$.
That is,
\begin{equation} \label{fnconvclosureloc2}
 \tilde f(x) = 
  \min\{ \sum_{y \in N(x)} \lambda_{y} f(y) \mid
      \sum_{y \in N(x)} \lambda_{y} y = x,  
  (\lambda_{y})  \in \Lambda(x) \}
\quad (x \in \RR^{n}) ,
\end{equation} 
where $\Lambda(x)$ denotes the set of coefficients for convex combinations 
indexed by $N(x)$:
\begin{equation} \label{cvcoefLamda}
  \Lambda(x) = \{ (\lambda_{y} \mid y \in N(x) ) \mid 
      \sum_{y \in N(x)} \lambda_{y} = 1, 
      \lambda_{y} \geq 0 \ (y \in N(x))  \} .
\end{equation} 
It follows from this definition that,
for each $x \in \RR\sp{n}$,
the function $\tilde f$ 
restricted to $\overline{N(x)}$
is a convex function.
In general, we have
$\tilde f (x) \geq \check{f}(x)  \geq \overline{f}(x)$ for all $x \in \RR^{n}$,
where $\check{f}$ and $\overline{f}$ are defined by 
\eqref{fnconvenvInf}
and
\eqref{fnconvenvClSupAff},
respectively.

We say that a function $f$ is 
\kwd{integrally convex}\index{integrally convex function}
if its local convex extension $\tilde f$ is (globally) 
convex on the entire space $\RR^{n}$.
In this case, 
$\tilde f$ is a convex function satisfying 
$\tilde f (x) = f(x)$ for all $x \in \ZZ^{n}$,
which means that $f$ is convex-extensible.
Moreover, $\tilde f$ coincides with
$\check{f}$
and $\overline{f}$,
that is,
\begin{equation}
\label{ICfnextR}
 \tilde f (x) = \check{f}(x)  = \overline{f}(x)  \qquad (x \in \RR^{n}).
\end{equation}
In particular, we have
$\dom \tilde f = \overline{\dom f}$.
Since 
$\tilde f (x) = f(x)$ for $x \in \ZZ^{n}$,
\eqref{ICfnextR} implies
\begin{equation}
\tilde f (x) = \check{f}(x)  = \overline{f}(x) = f(x)  \qquad (x \in \ZZ^{n}).
\label{ICfnextZ}
\end{equation}

\begin{proposition}  \label{PRicvsetfun}
\quad

\noindent
{\rm (1)}
The effective domain of an integrally convex function
is integrally convex.

\noindent
{\rm (2)}
A set $S \subseteq \ZZ^{n}$ is
integrally convex if and only if its indicator function $\delta_{S}$ is 
integrally convex.
\finboxARX
\end{proposition}

The following is an example of a convex-extensible function
that is not integrally convex. 

\begin{example} \rm  \label{EXnonintconvfn}
Let $f: \ZZ\sp{2} \to \RR$ be defined by
$f(x_{1}, x_{2})= | 2 x_{1} - x_{2} |$
for all $(x_{1}, x_{2}) \in \ZZ\sp{2}$.
Obviously, this function is convex-extensible and the convex envelope
 is given by
$\overline{f}(x_{1}, x_{2})= | 2 x_{1} - x_{2} |$ 
for all $(x_{1}, x_{2}) \in \RR\sp{2}$.
For $y=(1/2,1)$ we have
$N(y)=\{ (0,1), (1,1) \}$ 
and the local convex extension $\tilde f$ of $f$ around $y$ is given by 
\[
\tilde f(1/2,1)= (f(0,1)+f(1,1))/2=(1+1)/2= 1.
\]
On the other hand, 
$y = (1/2,1)$
is the midpoint of 
$u=(0,0)$ and $v=(1,2)$
with
$\tilde f(u)= f(0,0) = 0$
and
$\tilde f(v)= f(1,2) = 0$.
This shows that the function $\tilde f$ is not convex,
and  $f$ is not integrally convex.
Also note that $0=\overline{f}(1/2,1) \not = \tilde f(1/2,1)=1$.
\finbox
\end{example}

Integrally convex functions in two variables ($n = 2$) can be defined 
by simple inequality conditions
without referring to the local convex extension $\tilde{f}$
(see Theorem~\ref{THcharICfndim2} in Section~\ref{SCcharICfn}).

\begin{remark} \rm \label{RMintcnvconcept}
The concept of integrally convex functions is introduced in \cite{FT90} 
for functions defined on a box of integers. 
The extension to functions with general integrally convex effective domains
is straightforward, which is found in \cite{Mdcasiam}.
\finbox
\end{remark}

\subsection{Examples}

Three classes of integrally convex functions are given below.

\begin{example} \rm  \label{EXsepIC}
A function
$\Phi: \ZZ^{n} \to \RR \cup \{ +\infty \}$
in $x=(x_{1}, x_{2}, \ldots,x_{n}) \in \ZZ^{n}$
is called  
{\em separable convex}
if it can be represented as
\begin{equation}  \label{sepvexdef}
\Phi(x) = \varphi_{1}(x_{1}) + \varphi_{2}(x_{2}) + \cdots + \varphi_{n}(x_{n})
\end{equation}
with univariate discrete convex functions
$\varphi_{i}: \ZZ \to \RR \cup \{ +\infty \}$, 
which means, by definition, that 
$\dom \varphi_{i}$ is an interval of integers and
\begin{equation}  \label{univarvexdef}
\varphi_{i}(k-1) + \varphi_{i}(k+1) \geq 2 \varphi_{i}(k)
\qquad (k \in \ZZ).
\end{equation}
A separable convex function is integrally convex
(actually, both \Lnat- and \Mnat-convex). 
\finbox
\end{example}

\begin{example} \rm  \label{EXdiagdomIC}
A symmetric matrix $Q =(q_{ij})$ that satisfies the condition
\begin{equation}\label{midptdiagdomdef}
q_{ii} \geq \sum_{j \neq i} |q_{ij}|
\qquad (i=1,2,\ldots,n)
\end{equation}
is called a diagonally dominant matrix (with nonnegative diagonals).
If  $Q$ is diagonally dominant
in the sense of \eqref{midptdiagdomdef},
then $f(x) = x\sp{\top} Q x$ is integrally convex
\cite[Proposition~4.5]{FT90}.
The converse is also true if $n \leq 2$
\cite[Remark~4.3]{FT90}.
Recently it has been shown in \cite[Theorem~9]{TT21ddmc}
that the diagonally dominance \eqref{midptdiagdomdef}
of $Q$ is equivalent 
to the directed discrete midpoint convexity of $f(x) = x\sp{\top} Q x$;
see \cite{TT21ddmc} for details.
\finbox
\end{example}

\begin{example} \rm \label{EXtwosepIC}
A function 
$f: \ZZ^{n} \to \RR \cup \{ +\infty \}$
is called
{\em 2-separable convex}
if it can be expressed as the sum of univariate convex,
diff-convex, and sum-convex functions, i.e., if
\begin{equation*}  
f(x_1, x_2, \ldots, x_n) = 
\sum_{i=1}\sp{n} \varphi_{i}(x_{i}) 
+ \sum_{i \neq j}  \varphi_{ij}(x_{i} - x_{j}) + \sum_{i \neq j} \psi_{ij}(x_{i}+x_{j})  ,
\end{equation*}
where
$\varphi_{i}, \varphi_{ij}, \psi_{ij}: \mathbb{Z} \to \mathbb{R} \cup \{ +\infty \}$
$(i, j =1,2,\ldots,n; \  i \not = j)$
are univariate convex functions.
A 2-separable convex function is known 
to be integrally convex \cite[Theorem~4]{TT21ddmc},
whereas it is \Lnat-convex if $\psi_{ij} \equiv 0$ for all $(i,j)$ with $i \ne j$.
A quadratic function $f(x) = x\sp{\top} Q x$ 
with $Q$ satisfying \eqref{midptdiagdomdef}
is an example of a 2-separable convex function.
\finbox
\end{example}

In addition to the above, 
almost all kinds of discrete convex functions treated in discrete convex analysis
are integrally convex.
It is known that separable convex,
{\rm L}-convex, \Lnat-convex, {\rm M}-convex,  
\Mnat-convex, \LLnat-convex, and 
\MMnat-convex functions are integrally convex \cite{Mdcasiam}.
Multimodular functions \cite{Haj85} 
are also integrally convex
\cite{Mdcaprimer07}.
Moreover, BS-convex and UJ-convex functions \cite{Fuj14bisubmdc}
are integrally convex.


\subsection{Characterizations}
\label{SCcharICfn}

In this section we give two characterizations
of integrally convex functions
in terms of an inequality of the form
\begin{equation}  \label{intcnvmidpt}
\tilde{f}\, \bigg(\frac{x + y}{2} \bigg) 
\leq \frac{1}{2} (f(x) + f(y)) ,
\end{equation}
where $\tilde{f}$ denotes the local convex extension of $f$ 
defined by \eqref{fnconvclosureloc2}.
By the definition of $\tilde{f}$, 
the inequality \eqref{intcnvmidpt} above is true
for $(x,y)$ with $\| x - y \|_{\infty} \leq 1$
for any function $f: \ZZ^{n} \to \RR \cup \{ +\infty  \}$.
If $f$ is integrally convex,
the inequality \eqref{intcnvmidpt} holds for any $(x,y)$,
as follows.

\begin{proposition} \label{PRiccharAtoBC}
If $f$ is integrally convex, then
\eqref{intcnvmidpt} 
holds for every $x, y \in \dom f$.
\end{proposition}
\begin{proof}
The function $\tilde{f}$ is convex by integral convexity of $f$, and hence
\begin{equation*}
\tilde{f}\, \bigg(\frac{x + y}{2} \bigg) 
\leq \frac{1}{2} (\tilde{f}(x) + \tilde{f}(y))
= \frac{1}{2} (f(x) + f(y)),
\end{equation*}
where the equalities
$\tilde{f}(x)=f(x)$ and $\tilde{f}(y)=f(y)$ 
by \eqref{ICfnextZ} are used.
\qedJIAM
\end{proof}

Integral convexity of a function can be characterized 
by a local condition under the assumption 
that the effective domain is an integrally convex set.

\begin{theorem}
[\cite{FT90,MMTT19proxIC}] 
\label{THfavtarProp33}
Let $f: \mathbb{Z}^{n} \to \mathbb{R} \cup \{ +\infty  \}$
be a function 
with an integrally convex effective domain.
Then the following properties are equivalent.

{\rm (a)}
$f$ is integrally convex.

{\rm (b)}
Inequality \eqref{intcnvmidpt} 
holds for every $x, y \in \dom f$ with $\| x - y \|_{\infty} =2$. 
\end{theorem}
\begin{proof}
\, [(a) $\Rightarrow$ (b)]: 
This is shown in Proposition~\ref{PRiccharAtoBC}.

[(b) $\Rightarrow$ (a)]:
(The proof given in \cite[Appendix A]{MMTT19proxIC} is sketched here.) \ 
For an integer vector ${a} \in \ZZ\sp{n}$,
define a box 
$B \subseteq \RR\sp{n}$ 
of size two by
\begin{equation} \label{size2boxICprf}
B = [ {a}, {a} + 2 \bm{1} ]_{\RR}
= \{ x \in \RR\sp{n} \mid a_{i} \leq x_{i} \leq a_{i} + 2
\ (i=1,2,\ldots,n) \}.
\end{equation}
It can be shown (\cite[Lemma~A.1]{MMTT19proxIC})
that, if $\dom f$ is integrally convex and the condition {\rm (b)} 
is satisfied, then
$\tilde{f}$ is convex on $B \cap \overline{\dom f}$.

Fix arbitrary $x, y \in \overline{\dom f}$,
and denote by $L$ the (closed) line segment connecting $x$ and $y$.
We show that $\tilde{f}$ is convex on $L$.
Consider the boxes $B$ of the form of \eqref{size2boxICprf}
that intersect $L$.
There exists a finite number of such boxes,
say,
$B_{1},B_{2}, \ldots, B_{m}$,
and $L$ is covered by the line segments $L_{j} = L \cap B_{j}$
$(j=1,2,\ldots, m)$.
That is,  $ L = \bigcup_{j=1}\sp{m} L_{j}$.
For each point $z \in L \setminus \{ x, y \}$,
there exists some $L_{j}$ that contains  $z$ 
in its interior, and 
$\tilde{f}$ is convex on $L_{j}$ by 
the above-mentioned fact.
Hence  $\tilde{f}$ is convex on $L$
(cf.~\cite[Lemma 2]{Tuy95}).
This implies the convexity of $\tilde{f}$,
that is, the integral convexity of $f$.
\qedJIAM
\end{proof}

The second characterization 
of integral convexity of a function 
is free from the assumption on the effective domain,
but is not a local condition as it refers to 
all pairs  $(x, y)$ with $\| x - y \|_{\infty} \geq 2$.

\begin{theorem}
[{\cite[Theorem~A.1]{MMTT20dmc}}]
\label{THicchardmc}
Let $f: \mathbb{Z}^{n} \to \mathbb{R} \cup \{ +\infty  \}$
be a function 
with $\dom f \neq \emptyset$.
Then the following properties are equivalent.

{\rm (a)}
$f$ is integrally convex.

{\rm (b)}
Inequality \eqref{intcnvmidpt} 
holds for every $x, y \in \dom f$ with $\| x - y \|_{\infty} \geq 2$. 
\end{theorem}
\begin{proof}
\, [(a) $\Rightarrow$ (b)]: 
This is shown in Proposition~\ref{PRiccharAtoBC}.

[(b) $\Rightarrow$ (a)]:
By Theorem~\ref{THfavtarProp33},
it suffices to show that 
$\dom f$ is an integrally convex set,
which follows from 
Theorem~\ref{THicSetmidpt}
applied to $S=\dom f$.
Note that the condition \eqref{intcnvsetdist234} in 
Theorem~\ref{THicSetmidpt}
 holds  by the assumption (b).
\qedJIAM
\end{proof}

\begin{remark} \rm \label{RMintcnvchar}
Theorem~\ref{THfavtarProp33} originates in \cite[Proposition~3.3]{FT90},
which shows the  equivalence of (a) and (b)
when the effective domain is a box of integers,
while their equivalence 
for a general integrally convex effective domain 
is proved in \cite[Appendix A]{MMTT19proxIC}. 
Theorem~\ref{THicchardmc} is given in 
\cite[Theorem~A.1]{MMTT20dmc}
with a direct proof without using Theorem~\ref{THfavtarProp33},
while here we have given an alternative proof that relies on 
Theorem~\ref{THfavtarProp33} via
Theorem~\ref{THicSetmidpt}.
\finbox
\end{remark}

Integrally convex functions in two variables ($n = 2$) can be characterized 
by simple inequality conditions as follows.
We use notation $f_{z}(x):=f(z+x)$.

\begin{theorem}  \label{THcharICfndim2}
A function 
$f: \mathbb{Z}^{2} \to \mathbb{R} \cup \{ +\infty  \}$
is integrally convex if and only if
its effective domain is an integrally convex set
and the following five inequalities
\begin{align*}
&  g( 0,0 ) +  g( 2,1 )   \geq g( 1,1 ) +  g( 1,0 ),
\\
&  g( 0,0 ) +  g( 2,-1 )  \geq g( 1,-1 ) +  g( 1,0 ),
\\
&  g( 0,0 ) +  g( 2,0 )  \geq 2 g( 1,0 ),
\\
&  g( 0,0 ) +  g( 2,2 )   \geq 2 g( 1,1 ),
\\
&  g( 0,0 ) +  g( 2,-2 )  \geq 2 g( 1,-1 )
\end{align*}
 are satisfied by both
$g(x_{1},x_{2})=f_{z}(x_{1},x_{2})$ 
and $g(x_{1},x_{2})=f_{z}(x_{2},x_{1})$
for any $z \in \dom f$.
\end{theorem}
\begin{proof}
This follows immediately from Theorem~\ref{THfavtarProp33},
since when $n=2$ and $\| x - y \|_{\infty}= 2$,
\eqref{intcnvmidpt} is equivalent to 
\begin{equation} \label{intcnvmidptdim2}
f(x) + f(y) \geq  
 f \left(\left\lceil \frac{x+y}{2} \right\rceil\right) 
 + f \left(\left\lfloor \frac{x+y}{2} \right\rfloor\right) .
\end{equation}
For example, if $x=z$ and $y = z +(2,1)$, then
$\left\lceil (x+y)/2 \right\rceil = z +(1,1)$ and
$\left\lfloor (x+y)/2 \right\rfloor  = z+ (1,0)$,
and \eqref{intcnvmidptdim2} gives the first inequality
for $g(x_{1},x_{2})=f_{z}(x_{1},x_{2})$. 
\qedJIAM
\end{proof}

\begin{remark} \rm \label{RMparaineqdim2}
For a function $g: \mathbb{Z}^{2} \to \mathbb{R} \cup \{ +\infty  \}$
(in general), an inequality of the form
\begin{equation} \label{paraineq2dimGen}
 g(0,0) + g(a+b,a) \geq  g(a,a) + g(b,0) 
\qquad ( a,b \geq 0; a,b \in \ZZ)
\end{equation}
is called the (basic) 
\kwd{parallelogram inequality}
in \cite{MMTT19proxIC}.
It is shown in \cite[Proposition~3.3]{MMTT19proxIC}
that for any integrally convex function $f$ in two variables
and a point $z \in \dom f$, the function $g(x)=f_{z}(x)$
satisfies the inequality \eqref{paraineq2dimGen}.
Note that \eqref{paraineq2dimGen} with $(a,b)=(1,1)$
coincides with the first inequality in 
Theorem~\ref{THcharICfndim2}.
Furthermore, the inequality \eqref{paraineq2dimGen} holds also for
$g(x_{1},x_{2})=f_{z}(x_{2},x_{1}), 
f_{z}(x_{1},-x_{2})$, and
$f_{z}(-x_{2},x_{1})$,
as integral convexity is preserved under 
such coordinate inversions
 (cf., \eqref{indepsigninvfndef},  \eqref{permfndef}).
\finbox
\end{remark}

In this section we have given three theorems 
(Theorems \ref{THfavtarProp33}, \ref{THicchardmc}, and \ref{THcharICfndim2})
to characterize integrally convex functions.
In Section~\ref{SCminzerIC} we give two additional theorems
(Theorems \ref{THfnargmincharExt} and \ref{THfnargmincharBnd}).
Their logical dependence (in our presentation) is illustrated in Fig.~\ref{FGthmchar}.

\begin{figure}
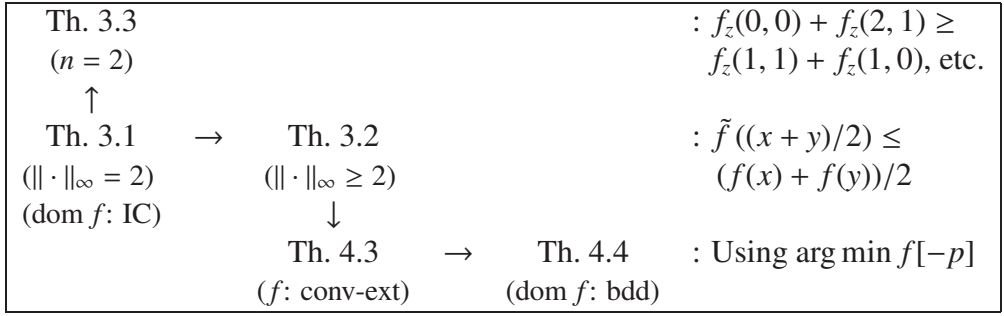

\centering
\begin{tabular}{|cccccl|}
\hline
 Th.~\ref{THcharICfndim2}
 & & & & &
: $f_{z}( 0,0 ) +  f_{z}( 2,1 ) \geq$  
\\
{\small ($n=2$)} & & & & & \ \ 
$f_{z}( 1,1 ) +  f_{z}( 1,0 )$, etc.
\\
 $\uparrow$
& 
& & & &
\\
Th.~\ref{THfavtarProp33} 
& $\rightarrow$ & Th.~\ref{THicchardmc}
& & &
: $\tilde{f}\, ((x + y)/{2} ) \leq$ 
\\
{\small ($ \| \cdot \|_{\infty} = 2$) }
&  
& 
{\small ($\| \cdot \|_{\infty}  \geq 2$) }
& & &
\ \ \  
$ (f(x) + f(y))/2$
\\
{\small ($\dom f$: IC)}
&  &  $\downarrow$  & & &
\\
 & & 
Th.~\ref{THfnargmincharExt}
& $\to$ & 
Th.~\ref{THfnargmincharBnd} 
& 
: Using $\argmin f[-p]$
\\
& & {\small ($f$: conv-ext)} & & {\small ($\dom f$: bdd)}  &
\\ \hline
\end{tabular}
\caption{Characterizations of integrally convex functions}
\label{FGthmchar}
\end{figure}


\subsection{Simplicial divisions}

As is well known 
(\cite[Section 16.3]{Fuj05book},
\cite[Section 7.7]{Mdcasiam}),
the convex extension of an L$\sp{\natural}$-convex function 
can be constructed in a systematic manner using
a regular simplicial division (the Freudenthal simplicial division) 
of unit hypercubes.
This is a generalization of the Lov{\'a}sz extension 
for a submodular set function. 
In addition, the concepts of BS-convex and UJ-convex functions 
are introduced on the basis of 
other regular simplicial divisions
in \cite{Fuj14bisubmdc}.

By definition, an integrally convex function $f$ is convex-extensible,
and its convex envelope
$\overline{f}$ can be constructed
locally within each unit hypercube, since $\overline{f}$ 
coincides with the local convex extension $\tilde{f}$.
However, general integrally convex functions
are not associated with a regular simplicial division.
Indeed, the following construction shows that, when $n=2$, 
an arbitrary triangulation 
can arise from an integrally convex function.

Consider the rectangular domain 
$[\veczero, a]_{\RR}$,
where $a = (a_{1}, a_{2})$
with positive integers $a_{1}$ and $a_{2}$,
and
assume that we are given an arbitrary triangulation
of each unit square in 
the domain 
$[\veczero, a]_{\RR}$
such as the one in Fig.~\ref{FGicvfndim2}(a).
We can construct an integrally convex function $f$
such that the convex envelope
$\overline{f}$
corresponds to the given triangulation.

\begin{figure}\begin{center}
\includegraphics[width=0.75\textwidth,clip]{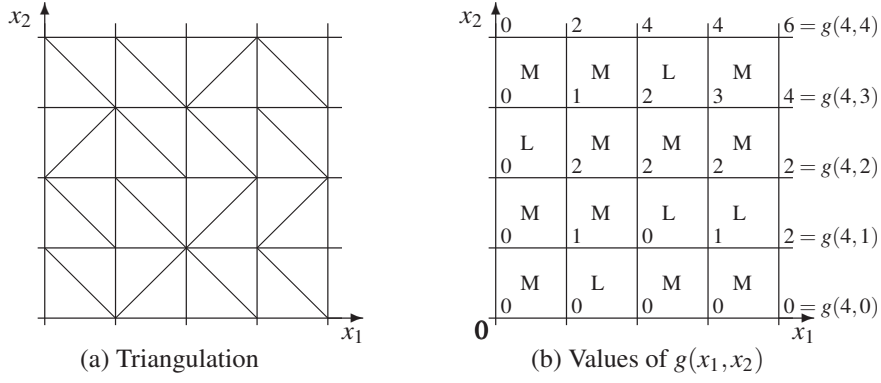}
\caption{A given triangulation and the corresponding function $g$.}
\label{FGicvfndim2}
\end{center}\end{figure}

According to the given triangulation,
we classify the unit squares
into two types, type M and type L.
We say that a unit square is of type M (resp., type L)
if it has a diagonal line segment on 
$x_{1} + x_{2} = c$ (resp., $x_{1} - x_{2} = c$)
for some $c$;
see Fig.~\ref{FGicvfndim2}(b).
For each $x = (x_{1}, x_{2}) \in [\veczero, a]_{\ZZ}$,
we denote 
the number of unit squares of type M (resp., type L) contained in 
the domain $[\veczero, x]_{\RR}$
by 
$g(x; {\rm M})$
(resp., $g(x; {\rm L})$),
and define 
$g(x) := g(x; {\rm M}) - g(x; {\rm L})$.
For $x=(2,2)$ in Fig.~\ref{FGicvfndim2}(b), for example,
we have
$g(x; {\rm M})=3$ and $g(x; {\rm L}) = 1$, and therefore
$g(2,2) = g(x; {\rm M}) - g(x; {\rm L}) = 3-1 = 2$.
Finally, we define a function 
$f$ 
on $[\veczero, a]_{\ZZ}$
by
\begin{equation}  \label{triangIC}
f(x_{1}, x_{2}) = A ( {x_{1}}\sp{2} + {x_{2}}\sp{2}) + g(x_{1}, x_{2})
\qquad (x \in [\veczero, a]_{\ZZ})
\end{equation}
with a positive constant $A$.
If $A \geq a_{1} + a_{2}$,
this function $f$ is integrally convex 
and 
the associated triangulation of each unit square
coincides with the given one
(proved in Remark~\ref{RMicdim2triang}).
It is noted that, while $f$ is integrally convex,  $g$ itself may not be integrally convex.
For example, 
in Fig.~\ref{FGicvfndim2}(b),
we have
\begin{align*}
 & [f(0,0) + f(2,1)] -  [f(1,0) + f(1,1)] =  2A -1 >0, 
\\
 & [g(0,0) + g(2,1)] -  [g(1,0) + g(1,1)] = -1 < 0 
\end{align*}
(cf., Theorem~\ref{THcharICfndim2}).

\begin{remark} \rm  \label{RMicdim2triang}
First, we prove the integral convexity of $f$ in \eqref{triangIC} 
by showing that
\begin{equation} \label{midptcnvfn=icdim2triang}
 f(x) + f(y) 
  - f \left(\left\lceil \frac{x+y}{2} \right\rceil\right) 
  - f \left(\left\lfloor \frac{x+y}{2} \right\rfloor\right) 
\geq 0
\end{equation}
holds for every $x, y  \in [\veczero, a]_{\ZZ}$
with $\| x - y \|_{\infty} =2$.
By symmetry between $x$ and $y$ and that between coordinate axes, 
we have five cases to consider
(cf., Theorem~\ref{THcharICfndim2}):
(i) $y=x+(2,1)$, 
(ii) $y=x+(2,-1)$, 
(iii) $y=x+(2,0)$,
(iv) $y=x+(2,2)$, and
(v) $y=x+(2,-2)$.
Let $h(x) := {x_{1}}\sp{2} + {x_{2}}\sp{2}$, for which we have
\[
 h(x) + h(y) -
   h \left(\left\lceil \frac{x+y}{2} \right\rceil\right) 
  - h \left(\left\lfloor \frac{x+y}{2} \right\rfloor\right) 
= \begin{cases} 
  2 & \mbox{(case (i), (ii), (iii))}, \\ 
  4  & \mbox{(case (iv), (v))}. 
 \end{cases}
\]
On the other hand, we have
\[
\left|
 g(x) + g(y) -
   g \left(\left\lceil \frac{x+y}{2} \right\rceil\right) 
  - g \left(\left\lfloor \frac{x+y}{2} \right\rfloor\right) 
 \right|
\leq   2( a_{1} + a_{2})
\]
in either case.
Therefore, if $A \geq a_{1} + a_{2}$, the inequality 
in \eqref{midptcnvfn=icdim2triang} holds.

Next, we observe that the function $f = A h + g$ induces a triangulation 
of the specified type within each unit square.
Consider a square $[x, y]_{\RR}$ with $y=x+\vecone$ and $x, y  \in [\veczero, a]_{\ZZ}$.
Let 
$u := (x_{1}+1,x_{2})$ and $ v:= (x_{1},x_{2}+1)$.
Then 
$h(x) + h(y) - h(u) - h(v) = 0$, while
$g(x) + g(y) - g(u) - g(v)$ is equal to $+1$ and $-1$
according to whether the square $[x, y]_{\RR}$ is
of type M or L.
This shows that the triangulation of $[x, y]_{\RR}$ 
induced by $f$ coincides with the given one.
\finbox
\end{remark}

Next we give an example of a simplicial division associated 
with an integrally convex function in three variables.

\begin{figure}\begin{center}
\includegraphics[width=0.40\textwidth,clip]{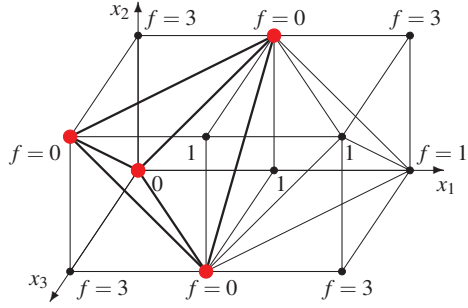}
\caption{Simplicial division associated with an integrally convex function.}
\label{FGconvextICdim3}
\end{center}\end{figure}

\begin{example} \rm \label{EXicext2}
Consider 
$S = [ (0,0,0), (2,1,1) ]_{\ZZ}
=\{ x \in \ZZ\sp{3} 
\mid  0 \leq x_{1} \leq 2, \ 0 \leq x_{i} \leq 1 \ (i=2,3) \}$
(see Fig.~\ref{FGconvextICdim3}), 
and define
$f : \ZZ\sp{3} \rightarrow \RR \cup \{ +\infty \}$ with $\dom f = S$
by
\[
( f(x_{1},0,0) )_{x_{1}=0,1,2}
= (0,1,1),
\quad
( f(x_{1},1,0)  )_{x_{1}=0,1,2}
= (3,0,3),
\]
and
$f(x_{1},0,0) = f(x_{1},1,1)$,
$f(x_{1},1,0) = f(x_{1},0,1)$
for 
$x_{1}=0,1,2$.
By Theorem~\ref{THfavtarProp33} we can verify that $f$ is integrally convex.
For example, for 
$x= (0,0,0)$ and 
$y= (2,1,1)$, 
we have
\[
\tilde{f} \big(\frac{x + y}{2} \big) 
=\tilde{f}(1,\frac{1}{2},\frac{1}{2}) 
=\frac{1}{2} (f(1,1,0) + f(1,0,1)) = 0
\leq  \frac{1}{2} (f(x) + f(y)) = \frac{1}{2}
\]
in \eqref{intcnvmidpt}.
The simplicial division of 
$\overline{S} = [ (0,0,0), (2,1,1) ]_{\RR}$
for the convex extension of $f$
is symmetric with respect to the plane $x_{1}=1$.
The left cube 
$[ (0,0,0), (1,1,1) ]_{\RR}$
is decomposed into five simplices;
one of them has vertices at $(0,0,0)$, $(1,1,0)$, $(1,0,1)$, $(0,1,1)$
and has volume $1/3$
(drawn in bold line), whereas the other four simplicies 
are congruent to the standard simplex, having volume $1/6$.
The right cube 
$[ (1,0,0), (2,1,1) ]_{\RR}$
is decomposed similarly into five simplices;
one of them has vertices at 
$(2,0,0)$, $(1,1,0)$, $(1,0,1)$, $(2,1,1)$, 
and has volume $1/3$, whereas the other four simplicies 
are congruent to the standard simplex, having volume $1/6$.
Thus the simplices are not uniform in volume,
whereas they have the same volume ($=1/6$) for an \Lnat-convex function.
It is added that this function $f$  
is neither \LLnat-convex nor \MMnat-convex,
since 
$\argmin f = \{  (0,0,0),  \ (1,1,0), \  (1,0,1), \  (0,1,1) \}$
is neither \LLnat-convex nor \MMnat-convex
 (as discussed in Example~\ref{EXnonL2M2boxTDI}).
\finbox
\end{example}


\subsection{Basic operations}
\label{SCoperfn}

In this section we show how integral convexity of a function 
behaves under basic operations.
Let $f$ be a function on $\ZZ\sp{n}$, i.e.,  
$f: \ZZ\sp{n} \to \RR \cup \{ +\infty \}$.

\paragraph{Origin shift:}

For an integer vector $b \in \ZZ\sp{n}$, the
{\em origin shift}
of $f$ by $b$
means a function $g$ on $\ZZ\sp{n}$ defined by
$g(y) = f(y-b)$.
The origin shift of 
an integrally convex function
is an integrally convex function.

\paragraph{Inversion of coordinates:}

The {\em independent coordinate inversion} of $f$ 
means a function $g$ on $\ZZ\sp{n}$ defined by
\begin{equation} \label{indepsigninvfndef}
g(y_{1},y_{2}, \ldots, y_{n}) = 
f(\tau_{1} y_{1}, \tau_{2} y_{2}, \ldots,  \tau_{n}y_{n})
\end{equation}
with an arbitrary choice of $\tau_{i} \in \{ +1, -1 \}$ $(i=1,2,\ldots,n)$.
The independent coordinate inversion of an integrally convex function 
is an integrally convex function.
This is a nice property of integral convexity, not shared by
\Lnat-, \LLnat-, \Mnat-, or \MMnat-convexity.

\paragraph{Permutation of coordinates:}

For a permutation $\sigma$ of $(1,2,\ldots,n)$,
the {\em permutation} of $f$ by $\sigma$
means a function $g$ on $\ZZ\sp{n}$ defined by
\begin{equation} \label{permfndef}
 g(y_{1},y_{2}, \ldots, y_{n}) =  f(y_{\sigma(1)}, y_{\sigma(2)}, \ldots, y_{\sigma(n)}) .
\end{equation}
The permutation of an integrally convex function is integrally convex.

\paragraph{Variable-scaling:}

For a positive integer $\alpha$,
the {\em variable-scaling} 
(or {\em scaling} for short) of $f$ by $\alpha$
means a function $g$ on $\ZZ\sp{n}$ defined by
\begin{equation} \label{scalefndef}
g(y_{1},y_{2}, \ldots, y_{n}) = f(\alpha y_{1}, \alpha y_{2}, \ldots, \alpha y_{n}).
\end{equation}
Note that the same scaling factor $\alpha$ is used for all coordinates.
If $\alpha = 2$, for example, this operation amounts to considering the
function values at even points.
The scaling operation is used effectively in
minimization algorithms (see Section~\ref{SCminiztn}).
The scaling of an integrally convex function
is not necessarily integrally convex.
The indicator function 
$f=\delta_{S}$ of the integrally convex set $S \subseteq \ZZ\sp{3}$
in Example~\ref{EXscICsetNG422} is such an example.
Another example of a function 
on the integer box $[(0,0,0), (4,2,2)]_{\ZZ}$
can be found in \cite[Example 3.1]{MMTT19proxIC}.
In the case of $n = 2$, 
integral convexity admits the scaling operation.
That is,
if 
$f: \ZZ^{2} \to \RR \cup \{ +\infty  \}$
is integrally convex, then 
$g: \ZZ^{2} \to \RR \cup \{ +\infty  \}$
is integrally convex
(\cite[Theorem~3.2]{MMTT19proxIC}).

As another kind of scaling, 
the dilation operation can be defined for a function
$f: \ZZ\sp{n} \to \RR \cup \{ +\infty \}$
if it is convex-extensible.
For any positive integer $\alpha$,
the 
\kwd{$\alpha$-dilation}
of $f$ is 
defined as a function
$g$ on $\ZZ\sp{n}$ 
given by 
\begin{equation} \label{dilatfndef}
g(y_{1},y_{2}, \ldots, y_{n}) 
= \overline{f}(y_{1}/\alpha , y_{2}/\alpha , \ldots, y_{n}/\alpha),
\end{equation}
where  $\overline{f}$ denotes the convex envelope of $f$.
The dilation of an integrally convex function
is not necessarily integrally convex.
For example, the indicator function 
$f=\delta_{S}$ of the integrally convex set $S \subseteq \ZZ\sp{4}$
in Example~\ref{EXnonboxTDI2} is an integrally convex function
for which the $2$-dilation is not integrally convex.

\begin{remark} \rm  \label{RMscaldilfn}
Although integral convexity is not compatible with 
the dilation operation,
other kinds of discrete convexity such as
L-, \Lnat-, \LL-, \LLnat-, M-, 
\hbox{\Mnat-}, 
\MM-, \MMnat-convexity, and multimodularity
are preserved under dilation.
In contrast, the scaling operation in \eqref{scalefndef} 
preserves L-convexity and its relatives 
(box, L-, \Lnat-, \LL-, \LLnat-convexity, and multimodularity),
and not M-convexity and its relatives 
(M-, \Mnat-, \MM-, \MMnat-convexity).
\finbox
\end{remark}

\paragraph{Value-scaling:}

For a function
$f: \ZZ\sp{n} \to \RR \cup \{ +\infty \}$ 
and a nonnegative factor $a \geq 0$,
the {\em value-scaling} of $f$ by $a$ 
means  a function $g: \ZZ\sp{n} \to \RR \cup \{ +\infty \}$ defined by
$g(y) = a f(y)$ for $y \in \ZZ\sp{n}$.
We may also introduce an additive constant $b \in \RR$ 
and a linear function $\langle c, y \rangle = \sum_{i=1}\sp{n} c_{i} y_{i}$,
where $c \in \RR\sp{n}$,
to obtain
\begin{equation} \label{lintranfndef}
 g(y) = a f(y) + b + \langle c, y \rangle
  \qquad (y \in \ZZ\sp{n}) .
\end{equation}
The operation \eqref{lintranfndef} 
preserves integral convexity of a function.

\paragraph{Restriction:}

For a function
$f: \ZZ\sp{N} \to \RR \cup \{ +\infty \}$ 
and a subset $U$ of the index set 
$N = \{ 1,2,\ldots, n \}$,
the
{\em restriction}
of $f$ to $U$ is a function $g: \ZZ\sp{U} \to \RR \cup \{ +\infty \}$ defined by
\begin{equation*}  
 g(y) = f(y,\veczero_{N \setminus U})
  \qquad (y \in \ZZ\sp{U}) ,
\end{equation*}
where $\veczero_{N \setminus U}$ denotes the zero vector 
in $\ZZ\sp{N \setminus U}$.
The notation 
$(y,\veczero_{N \setminus U})$ means the vector
whose $i$th component is equal to $y_{i}$ for $i \in U$
and to 0 for $i \in N \setminus U$.
The restriction 
of an integrally convex function is integrally convex 
(if the effective domain of the resulting function is nonempty).

\paragraph{Projection:}

For a function
$f: \ZZ\sp{N} \to \RR \cup \{ +\infty \}$ 
and a subset $U$ of the index set $N = \{ 1,2,\ldots, n \}$,
the {\em projection}
of $f$ to $U$ is a function
$g: \ZZ\sp{U} \to \RR \cup \{ -\infty, +\infty \}$
defined by
\begin{equation} \label{projfndef} 
  g(y)  =  \inf \{ f(y,z) \mid z \in \ZZ\sp{N \setminus U} \}
  \qquad (y \in \ZZ\sp{U}) ,
\end{equation}
where the notation $(y,z)$ means the vector
whose $i$th component is equal to $y_{i}$ for $i \in U$
and to $z_{i}$ for $i \in N \setminus U$.
The projection is also called {\em partial minimization}.
The resulting function $g$ is referred to as 
the {\em marginal function} of $f$ in \cite{HL01}. 
The projection of an integrally convex function
is integrally convex
(\cite[Theorem~3.1]{MM19projcnvl})
if $g > -\infty$,
or else we have $g \equiv -\infty$ 
(see Proposition~\ref{PRprojICinfty} in Section~\ref{SCtechfn}).

\paragraph{Splitting:}

Suppose that we are given a family 
$\{ U_{1},  U_{2}, \dots , U_{n} \}$
of  disjoint nonempty sets
indexed by $N = \{ 1, 2, \dots , n\}$. 
Let $m_{i}= |U_{i}|$ for $i=1,2,\ldots, n$ and define 
$m= \sum_{i=1}\sp{n} m_{i}$, where $m \geq n$.
For each $i \in N$
we define an $m_{i}$-dimensional vector
$y_{[i]} = ( y_{j} \mid  j \in U_{i} )$
and express $y \in \ZZ\sp{m}$ as
$y = (y_{[1]}, y_{[2]}, \dots , y_{[n]})$. 
For a function $f: \ZZ \sp{n} \to \RR \cup \{+\infty\}$, 
the {\em splitting} of $f$ 
by $\{ U_{1},  U_{2}, \dots , U_{n} \}$
is defined as a function
$g: \ZZ \sp{m} \to \RR \cup \{+\infty\}$ given by
\begin{equation*} 
  g(y_{[1]}, y_{[2]}, \dots , y_{[n]}) 
  = f( y_{[1]}(U_{1}), y_{[2]}(U_{2}), \dots , y_{[n]}(U_{n})), 
\end{equation*}
where $y_{[i]}(U_{i}) = \sum\{  y_{j} \mid j \in U_{i} \}$. 
For example,
$g(y_{1}, y_{2}, y_{3}) = f(y_{1}, y_{2}+ y_{3})$
is a splitting of 
$f: \ZZ \sp{2} \to \RR \cup \{+\infty \}$
for $U_{1} = \{ 1 \}$ and $U_{2} = \{ 2,3 \}$,
where $n=2$ and $m=3$.
The splitting of an integrally convex function is integrally convex
(\cite[Proposition~4.4]{Mopernet21}).

\paragraph{Aggregation:}

Let 
$\mathcal{P} =  \{ N_{1},  N_{2}, \dots , N_{m} \}$
 be a  partition of $N = \{ 1,2, \ldots, n \}$ into disjoint nonempty subsets:
$N = N_{1} \cup N_{2} \cup \dots \cup N_{m}$
and
$N_{i} \cap N_{j} = \emptyset$ for $i \not= j$.
We have $m \leq n$.
For a function $f: \ZZ \sp{N} \to \RR \cup \{+\infty \}$, 
the {\em aggregation} of $f$ with respect to $\mathcal{P}$
is defined as a function
$g : \ZZ \sp{m} \to \RR \cup \{+\infty, -\infty\}$ given by 
\begin{equation*}  
g(y_{1}, y_{2}, \dots , y_{m}) 
= \inf \{ f(x)   \mid x(N_{j}) = y_{j} \ (j=1,2,\ldots,m) \} .
\end{equation*}
For example,
$g(y_{1}, y_{2}) = \inf \{ f(x_{1}, x_{2}, x_{3}) 
   \mid x_{1} = y_{1}, \ x_{2} + x_{3} = y_{2} \}$
is an aggregation of
$f: \ZZ \sp{3} \to \RR \cup \{+\infty \}$
for $N_{1} = \{ 1 \}$ and $N_{2} = \{ 2,3 \}$,
where $n=3$ and $m=2$.
The aggregation of an integrally convex function
is not necessarily integrally convex
(Example~\ref{EXicvsetaggr}).

\paragraph{Direct sum:}

The {\em direct sum} of two functions 
$f_{1}: \ZZ\sp{n_{1}} \to \RR \cup \{ +\infty \}$ and
$f_{2}: \ZZ\sp{n_{2}} \to \RR \cup \{ +\infty \}$
is a function 
$f_{1} \oplus f_{2}: \ZZ\sp{n_{1}+n_{2}} \to \RR \cup \{ +\infty \}$
defined as
\begin{equation*} 
(f_{1} \oplus f_{2})(x,y)= f_{1}(x) + f_{2}(y)
\qquad (x \in \ZZ\sp{n_{1}}, y \in \ZZ\sp{n_{2}}) .
\end{equation*}
The direct sum of two integrally convex functions is integrally convex.

\paragraph{Addition:}

The {\em sum} of two functions 
$f_{1}, f_{2}: \ZZ\sp{n} \to \RR \cup \{ +\infty \}$
is defined by
\begin{equation} \label{fnsumdef}
(f_{1} + f_{2})(x)= f_{1}(x) + f_{2}(x)
\qquad (x \in \ZZ\sp{n}) .
\end{equation}
For two sets $S_{1}$, $S_{2} \subseteq \ZZ\sp{n}$, 
the sum of their indicator functions
$\delta_{S_{1}}$ and $\delta_{S_{2}}$
coincides with 
the indicator function of their intersection
$S_{1} \cap S_{2}$,
that is,
$\delta_{S_{1}} +  \delta_{S_{2}} = \delta_{S_{1} \cap S_{2}}$. 
The sum of integrally convex functions
is not necessarily integrally convex
(Example~\ref{EXicsetinter}).
However, the sum of an integrally convex function
with a separable convex function 
\begin{equation} \label{f0sumphi}
 g(x) = f(x) + \sum_{i=1}\sp{n} \varphi_{i}(x_{i})
\qquad (x \in \ZZ\sp{n})
\end{equation}
is integrally convex.

\paragraph{Convolution:}

The (infimal) {\em convolution} of two functions
$f_{1}, f_{2}: \ZZ\sp{n} \to \RR \cup \{ +\infty \}$
is defined by
\begin{equation} \label{f1f2convdef}
(f_{1} \conv f_{2})(x) =
 \inf\{ f_{1}(y) + f_{2}(z) \mid x= y + z, \  y, z \in \ZZ\sp{n}  \}
\quad (x \in \ZZ\sp{n}) ,
\end{equation}
where it is assumed that the infimum is bounded from below 
(i.e., $(f_{1} \conv f_{2})(x) > -\infty$ for every $x \in \ZZ\sp{n}$).
For two sets $S_{1}$, $S_{2} \subseteq \ZZ\sp{n}$, 
the convolution of their indicator functions
$\delta_{S_{1}}$ and $\delta_{S_{2}}$
coincides with 
the indicator function of their Minkowski sum
$S_{1}+S_{2} = \{ y + z \mid y \in S_{1}, z \in S_{2} \}$,
that is,
$\delta_{S_{1}} \conv \delta_{S_{2}} = \delta_{S_{1}+S_{2}}$. 
The convolution of integrally convex functions
is not necessarily integrally convex
(Example~\ref{EXicdim2sumhole}).
The convolution of an integrally convex function and 
a separable convex function is integrally convex
\cite[Theorem~4.2]{MM19projcnvl}
(also \cite[Proposition~4.17]{Msurvop21}).

\begin{remark} \rm  \label{RMmlfnconvol}
The convolution operation
plays a central role in discrete convex analysis.
The convolution of two (or more) M$\sp{\natural}$-convex functions
is M$\sp{\natural}$-convex.
The convolution of two L$\sp{\natural}$-convex functions
is not necessarily L$\sp{\natural}$-convex,
but it is integrally convex.
The convolution of three \Lnat-convex functions
is no longer integrally convex
(Remark~\ref{RMmlsetMinkow}).
\finbox
\end{remark}


\subsection{Technical supplement}
\label{SCtechfn}

This section is a technical supplement concerning the projection 
operation defined in \eqref{projfndef}.
We first consider a direction $d$
in which the function value diverges to $-\infty$.

\begin{proposition}    \label{PRinfdirICfn}
Let $f: \ZZ\sp{n} \to \RR \cup \{ +\infty \}$ be an integrally convex function,
$y \in \dom f$, and $d \in \ZZ\sp{n}$.
If 
$\displaystyle \lim_{k \to \infty} f(y + k d)= -\infty$,
then for any $z \in \dom f$, we have
$\displaystyle \lim_{k \to \infty} f(z + k d)= -\infty$.
\end{proposition}
\begin{proof}
Let $S = \dom f$.
For each $x \in S$,
$g_{x}(k) = f(x + k d)$
is a convex function in $k \in \ZZ$,
which follows from the convex-extensibility of $f$.
Let $T$ denote the set of $x \in S$
for which $\displaystyle \lim_{k \to \infty} g_{x}(k) = -\infty$.
We  want to show that $T = \emptyset$ or $T=S$.
To prove this by contradiction, assume that both
$T$ and $S \setminus T$ are nonempty.
Consider $y \in T$ and $z \in S \setminus T$ 
that minimize $C \| y-z \|_{\infty} + \| y-z \|_{1}$,
where $C \gg 1$.
From integral convexity of $S$, we can easily show that
$\| y-z \|_{\infty} =1$ and
$S \cap N \big(\frac{y + z}{2} \big) = \{ y,z \}$.
Moreover, for
$y\sp{(k)} = y + k d$ and
$z\sp{(k)} = z + k d$,
we have
\begin{equation} \label{infdirICfnPrf1}
S \cap N \bigg(\frac{y\sp{(k)} + z\sp{(k)}}{2} \bigg) = \{ y\sp{(k)}, z\sp{(k)} \}
\qquad
(k=0,1,2,\ldots).
\end{equation}
(Proof of \eqref{infdirICfnPrf1}:
Since $g_{y}$ is convex, $y \in S$, and
$\displaystyle \lim_{k \to \infty} g_{y}(k) = -\infty$,
we have $y\sp{(k)} = y + k d \in S$ for all $k$.
This implies, by Proposition~\ref{PRinfdirICset},
that $z\sp{(k)} = z + k d \in S$ for all $k$.
We also have
$C \| y\sp{(k)}-z\sp{(k)} \|_{\infty} + \| y\sp{(k)}-z\sp{(k)} \|_{1}
= C \| y-z \|_{\infty} + \| y-z \|_{1}$.) \
Since $f$ is integrally convex, we have
\[
\tilde{f}\, \bigg(\frac{y\sp{(k+1)} + z\sp{(k-1)}}{2} \bigg) 
\leq \frac{1}{2} (f(y\sp{(k+1)}) + f(z\sp{(k-1)})) 
\]
by Proposition~\ref{PRiccharAtoBC},
where the left-hand side can be expressed as 
\[
\tilde{f}\, \bigg(\frac{y\sp{(k+1)} + z\sp{(k-1)}}{2} \bigg) 
=
\tilde{f}\, \bigg(\frac{y\sp{(k)} + z\sp{(k)}}{2} \bigg) 
=
 \frac{1}{2} (f(y\sp{(k)}) + f(z\sp{(k)})) 
\]
by 
$y\sp{(k+1)} + z\sp{(k-1)} = y\sp{(k)} + z\sp{(k)} $
and \eqref{infdirICfnPrf1}. 
Therefore, we have
\[
f(y\sp{(k)}) + f(z\sp{(k)}) \leq f(y\sp{(k+1)}) + f(z\sp{(k-1)}).
\]
By adding these inequalities for $k=1,2,\ldots, \hat k$, we obtain
\[
 f(z\sp{(\hat k)}) \leq f(y\sp{(\hat k+1)}) + f(z) - f(y\sp{(1)}) .
\]
By letting $\hat k \to \infty$ we obtain a contradiction,
since the right-hand side tends to $-\infty$ while
the left-hand side does not.
\qedJIAM
\end{proof}

Using the above proposition we can show that
the projection 
$g(y)  =  \inf \{ f(y,z) \mid z \in \ZZ\sp{N \setminus U} \}$,
defined in \eqref{projfndef},
is away from the value of $-\infty$
unless it is identically equal to $-\infty$.

\begin{proposition}  \label{PRprojICinfty}
Let $f: \ZZ\sp{n} \to \RR \cup \{ +\infty \}$ be an integrally convex function,
and $g$ be
the projection of $f$ to $U$.
If $g(y\sp{0}) = -\infty$ for some $y\sp{0} \in \ZZ\sp{U}$, 
then 
$g(y) = -\infty$ for all $y \in \ZZ\sp{U}$.
\end{proposition}
\begin{proof}
First suppose $|N \setminus U|=1$ 
with $N \setminus U = \{ v \}$,
and
assume $g(y\sp{0}) = -\infty$ for some $y\sp{0} \in \ZZ\sp{U}$.
Then $x\sp{0} := (y\sp{0}, z\sp{0}) \in \dom f$ for some $z\sp{0} \in \ZZ$,
and 
$\displaystyle \lim_{k \to \infty} f(x\sp{0} + k d)= -\infty$
for $d=\unitvec{v}$ or $d=-\unitvec{v}$.
This implies, by Proposition~\ref{PRinfdirICfn},
that 
$\displaystyle \lim_{k \to \infty} f(x + k d)= -\infty$
for all $x \in \dom f$, which shows $g \equiv -\infty$.
Next we consider the case where $|N \setminus U| \geq 2$.
Let $N \setminus U =\{ v_{1}, v_{2}, \ldots, v_{r} \}$
with $2 \leq r < n$,
and
$g\sp{(k)}$ denote the projection of $f$ to 
$U \cup \{ v_{k+1}, v_{k+2}, \ldots, v_{r} \}$
for $k =0,1,\ldots, r$.
Then $g\sp{(0)}=f$,
$g\sp{(r)}=g$, and
$g\sp{(k)}$ is the projection of $g\sp{(k-1)}$
for $k =1,2,\ldots, r$.
By the argument for the case with $|N \setminus U|=1$, we have
$g\sp{(1)} > -\infty$ 
or 
$g\sp{(1)} \equiv -\infty$.
In the latter case we obtain
$g\sp{(k)} \equiv -\infty$ for $k=1,2,\ldots,r$.
In the former case,
$g\sp{(1)}$ is an integrally convex function
by \cite[Theorem~3.1]{MM19projcnvl}, as already mentioned
in Section~\ref{SCoperfn}.
This allows us to apply the same argument 
to $g\sp{(1)}$ to obtain that
$g\sp{(2)} > -\infty$ 
or 
$g\sp{(2)} \equiv -\infty$.
Continuing this way we arrive at the conclusion that
$g > -\infty$ 
or 
$g \equiv -\infty$.
\qedJIAM
\end{proof}



\section{Minimization and minimizers}
\label{SCminiztn}

\subsection{Optimality conditions}
\label{SCmincrit}

The global minimum of an integrally convex function 
can be characterized by a local condition.

\begin{theorem}[{\cite[Proposition~3.1]{FT90}}; 
  see also {\cite[Theorem~3.21]{Mdcasiam}}] 
  \label{THintcnvlocopt}
Let $f: \mathbb{Z}^{n} \to \mathbb{R} \cup \{ +\infty  \}$
be an integrally convex function and $x^{*} \in \dom f$.
Then $x^{*}$ is a minimizer of $f$
if and only if
$f(x^{*}) \leq f(x^{*} +  d)$ for all 
$d \in  \{ -1, 0, +1 \}^{n}$.
\finboxARX
\end{theorem}

A more general form of this local optimality criterion 
is known as ``box-barrier property'' 
in Theorem~\ref{THintcnvbox} below (see Fig.~\ref{FGboxbarrier}).
A special case of Theorem~\ref{THintcnvbox} 
with $\hat x = x\sp{*}$, $p= x\sp{*} - \bm{1}$, and $q= x\sp{*} + \bm{1}$
coincides with Theorem~\ref{THintcnvlocopt} above.

\begin{theorem}[Box-barrier property {\cite[Theorem~2.6]{MMTT19proxIC}}] \label{THintcnvbox}
Let $f: \mathbb{Z}^{n} \to \mathbb{R} \cup \{ +\infty  \}$
be an integrally convex function, 
and 
$p \in (\mathbb{Z} \cup \{ -\infty  \})^{n}$
and
$q \in (\mathbb{Z} \cup \{ +\infty  \})^{n}$,
where $p \leq q$.
Define%
\begin{align*}
S &= \{ x \in \mathbb{Z}^{n} \mid p_{i} < x_{i} < q_{i} \ (i=1,2,\ldots,n) \},
\\
W_{i}^{+} &= \{ x \in \mathbb{Z}^{n} \mid 
 x_{i} = q_{i}, \  p_{j} \leq x_{j} \leq q_{j} \ (j \not= i) \}
\quad (i=1,2,\ldots,n),
\\
W_{i}^{-} &= \{ x \in \mathbb{Z}^{n} \mid 
 x_{i} = p_{i}, \  p_{j} \leq x_{j} \leq q_{j} \ (j \not= i) \}
\quad (i=1,2,\ldots,n),
\end{align*}
and $W = \bigcup_{i=1}^{n} (W_{i}^{+} \cup W_{i}^{-})$.
Let $\hat x \in S \cap \dom f$.
If 
$f(\hat x) \leq f(y)$ for all $y \in W$,
then 
$f(\hat x) \leq f(z)$ for all $z \in \ZZ^{n} \setminus S$.
\end{theorem}
\begin{proof}
(The proof of \cite{MMTT19proxIC} is described here.) \ 
Let $U_{i}^{+}$ and $U_{i}^{-}$
denote the convex hulls of
$W_{i}^{+}$ ans $W_{i}^{-}$, respectively, 
and define
$U = \bigcup_{i=1}^{n} (U_{i}^{+} \cup U_{i}^{-})$.
Then $W = U \cap \mathbb{Z}^{n}$. 
For a point $z \in \ZZ^{n} \setminus S$, 
the line segment connecting $\hat x$ and $z$ 
intersects $U$ at a point, say, $u \in \RR^{n}$.
Then its integral neighborhood $N(u)$ 
is contained in $W$.
Since the local convex extension $\tilde{f}(u)$ 
is a convex combination of 
the $f(y)$'s  with $y \in N(u)$,
and 
$f(y) \geq f(\hat x)$ for every $y \in W$,
we have 
$\tilde{f}(u) \geq f(\hat x)$.
On the other hand, it follows from 
the convexity of $\tilde{f}$ that
$\tilde{f}(u) \leq (1 - \lambda) f(\hat x) + \lambda f(z)$
for some $\lambda$ with $0 < \lambda \leq 1$.
Hence 
$f(\hat x) \leq \tilde{f}(u) \leq (1 - \lambda) f(\hat x) + \lambda f(z)$,
and therefore,
$f(\hat x) \leq f(z)$.
\qedJIAM
\end{proof}

\begin{figure}\begin{center}
\includegraphics[height=30mm]{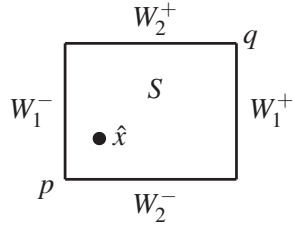}
\caption{Box barrier property}
\label{FGboxbarrier}
\end{center}\end{figure}

\begin{remark} \rm  \label{RMintconvoptverify}
The optimality criterion in 
Theorem~\ref{THintcnvlocopt}
is certainly local, but not satisfactory
from the computational complexity viewpoint.
We need $O(3\sp{n})$ function evaluations
to verify the local optimality condition.
\finbox
\end{remark}

\subsection{Minimizer sets}
\label{SCminzerIC}

It is (almost) always the case that 
if a function 
is equipped with some kind of discrete convexity,
then the set of its minimizers
is equipped with the discrete convexity of the same kind.
This is indeed the case with an integrally convex function.

\begin{proposition}  \label{PRargminICset}
Let $f: \ZZ\sp{n} \to \RR \cup \{ +\infty \}$ be an integrally convex function.
If it attains a (finite) minimum,
the set of its minimizers 
is integrally convex.
\end{proposition}
\begin{proof}
Let $\alpha$ denote the minimum value of $f$,
and $S :=\argmin f$.
Take any $x \in \overline{S}$.
We have
$\alpha = \overline{f}(x) = \tilde{f}(x)$.
By the definition
of $\tilde{f}$
in \eqref{fnconvclosureloc2},
this implies $x \in  \overline{S \cap N(x)}$.
Thus \eqref{icsetdef1} holds, showing the integral convexity of $S$.
\qedJIAM
\end{proof}

In the following we discuss how integral convexity of a function
can be characterized in terms of the integral convexity of the minimizer sets.
For a function 
$f: \ZZ\sp{n} \to \RR \cup \{ +\infty \}$ 
and a vector $p \in \RR\sp{n}$, 
$f[-p]$ will denote the function defined by
\begin{equation}\label{notatfp}
f[-p](x) 
= f(x) - \langle p , x \rangle  
= f(x) - \sum_{i=1}\sp{n} p_{i} x_{i}
\qquad (x \in \ZZ\sp{n}).
\end{equation}
We use notation
\begin{equation}\label{notatargminfp}
\argmin f[-p] = \{ x \in \ZZ\sp{n} \mid 
   f[-p](x) \leq f[-p](y) \mbox{ for all $y \in \ZZ\sp{n}$} \}
\end{equation}
for the set of the minimizers of $f[-p]$.

\begin{theorem}  \label{THfnargmincharExt}
Let $f: \ZZ\sp{n} \to \RR \cup \{ +\infty \}$ and
assume that the convex envelope $\overline{f}$
is a polyhedral convex function
and 
\begin{equation} \label{fnextAssume}
 \overline{f}(x) = f(x)  \qquad (x \in \ZZ^{n}).
\end{equation}
Then 
$f$ is integrally convex if and only if
$\argmin f[-p]$ is an integrally convex set
for each $p \in \RR\sp{n}$
for which $f[-p]$ attains a (finite) minimum.
\end{theorem}

\begin{proof}
The only-if-part is immediate from Proposition~\ref{PRargminICset},
since if $f$ is integrally convex, so is $f[-p]$ for any $p$.
We prove the if-part by using Theorem~\ref{THicchardmc}.
Take any $x,y \in \dom f$ with $\| x - y \|_{\infty} \geq 2$,
and let $u:=(x+y)/2$.
Our goal is to show 
\begin{equation}  \label{prfICcharExt1}
\tilde{f} (u) 
\leq \frac{1}{2} (f(x) + f(y))
\end{equation}
in \eqref{intcnvmidpt},
where $\tilde{f}$ is the local convex extension of $f$.
The midpoint $u$ 
belongs to the convex hull
$\overline{\dom f}$ 
of $\dom f$, 
which implies that 
$\overline{f}(u)$ is finite and
$u \in \argmin \overline{f}[-p]$ 
for some $p$.
Let $S_{p} := \argmin f[-p]$
for this $p$, where $S_{p} \subseteq \ZZ\sp{n}$.
Since
$\argmin \overline{f}[-p] = \overline{\argmin f[-p]} = \overline{S_{p}}$,
we have
$u \in \overline{S_{p}}$.
By the assumed integral convexity of $S_{p}$, 
this implies  $u \in  \overline{S_{p} \cap N(u)}$
(see \eqref{icsetdef1}).
Therefore, there exist 
$z\sp{(1)}, z\sp{(2)}, \ldots, z\sp{(m)} \in S_{p} \cap N(u)$
as well as 
positive numbers
$\lambda_{1}, \lambda_{2}, \ldots, \lambda_{m}$
with
$\sum_{i=1}\sp{m} \lambda_{i} = 1$
such that
\begin{equation} \label{prfICcharExt2}
u = \sum_{i=1}\sp{m} \lambda_{i} z\sp{(i)},
\qquad
 \tilde f(u) =  
\sum_{i=1}\sp{m} \lambda_{i} f(z\sp{(i)}).
\end{equation}
Since each $z\sp{(i)}$ $(\in S_{p})$
is a minimizer of $f[-p]$,
we have
\[
\sum_{i=1}\sp{m} \lambda_{i} f[-p](z\sp{(i)})
\leq \frac{1}{2} (f[-p](x) +  f[-p](y)),
\]
that is,
\begin{equation} \label{prfICcharExt3}
\sum_{i=1}\sp{m} \lambda_{i} f(z\sp{(i)})
-  \sum_{i=1}\sp{m} \lambda_{i} \langle p, z\sp{(i)} \rangle 
\leq \frac{1}{2} (f(x) +  f(y))
 - \frac{1}{2} (\langle p , x \rangle + \langle p , y \rangle ).
\end{equation}
For the linear parts in this expression we have
\begin{align*}
& \sum_{i=1}\sp{m} \lambda_{i} \langle p, z\sp{(i)} \rangle 
= \langle p, \sum_{i=1}\sp{m} \lambda_{i}  z\sp{(i)} \rangle 
 = \langle p , u \rangle,
\\ &
 \frac{1}{2} (\langle p , x \rangle + \langle p , y \rangle )
=  \langle p , \frac{1}{2}( x + y ) \rangle 
 = \langle p , u \rangle,
\end{align*}
and therefore, \eqref{prfICcharExt3} is equivalent to
\[
\sum_{i=1}\sp{m} \lambda_{i} f(z\sp{(i)})
\leq \frac{1}{2} (f(x) +  f(y)).
\]
Combining this with the expression of $\tilde f(u)$
in \eqref{prfICcharExt2}, we obtain \eqref{prfICcharExt1}.
This completes the proof of Theorem~\ref{THfnargmincharExt}.
\qedJIAM
\end{proof}

The next theorem gives a similar characterization of 
integrally convex functions under a different assumption.

\begin{theorem}[{\cite[Theorem~3.29]{Mdcasiam}}]
  \label{THfnargmincharBnd}
Let $f: \ZZ\sp{n} \to \RR \cup \{ +\infty \}$ be a function
with a bounded nonempty effective domain.
Then 
$f$ is integrally convex if and only if
$\argmin f[-p]$ is an integrally convex set
for each $p \in \RR\sp{n}$.
\end{theorem}
\begin{proof}
The only-if-part is immediate from Proposition~\ref{PRargminICset},
since if $f$ is integrally convex, so is $f[-p]$ for any $p$.
To prove the if-part by Theorem~\ref{THfnargmincharExt},
define $S_{p} := \argmin f[-p]$,
which is integrally convex by assumption.
Note that $S_{p}$ is
nonempty for every $p$
by the assumed boundedness of $\dom f$.
The boundedness of $\dom f$
also implies that
the convex envelope $\overline{f}$
is a polyhedral convex function.
The integral convexity of $S_{p}$ implies
that
$S_{p}$ is hole-free ($S_{p} = \overline{S_{p}} \cap \ZZ\sp{n}$),
from which we obtain
$\overline{f}(x) = f(x)$ $(x \in \ZZ^{n})$
in \eqref{fnextAssume}.
Then the integral convexity of $f$ follows from Theorem~\ref{THfnargmincharExt}.
\qedJIAM
\end{proof}

\begin{remark} \rm \label{RMicfnArgminp}
The characterization of integral convexity of $f$ by $\argmin f[-p]$
originates in \cite[Theorem~3.29]{Mdcasiam},
which is stated as Theorem~\ref{THfnargmincharBnd} above.
In this paper, we have given an alternative proof to this theorem
by first establishing Theorem~\ref{THfnargmincharExt}
that employs the assumption of convex-extensibility.
See also Fig.~\ref{FGthmchar} in Section~\ref{SCcharICfn}.
\finbox
\end{remark}

\begin{remark} \rm  \label{RMargmincharbdddom}
Theorems \ref{THfnargmincharExt} and \ref{THfnargmincharBnd} impose the assumption of
convex-extensibility
of $f$ or boundedness of $\dom f$. 
Such an assumption seems inevitable.
Consider a function $f: \ZZ \to \RR$ defined by
$ 
f(x)=
   \left\{  \begin{array}{ll}
    0 & ( x = 0 ) , \\
    1 & ( x \not= 0 ) .
            \end{array}  \right.
$
Then $\argmin f[-p]$ is equal to $\{ 0 \}$ or the empty set
for each $p \in \RR$.
However, this function is not integrally convex.
Note that $f$ is not convex-extensible nor $\dom f$ is bounded. 
\finbox
\end{remark}

\subsection{Proximity theorems}
\label{SCprox}

The proximity-scaling approach is a fundamental technique
in designing efficient algorithms for discrete or combinatorial optimization.
For a function
$f: \mathbb{Z}^{n} \to \mathbb{R} \cup \{ +\infty  \}$
in integer variables and a positive integer 
$\alpha \in \mathbb{Z}_{++}$, called 
a scaling unit, the $\alpha$-scaling of $f$ means the function 
$f^{\alpha}$
defined by $f^{\alpha}(x) = f(\alpha x) $ $(x \in \mathbb{Z}^{n})$
(cf., \eqref{scalefndef}).
A proximity theorem is a result guaranteeing that a (local) minimum 
of the scaled function $f^{\alpha}$
is (geometrically) close to a minimizer 
of the original function $f$. 
More precisely,
we say that $x^{\alpha} \in \dom f$ is an {\em $\alpha$-local minimizer} of $f$
(or {\em $\alpha$-local minimal} for $f$)
if $f(x^{\alpha}) \leq f(x^{\alpha}+ \alpha d)$
for all $d \in  \{ -1,0, +1 \}^{n}$,
and a proximity theorem gives
a bound
$B(n,\alpha)$ such that
for any $\alpha$-local minimizer $x^{\alpha}$ of $f$,
there exists a minimizer $x^{*}$ of $f$ satisfying 
$ \| x^{\alpha} - x^{*}\|_{\infty} \leq B(n,\alpha)$.
The scaled function $f^{\alpha}$ is expected to be simpler
and hence easier to minimize,
whereas the quality of the obtained minimizer of $f^{\alpha}$ 
as an approximation to the minimizer of $f$
is guaranteed by a proximity theorem.
The proximity-scaling approach consists in applying this idea
for a decreasing sequence of $\alpha$, often by halving the scale unit $\alpha$.

In discrete convex analysis
the following proximity theorems are known for 
\Lnat-convex and \Mnat-convex functions.

\begin{theorem} [\protect{\cite{IS02}}; {\cite[Theorem~7.18]{Mdcasiam}}]  
\label{THlfnproximity} 
Suppose that $f$ is an \Lnat-convex function,
$\alpha \in \mathbb{Z}_{++}$, 
and $x^{\alpha} \in \dom f$.
If $f(x^{\alpha}) \leq f(x^{\alpha} +  \alpha d)$ for all 
$d \in  \{ 0, 1 \}^{n} \cup \{ 0, -1 \}^{n}$,
then there exists a minimizer 
$x^{*}$ of $f$ satisfying  $\| x^{\alpha} - x^{*}  \|_{\infty} \leq n (\alpha -1)$.
\finboxARX
\end{theorem}

\begin{theorem}[\protect{\cite{MMS02}}; \protect{\cite[Theorem~6.37]{Mdcasiam}}]\label{THmfnproximity}
Suppose that $f$ is an \Mnat-convex function,
$\alpha \in \mathbb{Z}_{++}$, 
and $x^{\alpha} \in \dom f$.
If $f(x^{\alpha}) \leq f(x^{\alpha} + \alpha d)$ for all 
$d \in \{ \unitvec{i}, -\unitvec{i}  \  (1 \leq i \leq n) \}
   \cup \{ \unitvec{i} - \unitvec{j} \  (i \not= j) \} $,
then there exists a minimizer 
$x^{*}$ of $f$ satisfying $\| x^{\alpha} - x^{*}  \|_{\infty} \leq n (\alpha -1)$.
\finboxARX
\end{theorem}

The proximity bounds in Theorems 
\ref{THlfnproximity} and \ref{THmfnproximity}
are known to be tight 
\cite[Examples 4.2 and 4.3]{MMTT19proxIC}.
It is noteworthy that
we have the same proximity bound
$B(n,\alpha) = n (\alpha -1)$
for \Lnat-convex and \Mnat-convex functions,
and that $B(n,\alpha)$ is linear in $n$.
For integrally convex functions with $n \geq 3$, 
this bound $n (\alpha -1)$ is no longer valid,
which is demonstrated by 
Examples 4.4 and 4.5 in \cite{MMTT19proxIC}.
More specifically, the latter example shows 
a quadratic lower bound 
$B(n,\alpha) \geq (n-2)^{2}(\alpha -1)/4$
for an integrally convex function
arising from bipartite graphs.

The following is a proximity theorem for integrally convex functions.

\begin{theorem}
[{\cite[Theorem~5.1]{MMTT19proxIC}}]
  \label{THproxintcnv}
Let $f: \mathbb{Z}^{n} \to \mathbb{R} \cup \{ +\infty  \}$
be an integrally convex function,
$\alpha \in \mathbb{Z}_{++}$, 
and $x^{\alpha} \in \dom f$.
If $f(x^{\alpha}) \leq f(x^{\alpha}+ \alpha d)$ for all $d \in  \{ -1,0, +1 \}^{n}$,
then $\argmin f \not= \emptyset$ and
there exists $x^{*} \in \argmin f$ with
\begin{equation}\label{icproximity0}
 \| x^{\alpha} - x^{*}\|_{\infty} \leq \beta_{n} (\alpha - 1) ,
\end{equation}
where $\beta_{n}$ is defined by 
\begin{equation}\label{eqMrec1}
\beta_{1}=1, \quad \beta_{2}=2; \qquad  
 \beta_{n} =  \frac{n+1}{2} \beta_{n-1}  + 1
\quad (n=3,4,\ldots).
\end{equation}
The proximity bound $\beta_{n}$ satisfies
\begin{equation}\label{eqMrec2est}
 \beta_{n} \leq \frac{(n+1)!}{2^{n-1}} 
\qquad (n=3,4,\ldots).
\end{equation}
\finboxARX
\end{theorem}

The bound
$\beta_{n} (\alpha - 1)$ in \eqref{icproximity0}
is superexponential in $n$, as seen from \eqref{eqMrec2est}.
The numerical values of $\beta_{n}$ are as follows:
\begin{equation}\label{betatable}
\begin{array}{l|cccccc}
\hline
\mbox{Dimension $n$}
  &   2 & 3 & 4 & 5 & 6 & 7
\\ \hline
\mbox{Value \eqref{eqMrec1} of $\beta_{n}$} 
  & \phantom{0}2\phantom{0} & \phantom{0}5\phantom{0}  & 13.5 & 41.5 & 146.25 & 586
\\
 \mbox{Bound \eqref{eqMrec2est} on $\beta_{n}$} 
  & - & 6 & 15\phantom{.0} & 45\phantom{.0}  &  157.5\phantom{0}    & 630
\\ \hline
\end{array}
\end{equation}

While the proof of
Theorem~\ref{THproxintcnv}
for general $n$ is quite long
\cite{MMTT19proxIC},
its special case for $n=2$ 
admits an alternative method
based on the 
box-barrier property
(Theorem~\ref{THintcnvbox})
and the parallelogram inequality 
(Remark~\ref{RMparaineqdim2}).
See the proof of \cite[Theorem~4.1]{MMTT19proxIC}
for the detail.

Finally we mention that proximity theorems 
are also available for 
\LLnat-convex and \MMnat-convex functions \cite{MT04prox}.
The proximity bound for \LLnat-convex functions is linear in $n$
\cite[Theorem~6]{MT04prox}
and that for \MMnat-convex functions is quadratic in $n$
\cite[Theorem~10]{MT04prox}.

\subsection{Scaling algorithm}
\label{SCalgo}

In spite of the fact that 
integral convexity is not preserved under variable-scaling,
it is possible to design a scaling algorithm
for minimizing an integrally convex function
with a bounded effective domain.

In the following we briefly describe the algorithm of \cite{MMTT19proxIC}.
The algorithm is justified by
the proximity bound in Theorem~\ref{THproxintcnv}
and the optimality criterion in Theorem~\ref{THintcnvlocopt}.
Let $K_{\infty}$ denote the $\ell_{\infty}$-size of the effective domain of $f$, 
i.e.,
\[
K_{\infty} := \max\{ \| x -y \|_{\infty} \mid x, y \in \dom f \}.
\] 
Initially, the scaling unit $\alpha$ is set to 
$2\sp{\lceil \log_{2} K_{\infty} \rceil} \approx K_{\infty}$.
In Step~S1 of the algorithm,
the function $\hat f(y)$ is 
the restriction of $f(x + \alpha y)$ to the set
\[
 Y:=\{ y \in \mathbb{Z}\sp{n} \mid 
 \| \alpha y \|_{\infty} \leq \beta_{n} (2\alpha - 1) \},
\]
which is a box of integers.
Then a local minimizer $y\sp{*}$ of $\hat f(y)$ is found to update $x$ to $x+ \alpha y\sp{*}$. 
A local minimizer of $\hat f(y)$ can be found, 
e.g., by any descent method (the steepest descent method, in particular),
although $\hat f(y)$ is not necessarily integrally convex.

\begin{tabbing}     
\= {\bf Scaling algorithm for integrally convex functions}%
\\
\> \quad  S0: 
   \= Find an initial vector $x$ with $f(x) < +\infty$, and set 
   $\alpha := 2\sp{\lceil \log_{2} K_{\infty} \rceil}$.
\\
\> \quad  S1:
   \>  Find an integer vector $y\sp{*}$  that locally minimizes  
\\ \> \> \quad 
      $\hat f(y) = 
   \left\{ \begin{array}{ll}  
  f(x + \alpha y) 
    & (\| \alpha y \|_{\infty} \leq \beta_{n} (2\alpha - 1) ) , \\
   +\infty & (\mbox{otherwise}) , \\
 \end{array}\right.
$
\\ \> \>
    in the sense of   
$\hat f(y\sp{*}) \leq \hat f(y\sp{*} +  d)$ ($\forall d \in  \{ -1, 0, +1 \}^{n}$)
\\ \> \>
   (e.g., by the steepest descent method),  
   and set $x:= x+ \alpha y\sp{*}$.  \\
\> \quad  S2: 
    \> If $\alpha = 1$, then stop \
       ($x$ is a minimizer of $f$).             
\\
\> \quad  S3: 
  \> Set  $\alpha:=\alpha/2$, and go to S1.  
\end{tabbing}

\begin{tabbing}     
\= {\bf The steepest descent method to locally minimize $\hat f(y)$}%
\\
\> \quad  D0: 
   \= Set $z := \bm{0}$.
\\
\> \quad  D1:
   \>  Find $d \in  \{ -1, 0, +1 \}^{n}$ that minimizes $\hat f(z +  d)$.
\\
\> \quad  D2: 
    \> If $\hat f(z) \leq \hat f(z +  d)$, then set $y\sp{*}:=z$ and stop
\\ 
\> \> 
  ($y\sp{*}$ is a local minimizer of $\hat f$).             
\\
\> \quad  D3: 
  \> Set  $z:= z+d$, and go to D1.  
\end{tabbing}

In the final phase with $\alpha = 1$, $\hat f$
is an integrally convex function,
and hence, by Theorem~\ref{THintcnvlocopt},
the local minimizer in Step~S1
is a global minimizer of $\hat f$.
Furthermore, 
it can be shown, with the use of Theorem~\ref{THproxintcnv},
that this point is a global minimizer of $f$.

The complexity of the algorithm is as follows.
The number of iterations in the descent method
is bounded by 
$|Y| \leq (4\beta_{n})\sp{n}$.
For each $z$, the $3\sp{n}$ neighboring points are examined
to find a descent direction or verify its local minimality.
Thus Step~S1
can be done with at most $(12\beta_{n})\sp{n}$ function evaluations.
The number of scaling phases is $\log_{2} K_{\infty}$.
Therefore, 
the total number of function evaluations in the algorithm
is bounded by
$(12\beta_{n})\sp{n}\log_{2} K_{\infty}$.
For a fixed $n$, this gives a polynomial bound 
$O(\log_{2} K_{\infty})$ in the problem size.
It is emphasized in \cite[Remark~6.2]{MMTT19proxIC} that
the linear dependence of
$B(n,\alpha)= \beta_{n} (\alpha - 1)$
on $\alpha$ is critical for the 
complexity $O(\log_{2} K_{\infty})$.

Finally, we mention that no algorithm 
can minimize every integrally convex function
in time polynomial in $n$, since any function  
on the unit cube $\{ 0,1 \}\sp{n}$ is integrally convex.



\section{Subgradient and biconjugacy} 
\label{SCsubgrbiconj}

\subsection{Subgradient} 
\label{SCsubgr}

In convex analysis \cite{BL06,HL01,Roc70},
the {\em subdifferential} of 
a convex  function
$g: \RR\sp{n} \to \RR \cup \{ +\infty \}$
at $x \in \dom g$
is defined by
\begin{equation} \label{subgRRdef}
 \subgR g(x)
= \{ p \in  \RR\sp{n} \mid    
  g(y) - g(x)  \geq  \langle p, y - x \rangle   \ \ \mbox{\rm for all }  y \in \RR\sp{n} \},
\end{equation}
and an element $p$ of $\subgR g(x)$ is called a 
{\em subgradient}
of $g$ at $x$.
Analogously,
for a function
$f: \ZZ\sp{n} \to \RR \cup \{ +\infty \}$,
the subdifferential of 
$f$ at $x \in \dom f$
is defined as 
\begin{equation} \label{subgZRdef}
 \subgR f(x)
= \{ p \in  \RR\sp{n} \mid    
  f(y) - f(x)  \geq  \langle p, y - x \rangle   \ \ \mbox{\rm for all }  y \in \ZZ\sp{n} \},
\end{equation}
and an element $p$ of $\subgR f(x)$ is called a subgradient of $f$ at $x$.
An alternative expression 
\begin{equation} \label{subgZRargmin}
 \subgR f(x)
= \{ p \in  \RR\sp{n} \mid  x \in \argmin f[-p]  \}
\end{equation}
is often convenient, where
$f[-p](x) = f(x) - \langle p , x \rangle$  
defined in \eqref{notatfp}.
If $f$ is convex-extensible in the sense of
$f = \overline{f}\,|_{\ZZ\sp{n}}$ in \eqref{fnconvextcl},
where
$\overline{f}$ is the convex envelope of $f$ defined in \eqref{fnconvenvClSupAff},
then $\subgR f(x) = \subgR \overline{f}(x)$
for each $x \in \ZZ^{n}$.

When
$f: \ZZ\sp{n} \to \RR \cup \{ +\infty \}$
is integrally convex, $\subgR f(x)$ is nonempty for $x \in \dom f$,
since 
$f = \overline{f}\,|_{\ZZ\sp{n}}$ 
by \eqref{ICfnextZ}.
Furthermore, we can rewrite \eqref{subgZRdef}
by making use of Theorem~\ref{THintcnvlocopt}
for the minimality of an integrally convex function.
Namely, by Theorem~\ref{THintcnvlocopt} applied to $f[-p]$,
we may restrict $y$ in \eqref{subgZRdef}
to the form of $y=x + d$ with $d \in \{-1,0,+1\}^n$,
to obtain
\begin{equation} \label{subgIC}
 \subgR f(x)
= \{ p \in  \RR\sp{n} \mid    
  \sum_{j=1}\sp{n} d_{j} p_{j} \leq f(x+d) - f(x) 
 \ \ \mbox{for all} \ \  d \in \{ -1,0,+1 \}\sp{n} \} .
\end{equation}
This expression shows, in particular, that
$\subgR f(x)$ is a polyhedron described by inequalities with
coefficients taken from $\{ -1, 0, +1 \}$.

To discuss an integrality property of $\subgR f(x)$
in Section \ref{SCsubgrint},
it is useful 
to investigate the projections of $\subgR f(x)$ along coordinate axes.
Let $P := \subgR f(x)$ for notational simplicity,
and for each $l = 1,2,\ldots,n$,
let $[P]_{l}$ denote
the projection of $P$
to the space of $(p_{l},p_{l+1},\ldots,p_n)$.
Inequality systems to describe the projections
$[P]_{l}$ for $l = 1,2,\ldots,n$
can be obtained 
by applying the Fourier--Motzkin elimination procedure \cite{MS22kolin2,Sch86} to 
the system of inequalities 
in \eqref{subgIC},
where
the variable $p_{1}$ is eliminated first, and then $p_{2}, p_{3}, \ldots$,
to finally obtain an inequality in $p_{n}$ only.
By virtue of the integral convexity of $f$, a drastic simplification occurs
in this elimination process. 
The inequalities that are generated  
in the elimination process
are actually redundant and need not be added
to the current system of inequalities,
which is a crucial observation made in \cite{MT20subgrIC}.
Thus we obtain the following theorem.

\begin{theorem}[\cite{MT20subgrIC}] \label{THelimICproj}
Let 
$f: \ZZ\sp{n} \to \RR \cup \{ +\infty \}$
be an integrally convex function and $x \in \dom f$.
Then 
$\subgR f(x)$ is a nonempty polyhedron,
and for each $l = 1,2,\ldots,n$,
the projection 
$[\subgR f(x)]_{l}$
of $\subgR f(x)$
to the space of $(p_{l},p_{l+1},\ldots,p_n)$
is given by
\begin{align} 
[\subgR f(x)]_{l}  = 
& \  \{ (p_{l},p_{l+1},\ldots,p_n) 
\mid   
  \sum_{j=l}\sp{n} d_{j} p_{j} \leq f(x+d) - f(x) 
\nonumber \\ &
 \ \ \mbox{\rm for all} \ \  d \in \{ -1,0,+1 \}\sp{n}
 \ \ \mbox{\rm with} \ \  d_{j} = 0 \ (1 \leq j \leq l -1)
 \} .
\label{subgICproj}
\end{align}
\finboxARX
\end{theorem}

While the Fourier--Motzkin elimination 
for the proof of Theorem~\ref{THelimICproj} 
depends on the linear ordering of
$N = \{ 1,2,\ldots, n  \}$,
it is possible to formulate
the obtained identity 
\eqref{subgICproj}
without referring to the ordering of $N$.
This is stated below as a corollary.

\begin{corollary} \label{COsubrICprojJ}
Let $J$ be any nonempty subset of $N = \{ 1,2,\ldots, n  \}$.
Under the same assumption as in Theorem~\ref{THelimICproj},
the projection 
of $\subgR f(x)$
to the space of $(p_{j} \mid j \in J)$
is given by
$
\{ (p_{j} \mid j \in J) 
\mid   
  \sum_{j \in J} d_{j} p_{j} \leq f(x+d) - f(x) 
  \ \mbox{\rm for all}  \  d \in \{ -1,0,+1 \}\sp{n}
  \ \mbox{\rm with} \   d_{j} = 0 \ (j \in N \setminus J) \}$.
\finboxARX
\end{corollary}

\subsection{Integral subgradient} 
\label{SCsubgrint}

For an integer-valued function
$f: \ZZ\sp{n} \to \ZZ \cup \{ +\infty \}$,
we are naturally interested in integral vectors in $\subgR f(x)$. 
An integer vector $p$ belonging to $\subgR f(x)$ 
is called an {\em integral subgradient}
and the condition 
\begin{equation} \label{ICsubgrZ}
\subgR f(x) \cap \ZZ\sp{n} \neq \emptyset
\end{equation}
is sometimes referred to as
the {\em integral subdifferentiability}  of $f$ at $x$.

Integral subdifferentiability
of integer-valued integrally convex functions
is established recently by
the present authors \cite{MT20subgrIC}.

\begin{theorem}[{\cite[Theorem~3]{MT20subgrIC}}]  \label{THsubgrIC}
Let $f: \ZZ\sp{n} \to \ZZ \cup \{ +\infty \}$
be an integer-valued integrally convex function.
For every $x \in \dom f$,
we have
$\subgR f(x) \cap \ZZ\sp{n} \neq \emptyset$.
\end{theorem}
\begin{proof}
The proof is based on 
Theorem~\ref{THelimICproj} concerning projections of $\subgR f(x)$.
While the reader is referred to 
\cite{MT20subgrIC} for the formal proof,
we indicate the basic idea here.
By \eqref{subgICproj} for $l=n$, we have
\[
 \{ p_{n} \mid p \in \subgR f(x) \}  
= \{ p_{n} \mid 
   p_{n} \leq  f(x+ \unitvec{n} ) - f(x), \
   -  p_{n} \leq  f(x- \unitvec{n} ) - f(x)   
 \},
\]
which is an interval 
$[ \alpha_{n}, \beta_{n} ]_{\RR}$
with 
$\alpha_{n} = f(x) - f(x- \unitvec{n} )$
and
$\beta_{n}=  f(x+ \unitvec{n} ) - f(x)$.
We have $\alpha_{n} \leq \beta_{n}$ since $\subgR f(x) \ne \emptyset$,
while $\alpha_{n}, \beta_{n} \in \ZZ$ since $f$ is integer-valued.
Therefore, 
the interval
$[ \alpha_{n}, \beta_{n} ]_{\RR}$
contains an integer, say, $p_{n}\sp{*}$.
Next,
by \eqref{subgICproj} for $l=n-1$, the set
$ \{ p_{n-1} \mid p \in \subgR f(x), p_{n}= p_{n}\sp{*} \}$  
is described by six inequalities
\[
\sigma p_{n-1} + \tau p_{n}\sp{*} \leq  f(x + \sigma \unitvec{n-1} + \tau \unitvec{n} ) - f(x)
\quad (\sigma \in \{ +1, -1 \}, \ \tau \in \{ +1, -1, 0 \}).
\]
This implies that
$\{ p_{n-1} \mid p \in \subgR f(x), p_{n}= p_{n}\sp{*} \}$
is a nonempty interval, say,
$[ \alpha_{n-1}, \beta_{n-1} ]_{\RR}$
with $\alpha_{n-1}, \beta_{n-1} \in \ZZ$,
which contains an integer, say, $p_{n-1}\sp{*}$. 
This means that 
there exists $p \in \subgR f(x)$ 
such that
$p_{n}= p_{n}\sp{*} \in \ZZ$ and
$p_{n-1}= p_{n-1}\sp{*} \in \ZZ$. 
Continuing in this way
(with $l=n-2, n-3, \ldots, 2,1$), 
we can construct $p\sp{*} \in \subgR f(x) \cap \ZZ\sp{n}$. 
\qedJIAM
\end{proof}

Some supplementary facts concerning Theorem~\ref{THsubgrIC} are shown below.

\begin{itemize}
\item
Theorem~\ref{THsubgrIC} states that 
$\subgR f(x) \cap \ZZ\sp{n} \neq \emptyset$,
but it does not claim a stronger statement that $\subgR f(x)$ is an integer polyhedron.
Indeed, $\subgR f(x)$ is not necessarily an integer polyhedron.
For example, let 
$f: \mathbb{Z}^3 \to \mathbb{Z} \cup \{+\infty\}$ be defined by
$f(0,0,0)=0$ and $f(1,1,0)=f(0,1,1)=f(1,0,1)=1$
with $\dom f = \{ (0,0,0), (1,1,0), (0,1,1), (1,0,1) \}$.
This $f$ is integrally convex and 
$\subgR f(\veczero) = \{ p \in \RR\sp{3} \mid 
 p_{1} + p_{2} \leq 1,  
 p_{2} + p_{3} \leq 1,  
 p_{1} + p_{3} \leq 1   \}$
is not an integer polyhedron, having a non-integral vertex at $p=(1/2, 1/2, 1/2)$.
See \cite[Remark~4]{MT20subgrIC} for details.
In the special cases where
$f$ is 
\Lnat-convex, 
\Mnat-convex,
\LLnat-convex, or
\MMnat-convex,
the subdifferential
$\subgR f(x)$ is known \cite{Mdcasiam} to be an integer polyhedron.

\item
If $\subgR f(x)$ is bounded, 
$\subgR f(x)$ has an integral vertex,
although not every vertex of $\subgR f(x)$ is integral.
For example, let 
$ D = \{ x \in \{ -1,0,+1 \}\sp{3}  \mid  |x_{1}| + |x_{2}| + |x_{3}| \leq 2 \}$
and define $f$ with $\dom f = D$ by
$f(\veczero) = 0$ and
$f(x) = 1$   $(x \in D \setminus \{ \veczero \})$.
This $f$ is an integer-valued integrally convex function
and $\subgR f(\veczero)$ is a bounded polyhedron 
that has eight non-integral vertices
$(\pm 1/2, \pm 1/2, \pm 1/2)$
(with arbitrary combinations of double-signs)
and six integral vertices
$(\pm 1, 0, 0)$,
$(0, \pm 1, 0)$, and
$(0,0, \pm 1)$.
See \cite[Remark~7]{MT20subgrIC} 
for details.

\item
Integral subdifferentiability is not guaranteed
without the assumption of integral convexity.
Consider
$D = \{ (0,0,0), \pm (1,1,0), \pm (0,1,1), \pm (1,0,1) \}$
and define $f: \mathbb{Z}^3 \to \mathbb{Z} \cup \{+\infty\}$  
with $\dom f = D$ by
$f(x_{1},x_{2},x_{3}) =  (x_{1}+x_{2}+x_{3})/2 $.
This function is not integrally convex,
$\subgR f(\veczero) = \{ (1/2, 1/2, 1/2) \}$,
and 
$\subgR f(\veczero) \cap \ZZ\sp{3} = \emptyset$.
See \cite[Example 1.1]{Mdca98} or 
\cite[Example 1]{MT20subgrIC} 
for details.
\end{itemize}

The integral subdifferentiability 
formulated in Theorem~\ref{THsubgrIC}
can be strengthened with an additional box condition.
This stronger form plays the key role in 
the proof of the Fenchel-type min-max duality theorem 
(Theorem~\ref{THfencICsep}) discussed in Section \ref{SCfenc}.

Recall that an integral box means a set $B$ of real vectors represented as
$B = \{ p \in \RR\sp{n} \mid \alpha \leq p \leq \beta \}$
for integer vectors 
$\alpha \in (\ZZ \cup \{ -\infty \})\sp{n}$ and 
$\beta \in (\ZZ \cup \{ +\infty \})\sp{n}$ satisfying $\alpha \leq \beta$.
The following theorem states that
\begin{equation} \label{ICsubgrboxZ00}
\subgR f(x) \cap B \neq \emptyset
\ \Longrightarrow \ 
\subgR f(x) \cap B \cap \ZZ\sp{n} \neq \emptyset ,
\end{equation}
which may be referred to as 
{\em box-integral subdifferentiability}.

\begin{theorem}[{\cite[Theorem~1.2]{MT22ICfenc}}] \label{THsubgrICbox}
Let $f: \ZZ\sp{n} \to \ZZ \cup \{ +\infty \}$
be an integer-valued integrally convex function, $x \in \dom f$,
and $B$ be an integral box.
If $\subgR f(x) \cap B$ is nonempty,
then $\subgR f(x) \cap B$ 
is a polyhedron containing an integer vector.
If, in addition, $\subgR f(x) \cap B$ is bounded,
then $\subgR f(x) \cap B$ has an integral vertex.
\finboxARX
\end{theorem}

We briefly describe how Theorem~\ref{THsubgrICbox} has been proved in \cite{MT22ICfenc}.
Recall that Theorem~\ref{THsubgrIC} for integral subdifferentiability
(without a box)
is proved from a hierarchical system of inequalities
(Theorem~\ref{THelimICproj})
to describe 
the projection $[\subgR f(x)]_{l}$  
of $\subgR f(x)$ 
to the space of $(p_{l},p_{l+1},\ldots,p_n)$
for $l = 1,2,\ldots,n$,
where Theorem~\ref{THelimICproj} itself
is proved by means of the Fourier--Motzkin elimination.
This approach is extended 
in \cite{MT22ICfenc}
to prove Theorem~\ref{THsubgrICbox}.
Namely, 
Theorem~4.3 of \cite{MT22ICfenc} gives 
a hierarchical system of inequalities 
to describe
the projection 
$[\subgR f(x) \cap B]_{l}$
of $\subgR f(x) \cap B$
to the space of $(p_{l},p_{l+1},\ldots,p_n)$
for $l = 1,2,\ldots,n$.
The proof of this theorem is  based on 
the Fourier--Motzkin elimination.
Once inequalities for the projections 
$[\subgR f(x) \cap B]_{l}$ are obtained,
the derivation of box-integral subdifferentiability
(Theorem~\ref{THsubgrICbox})
is almost the same as that of integral subdifferentiability 
(Theorem~\ref{THsubgrIC})
from Theorem~\ref{THelimICproj}.
Finally we mention that alternative proofs of 
Theorems \ref{THsubgrIC} and \ref{THsubgrICbox}
can be found in
\cite{Fuj19rairo} and \cite{Fuj21ic}, respectively.

\subsection{Biconjugacy} 
\label{SCbiconj}

For an integer-valued function 
$f: \ZZ\sp{n} \to \ZZ \cup \{ +\infty \}$
with $\dom f \ne \emptyset$,
we define
$f\sp{\bullet}: \ZZ\sp{n} \to \ZZ \cup \{ +\infty \}$ by
\begin{equation}
f\sp{\bullet}(p)  = \sup\{  \langle p, x \rangle - f(x)   \mid x \in \ZZ\sp{n} \}
\qquad ( p \in \ZZ\sp{n}),
 \label{conjvexZpZ} 
\end{equation}
called the {\em integral conjugate} of $f$.
For any $x, p \in \ZZ\sp{n}$ we have
\begin{equation} \label{conjYoung}
f(x) + f\sp{\bullet}(p) \geq \langle p, x \rangle  ,
\end{equation}
which is a discrete analogue of Fenchel's inequality \cite[(1.1.3), p.~211]{HL01} or
the Fenchel--Young inequality \cite[Proposition~3.3.4]{BL06}.
When  $x \in \dom f$, we have
\begin{equation} \label{conjYoungEq}
f(x) + f\sp{\bullet}(p) = \langle p, x \rangle  
\iff 
p \in  \subgR f(x) \cap \ZZ\sp{n}.
\end{equation}
Note the asymmetric roles of $f$ and $f\sp{\bullet}$ in \eqref{conjYoungEq}.

The integral conjugate $f\sp{\bullet}$ of an integer-valued function $f$
is also an integer-valued function defined on $\ZZ\sp{n}$.
So we can apply the transformation \eqref{conjvexZpZ} 
to $f\sp{\bullet}$  
to obtain
$f\sp{\bullet\bullet} = (f\sp{\bullet})\sp{\bullet}$,
which is called the {\em integral biconjugate} of $f$.
It follows from 
\eqref{conjYoung} and \eqref{conjYoungEq} 
that,
for each $x \in \dom f$ we have
\begin{equation} \label{subgbiconj}
f\sp{\bullet\bullet}(x) = f(x)  
\iff 
\subgR f(x) \cap \ZZ\sp{n} \ne \emptyset.
\end{equation}
See \cite[Lemma 4.1]{Mdca98} or \cite[Lemma 1]{MT20subgrIC}
for the proof.
We say that $f$ enjoys
{\em integral biconjugacy} if
\begin{equation} \label{biconjdef}
f\sp{\bullet\bullet}(x) = f(x) \quad \mbox{for all $x \in \ZZ\sp{n}$}.    
\end{equation}

\begin{example} \rm  \label{EXsubdiffZempty}
Let $D = \{ (0,0,0), \pm (1,1,0), \pm (0,1,1), \pm (1,0,1) \}$
and consider the function
$f(x_{1},x_{2},x_{3}) =  (x_{1}+x_{2}+x_{3})/2 $
on $\dom f = D$.
(This is the function used in Section \ref{SCsubgrint}
as an example of an integer-valued function
lacking in integral subdifferentiability.) \ 
According to the definition 
\eqref{conjvexZpZ}, the integral conjugate 
$f\sp{\bullet}$ is given by 
\begin{equation} \label{fconjsubdiffZempty}
 f\sp{\bullet}(p) = 
\max \{  |p_{1} + p_{2} -1|,  |p_{2} + p_{3} -1|, |p_{1} + p_{3} -1| \}
\qquad
(p \in \ZZ\sp{3}) .
\end{equation}
For $x=\veczero$ we have $f(x)= 0$,
while
\eqref{fconjsubdiffZempty} shows
$f\sp{\bullet}(p) \geq 1$
for every integer vector $p \in \ZZ\sp{3}$.
Therefore we have strict inequality 
$f(x) + f\sp{\bullet}(p) > \langle p, x \rangle$
for $x=\veczero$ and every $p \in \ZZ\sp{3}$.
This is consistent with
\eqref{conjYoungEq}  
since 
$\subgR f(\veczero) = \{ (1/2, 1/2, 1/2) \}$
and hence 
$\subgR f(\veczero) \cap \ZZ\sp{3} = \emptyset$.
For the integral biconjugate 
$f^{\bullet\bullet}(x) 
= \sup \{ \langle p, x \rangle - f^{\bullet}(p) \mid p \in \ZZ^3 \}$
we have 
\[
f^{\bullet\bullet}(\veczero) 
= - \inf_{p \in \ZZ^3} \max \{ |p_{1}+p_{2}-1|, |p_{2}+p_{3}-1|, |p_{3}+p_{1}-1| \} 
= -1.
\]
Therefore we have $f^{\bullet\bullet}(\veczero) \neq f(\veczero)$.
This is consistent with \eqref{subgbiconj}
since $\subgR f(\veczero) \cap \ZZ\sp{3} = \emptyset$.
\finbox 
\end{example}

We now assume that 
$f: \ZZ\sp{n} \to \ZZ \cup \{ +\infty \}$
is an integer-valued integrally convex function.
The integral conjugate $f\sp{\bullet}$ 
is not necessarily integrally convex
(\cite[Example 4.15]{MS01rel}, \cite[Remark~2.3]{MT20subgrIC}).
Nevertheless,  
the integral biconjugate $f\sp{\bullet\bullet}$
coincides with $f$ itself, 
that is,
integral biconjugacy holds for an integer-valued integrally convex function.
This theorem is established by 
the present authors \cite{MT20subgrIC}
based on the integral subdifferentiability
$\subgR f(x) \cap \ZZ\sp{n} \ne \emptyset$
given in Theorem~\ref{THsubgrIC};
see Remark~\ref{RMbiconjdomcond} below for some technical aspects.

\begin{theorem}[{\cite[Theorem~4]{MT20subgrIC}}]  \label{THbiconjIC}
For any integer-valued integrally convex function 
$f: \ZZ^{n} \to \ZZ \cup \{ +\infty \}$
with $\dom f \ne \emptyset$,
we have
$f\sp{\bullet\bullet}(x) =f(x)$ for all $x \in \ZZ\sp{n}$.
\finboxARX
\end{theorem}

As special cases of Theorem~\ref{THbiconjIC}
we obtain integral biconjugacy 
for 
{\rm L}-convex, \Lnat-convex,
{\rm M}-convex,  \Mnat-convex,
\LLnat-convex, and \MMnat-convex functions
given in \cite[Theorems 8.12, 8.36, 8.46]{Mdcasiam},
and that for BS-convex and UJ-convex functions 
given in \cite[Corollary~2]{MT20subgrIC}.

\begin{remark} \rm  \label{RMbiconjdomcond}
There is a subtle gap between 
integral subdifferentiability
in \eqref{subgbiconj}
and integral biconjugacy in \eqref{biconjdef}
for a general integer-valued function 
$f: \ZZ\sp{n} \to \ZZ \cup \{ +\infty \}$
(which is not necessarily integrally convex).
While the latter imposes the condition 
$f\sp{\bullet\bullet}(x) = f(x)$ on all $x \in \ZZ\sp{n}$,
the former refers to $x$ in $\dom f$ only.
This means that integral subdifferentiability
may possibly be weaker than integral biconjugacy,
and this is indeed the case in general 
(see \cite[Remark~4.1]{Mdca98} or \cite[Remark~6]{MT20subgrIC}).
However, it is known \cite[Lemma 4.2]{Mdca98} that 
integral subdifferentiability does imply
integral biconjugacy
under the technical conditions that
$\dom f = \mathrm{cl}(\overline{\dom f}) \cap \mathbb{Z}^{n}$
and
$\mathrm{cl}(\overline{\dom f})$ is rationally-polyhedral,
where $\mathrm{cl}(\overline{\dom f})$ denotes 
the closure of the convex hull (or closed convex hull)
of $\dom f$; see Remark~\ref{RMconvhull} for this notation.
\finbox
\end{remark}

\begin{remark} \rm  \label{RMrealbiconj}
In convex analysis \cite{BL06,HL01,Roc70},
the {\em conjugate function} of 
$g: \RR\sp{n} \to \RR \cup \{ +\infty \}$ 
with $\dom g \ne \emptyset$
is defined to be a function
$g\sp{\bullet \RR}: \RR\sp{n} \to \RR \cup \{ +\infty \}$ 
given by
\begin{equation}
g\sp{\bullet \RR}(p)  := 
\sup\{  \langle p, x \rangle - g(x)   \mid x \in \RR\sp{n} \}
\qquad ( p \in \RR\sp{n}),
 \label{conjvexRpR0} 
\end{equation}
where a (non-standard) notation 
$g\sp{\bullet \RR}$ 
is introduced for discussion here.
The {\em biconjugate} of $g$ is defined as 
$(g\sp{\bullet \RR})\sp{\bullet \RR}$
by using the transformation \eqref{conjvexRpR0} twice.
We have biconjugacy 
$(g\sp{\bullet \RR})\sp{\bullet \RR} = g$
for closed convex functions $g$.
For a real-valued function 
$f: \ZZ\sp{n} \to \RR \cup \{ +\infty \}$
in discrete variables, we may also define 
\begin{equation}
f\sp{\bullet \RR}(p)  := \sup\{  \langle p, x \rangle - f(x)   \mid x \in \ZZ\sp{n} \}
\qquad ( p \in \RR\sp{n}).
 \label{conjvexZpR0} 
\end{equation}
Then the convex envelope $\overline{f}$ coincides with
$(f\sp{\bullet \RR})\sp{\bullet \RR}$,
which denotes the function 
obtained by applying
\eqref{conjvexZpR0} 
to $f$ and then \eqref{conjvexRpR0}
to $g = f\sp{\bullet \RR}$.
Therefore, if $f$ is convex-extensible in the sense of
$f = \overline{f}\,|_{\ZZ\sp{n}}$ in \eqref{fnconvextcl},
we have
$(f\sp{\bullet \RR})\sp{\bullet \RR} \, |_{\ZZ^{n}} = f$,
which is a kind of biconjugacy.
If $f$ is integer-valued,  we can naturally consider 
$(f\sp{\bullet})\sp{\bullet}$
using \eqref{conjvexZpZ} twice
as well as $(f\sp{\bullet \RR})\sp{\bullet \RR}$
using \eqref{conjvexZpR0} and then \eqref{conjvexRpR0}. 
It is most important to recognize 
that for any $f$ we have
$(f\sp{\bullet})\sp{\bullet}(x)\leq (f\sp{\bullet \RR})\sp{\bullet \RR}(x)$
for $x \in \ZZ^{n}$ and that the equality may fail even when 
$f = \overline{f}\, |_{\ZZ^{n}}$.
As an example, consider the function
$f(x_{1},x_{2},x_{3}) =  (x_{1}+x_{2}+x_{3})/2 $
in Example \ref{EXsubdiffZempty}.
The convex envelope $\overline{f}$ is
given by
$\overline{f}(x_{1},x_{2},x_{3}) =  (x_{1}+x_{2}+x_{3})/2 $
on the convex hull of $D$,
and therefore
$f = \overline{f}\, |_{\ZZ^{n}}$
holds.
Similarly to \eqref{fconjsubdiffZempty} we have
\begin{equation*} 
 f\sp{\bullet \RR}(p) = 
\max \{  |p_{1} + p_{2} -1|,  |p_{2} + p_{3} -1|, |p_{1} + p_{3} -1| \}
\quad
(p \in \RR\sp{3}) 
\end{equation*}
and hence
\[
(f\sp{\bullet \RR})\sp{\bullet \RR}(\veczero) 
= - \inf_{p \in \RR^3} \max \{ |p_{1}+p_{2}-1|, |p_{2}+p_{3}-1|, |p_{3}+p_{1}-1| \} 
= 0,
\]
where the infimum over $p \in \RR^3$ is attained by 
$p = (1/2, 1/2, 1/2)$.
Therefore
$(f\sp{\bullet \RR})\sp{\bullet \RR}(\veczero) =  f(\veczero)$,
whereas 
$
(f\sp{\bullet})\sp{\bullet}(\veczero) < f(\veczero)$
as we have seen in Example \ref{EXsubdiffZempty}.
Thus, integral biconjugacy 
$f = (f\sp{\bullet})\sp{\bullet}$
in \eqref{biconjdef}
is much more intricate than the equality 
$f = (f\sp{\bullet \RR})\sp{\bullet \RR}\, |_{\ZZ^{n}}$.
\finbox
\end{remark}


\subsection{Discrete DC programming}
\label{SCdcprog}

A discrete analogue of 
the theory of DC functions (difference of two convex functions), 
or discrete DC programming, 
has been proposed in \cite{MM15dcprog}
using \Lnat-convex and \Mnat-convex functions.
As already noted in \cite[Remark~4.7]{MM15dcprog},
such  theory of discrete DC functions can be developed for functions
that satisfy integral biconjugacy and integral subdifferentiability.
It is pointed out in \cite{MT20subgrIC} that 
Theorems \ref{THsubgrIC} and \ref{THbiconjIC}
for integrally convex functions
enable us to extend the theory of discrete DC functions 
to integrally convex functions.
In particular,
an analogue of 
the Toland--Singer duality~\cite{Sin79,Tol79}
can be established for integrally convex functions
as follows.

\begin{theorem}
[{\cite[Theorem~5]{MT20subgrIC}}] \label{THtolandsinger}
Let 
$g, h: \mathbb{Z}^{n} \to \mathbb{Z} \cup \{+\infty\}$
be integer-valued integrally convex functions.
Then
\begin{equation} \label{tolandsingerduality}
\inf \{ g(x) - h(x) \mid x \in \mathbb{Z}^{n} \} 
= \inf\{ h^{\bullet}(p) - g^{\bullet}(p) \mid p \in \mathbb{Z}^{n}  \} . 
\end{equation}
\finboxARX
\end{theorem}

As mentioned already in \cite{MT20subgrIC},
the assumption of integral convexity of $g$ is not needed
for \eqref{tolandsingerduality} to be true.
That is, \eqref{tolandsingerduality} holds for 
any  $g: \mathbb{Z}^{n} \to \mathbb{Z} \cup \{+\infty\}$
as long as $h: \mathbb{Z}^{n} \to \mathbb{Z} \cup \{+\infty\}$ is integrally convex.



\section{Discrete Fenchel duality}
\label{SCfenc}

\subsection{General framework of Fenchel duality}
\label{SCfencheldual}

The Fenchel duality is one of the expressions of the duality principle 
in the form of a min-max relation 
between a pair of convex and concave functions
$(f,g)$
and their conjugate functions.
As is well known, the existence of such min-max formula
guarantees the existence of a certificate of optimality
for the problem of minimizing 
$f - g$ over $\RR\sp{n}$ or $\ZZ\sp{n}$.

First we recall the framework for functions in continuous variables.
For
$f: \RR\sp{n} \to \RR \cup \{ +\infty \}$
with $\dom f \ne \emptyset$,
the function
$f\sp{\bullet}: \RR\sp{n} \to \RR \cup \{ +\infty \}$
defined by
\begin{equation}
f\sp{\bullet}(p)  := 
\sup\{  \langle p, x \rangle - f(x)   \mid x \in \RR\sp{n} \}
\qquad ( p \in \RR\sp{n})
 \label{conjvexRpR} 
\end{equation}
is called the 
{\em conjugate} 
(or {\em convex conjugate})
of $f$.
For 
$g: \RR\sp{n} \to \RR \cup \{ -\infty \}$
with $\dom g \ne \emptyset$,
the function
$g\sp{\circ}: \RR\sp{n} \to \RR \cup \{ -\infty \}$ 
defined by
\begin{equation}
g\sp{\circ}(p)  := \inf\{  \langle p, x \rangle - g(x)   \mid x \in \RR\sp{n} \}
\qquad ( p \in \RR\sp{n})
 \label{conjcavRpR} 
\end{equation}
is called the 
{\em concave conjugate} of $g$.
We have $g\sp{\circ}(p) = -f\sp{\bullet}(-p)$ if $g(x) = -f(x)$.
It follows from the definitions that
\begin{equation} \label{weakdualGenRR} 
  f(x) - g(x) \geq  g\sp{\circ}(p) - f\sp{\bullet}(p)   
\end{equation}
for any $x \in \RR\sp{n}$ and $p \in \RR\sp{n}$.
The relation 
\eqref{weakdualGenRR} is called {\em weak duality}.

The Fenchel duality theorem
says that
a min-max formula
\begin{equation}
  \inf\{ f(x) - g(x) \mid x \in \RR\sp{n}  \}
 = 
  \sup\{ g\sp{\circ}(p) - f\sp{\bullet}(p)
  \mid   p \in \RR\sp{n} \}
 \label{minmaxGenRR} 
\end{equation}
holds for convex and concave functions
$f: \RR\sp{n} \to \RR \cup \{ +\infty \}$
and 
$g: \RR\sp{n} \to \RR \cup \{ -\infty \}$
satisfying certain regularity conditions.
The relation \eqref{minmaxGenRR}
is called {\em strong duality} in contrast to weak duality.
See 
Bauschke--Combettes \cite[Section~15.2]{BC11},
Borwein--Lewis \cite[Theorem~3.3.5]{BL06},
Hiriart-Urruty--Lemar{\'e}chal \cite[(2.3.2), p.~228]{HL01},
Rockafellar \cite[Theorem~31.1]{Roc70}, 
Stoer--Witzgall \cite[Corollary 5.1.4]{SW70}
for precise statements.

\

We now turn to functions in discrete variables.
For any functions
$f: \ZZ\sp{n} \to \RR \cup \{ +\infty \}$
and  
$g: \ZZ\sp{n} \to \RR \cup \{ -\infty \}$
we define
\begin{align}
f\sp{\bullet}(p)  &= \sup\{  \langle p, x \rangle - f(x)   \mid x \in \ZZ\sp{n} \}
\qquad ( p \in \RR\sp{n}),
 \label{conjvexZpR} 
\\
g\sp{\circ}(p)  &= \inf\{  \langle p, x \rangle - g(x)   \mid x \in \ZZ\sp{n} \}
\qquad ( p \in \RR\sp{n}),
 \label{conjcavZpR} 
\end{align}
where $\dom f \ne \emptyset$ and $\dom g \ne \emptyset$ are assumed.
In this case, the generic form of the Fenchel duality reads: 
\begin{equation} 
 \inf \{ f(x) - g(x) \mid  x \in \ZZ\sp{n}  \} 
= \sup \{ g\sp{\circ}(p) - f\sp{\bullet}(p) \mid  p \in \RR\sp{n} \} ,
\label{minmaxGenZR3} 
\end{equation} 
which is expected to be true when $f$ and $g$
are equipped with certain discrete convexity and concavity, respectively.
Moreover, when $f$ and $g$ are integer-valued
($f: \ZZ\sp{n} \to \ZZ \cup \{ +\infty \}$
and  
$g: \ZZ\sp{n} \to \ZZ \cup \{ -\infty \}$),
we are particularly interested in the dual problem with an integer vector, that is,
\begin{equation} 
 \inf \{ f(x) - g(x) \mid  x \in \ZZ\sp{n}  \} 
= \sup \{ g\sp{\circ}(p) - f\sp{\bullet}(p) \mid  p \in \ZZ\sp{n} \} .
\label{minmaxGenZZ3} 
\end{equation}

\

To relate the discrete case to the continuous case,
it is convenient to consider the convex envelope $\overline{f}$ of $f$
and the concave envelope 
$\overline{g}$ of $g$,
where $\overline{g} := -\overline{(-g)}$,
that is,
$\overline{g}$ is defined to be the negative of 
the convex envelope of $-g$.
By the definitions of
convex and concave envelopes
and conjugate functions we have
\begin{align*} 
& \overline{f}(x) \leq f(x), 
\quad \overline{g}(x) \geq g(x)
\qquad (x \in \ZZ\sp{n}),
\\ &
\overline{f}\sp{\bullet}(p)=f\sp{\bullet}(p),
\quad \overline{g}\sp{\circ}(p)=g\sp{\circ}(p)
\qquad (p \in \RR\sp{n})
\end{align*} 
as well as weak dualities
\begin{align} 
 f(x) - g(x) 
 \geq g\sp{\circ}(p) - f\sp{\bullet}(p)   
\qquad (x \in \ZZ\sp{n}, p \in \RR\sp{n}),
\label{weakdualGenZR} 
\\
 \overline{f}(x) - \overline{g}(x) 
 \geq \overline{g}\sp{\circ}(p) - \overline{f}\sp{\bullet}(p)   
\qquad (x \in \RR\sp{n}, p \in \RR\sp{n}).
\label{weakdualEnvRR} 
\end{align} 
Thus we have the following chain of inequalities:
\begin{equation}
\mbox{${\mathrm{P}}(\ZZ)$} 
\geq 
\mbox{$\overline{\mathrm{P}}(\RR)$}
\geq
\mbox{$\overline{\mathrm{D}}(\RR)$}
= 
\mbox{$\mathrm{D}(\RR)$}
\geq 
\mbox{$\mathrm{D}(\ZZ)$}
 \label{fencweakZR2} 
\end{equation}
with notations
\begin{align*}  \label{}
& \mathrm{P}(\ZZ) :=  \inf\{ f(x) - g(x) \mid x \in \ZZ\sp{n}  \} ,
\\ & 
\overline{\mathrm{P}}(\RR) :=
\inf\{ \overline{f}(x) - \overline{g}(x) \mid x \in \RR\sp{n}  \},
\\ &
\overline{\mathrm{D}}(\RR) := 
  \sup\{ \overline{g}\sp{\circ}(p)  - \overline{f}\sp{\bullet}(p)
        \mid  p \in \RR\sp{n} \}, 
\\ &
\mathrm{D}(\RR) := 
 \sup\{ g\sp{\circ}(p) - f\sp{\bullet}(p) \mid  p \in \RR\sp{n} \}  ,
\\ &
\mathrm{D}(\ZZ) :=  \sup\{ g\sp{\circ}(p) - f\sp{\bullet}(p) \mid  p \in \ZZ\sp{n} \}  
\end{align*}
for the optimal values of the problems,
where $\mathrm{P}(\cdot)$ stands for ``Primal problem'' 
and $\mathrm{D}(\cdot)$ for ``Dual problem''.

The desired min-max relations
\eqref{minmaxGenZR3} and \eqref{minmaxGenZZ3} 
can be written as 
$\mathrm{P}(\ZZ) = \mathrm{D}(\RR)$
and
$\mathrm{P}(\ZZ) = \mathrm{D}(\ZZ)$, respectively.
The inequality 
between $\overline{\mathrm{P}}(\RR)$
and $\overline{\mathrm{D}}(\RR)$
in the middle of \eqref{fencweakZR2}
becomes an equality if the Fenchel duality 
\eqref{minmaxGenRR} holds for 
$(\overline{f}, \overline{g})$.
The first and the last inequality in \eqref{fencweakZR2}
 express possible discrepancy 
between discrete and continuous cases,
and we are naturally concerned with 
when these inequalities turn into equalities.
The concept of subdifferential plays the essential role here.

For the subdifferential $\subgR f(x)$
defined in \eqref{subgZRdef}
we observe that
\begin{equation} \label{conjsubgZRvex}
p \in \subgR f(x) \iff f(x) + f\sp{\bullet}(p)  = \langle p, x \rangle
\end{equation}
holds for any $p \in \RR\sp{n}$ and $x \in \dom f$.
Similarly, we have
\begin{equation} \label{conjsubgZRcav}
p \in \subgR' g(x) \iff g(x) + g\sp{\circ}(p)  = \langle p, x \rangle
\end{equation}
for any $p \in \RR\sp{n}$ and $x \in \dom g$,
where 
$\subgR' g(x)$ means the concave version of the subdifferential defined
by
\begin{equation} \label{subgcavZRdef}
\subgR' g(x) = -(\subgR (-g))(x)
= \{ p \in  \RR\sp{n} \mid    
  g(y) - g(x)  \leq  \langle p, y - x \rangle   \ \ \mbox{\rm for all }  y \in \ZZ\sp{n} \}.
\end{equation}
Suppose that the infimum $\mathrm{P}(\ZZ)$ is attained 
by some $x\sp{*} \in \ZZ\sp{n}$.
It follows from
\eqref{weakdualGenZR}, \eqref{conjsubgZRvex}, and \eqref{conjsubgZRcav}
that 
$\mathrm{P}(\ZZ) = \mathrm{D}(\RR)$ and
the supremum $\mathrm{D}(\RR)$ is attained by $p\sp{*} \in \RR\sp{n}$ 
if and only if 
$p\sp{*} \in \subgR f(x\sp{*}) \cap \subgR' g(x\sp{*})$.
Therefore, 
$\mathrm{P}(\ZZ) = \mathrm{D}(\RR)$ if 
\begin{equation} \label{comsubgrR}
\subgR f(x\sp{*}) \cap \subgR' g(x\sp{*})  \ne \emptyset .
\end{equation}
Furthermore, if 
\begin{equation} \label{comsubgrZ}
\subgR f(x\sp{*}) \cap \subgR' g(x\sp{*}) \cap \ZZ\sp{n} \ne \emptyset,
\end{equation}
then we have
$\mathrm{P}(\ZZ) = \mathrm{D}(\ZZ)$.
Thus \eqref{comsubgrR} and \eqref{comsubgrZ}, respectively, 
imply the Fenchel-type min-max formulas
\eqref{minmaxGenZR3} for real-valued functions and 
\eqref{minmaxGenZZ3} for integer-valued functions. 
It is noted that this implication does not presuppose 
the Fenchel duality $\overline{\mathrm{P}}(\RR) = \overline{\mathrm{D}}(\RR)$
for $(\overline{f}, \overline{g})$.

When $\overline{\mathrm{P}}(\RR) = \overline{\mathrm{D}}(\RR)$ is known to hold,
the min-max relation
$\mathrm{P}(\ZZ) = \mathrm{D}(\RR)$
for real-valued functions
follows, by \eqref{fencweakZR2}, from
$\mathrm{P}(\ZZ) = \overline{\mathrm{P}}(\RR)$,
where the latter condition  $\mathrm{P}(\ZZ) = \overline{\mathrm{P}}(\RR)$
holds if
\begin{equation} \label{evfgevfevg}
\overline{f - g} = \overline{f}- \overline{g} .
\end{equation}
It is emphasized that \eqref{evfgevfevg}
does not follow from the individual convex or concave-extensibility 
of $f$ and $g$.
See Example~\ref{EXnorealfenc} in Section~\ref{SCfencIC}.

\subsection{Fenchel duality for integrally convex functions}
\label{SCfencIC}

The following two examples show that 
the min-max formula 
\eqref{minmaxGenZR3} or \eqref{minmaxGenZZ3}
is not necessarily true 
when $f$ and $g$ are integrally convex and concave functions.
(Naturally, function $g$ is called {\em integrally concave} 
if $-g$ is integrally convex.)

\begin{example}[{\cite[Example 5.6]{Mbonn09}}]  \rm \label{EXnorealfenc}
Let $f, g: \ZZ\sp{2} \to \ZZ$ be defined as
\[
f(x_{1},x_{2}) = |x_{1}+x_{2}-1|,
\qquad
g(x_{1},x_{2}) = 1- |x_{1}-x_{2}|.
\]
The function $f$ is integrally convex (actually \Mnat-convex)
 and $g$ is integrally concave (actually \Lnat-concave).
We have
\begin{equation*}
  \begin{array}{ccccccccc}
 \displaystyle 
\min_{\ZZ}\{ f - g \}
 & \! > \!  & 
 \displaystyle 
\min_{\RR}\{ \overline{f} - \overline{g} \}
 & \! = \!  &   
 \displaystyle 
\max_{\RR}\{ \overline{g}\sp{\circ}
         - \overline{f}\sp{\bullet}  \}
 & \! = \!  &  
 \displaystyle 
 \max_{\RR}\{ {g}\sp{\circ}
         - {f}\sp{\bullet}  \}
 & \! = \! & 
 \displaystyle 
 \max_{\ZZ}\{ g\sp{\circ} - f\sp{\bullet} \}  .
\\
 (0) &  & (-1)  & & (-1)  && (-1) && (-1) 
  \end{array} 
\end{equation*}
Thus the min-max identity 
\eqref{minmaxGenZZ3} 
as well as
\eqref{minmaxGenZR3}
fails because of the primal integrality gap
$\mathrm{P}(\ZZ) > \overline{\mathrm{P}}(\RR)$.
Indeed, the condition 
$\overline{f - g} = \overline{f}- \overline{g}$
in \eqref{evfgevfevg} 
fails as follows.
Let $h := f - g$.
Since
$h(0,0) = h(1,0) = h(0,1) = h(1,1) = 0$, 
we have
$\overline{h}(1/2,1/2) = 0$,
whereas 
$\overline{f}(1/2,1/2) - \overline{g}(1/2,1/2) = 0 -1 = -1$.
Thus 
$\overline{f - g} \ne \overline{f}- \overline{g}$.
\finbox
\end{example}

\begin{example}[{\cite[Example 5.7]{Mbonn09}}]  \rm \label{EXnointfenc}
Let $f, g: \ZZ\sp{2} \to \ZZ$ be defined as
\[
f(x_{1},x_{2}) = \max(0,x_{1}+x_{2}),
\qquad
g(x_{1},x_{2}) = \min(x_{1},x_{2}).
\]
The function $f$ is integrally convex (actually \Mnat-convex)
 and $g$ is integrally concave (actually \Lnat-concave).
We have
\begin{equation*}
  \begin{array}{ccccccccc}
  \displaystyle \min_{\ZZ}\{ f - g \}
 & \! = \! & 
 \displaystyle \min_{\RR}\{ \overline{f} - \overline{g} \}
 & \! =\!  &  \displaystyle  \max_{\RR}\{ \overline{g}\sp{\circ}
         - \overline{f}\sp{\bullet}  \}
 & \! =\!  &  \displaystyle  \max_{\RR}\{ {g}\sp{\circ}
         - {f}\sp{\bullet}  \}
 & \! > \! &   \displaystyle \max_{\ZZ}\{ g\sp{\circ} - f\sp{\bullet} \} .
\\
 (0) &  &  (0) & & (0) && (0) && (-\infty) 
  \end{array} 
\end{equation*}
Although the min-max identity 
\eqref{minmaxGenZR3}
with real-valued $p$ holds,
the formula
\eqref{minmaxGenZZ3} 
with integer-valued $p$ 
fails because of the dual integrality gap
$\mathrm{D}(\RR) > \mathrm{D}(\ZZ)$.
The optimal value 
$\min_{\ZZ}\{ f - g \} = 0$
is attained by $x\sp{*} = \veczero$,
at which the subdifferentials are given by
$\subgR f(\veczero) = \{ (p_{1}, p_{2}) \mid 0 \leq p_{1}=p_{2} \leq 1 \}$
and
$\subgR' g(\veczero) = \{ (p_{1}, p_{2}) \mid  p_{1}+p_{2}= 1 , 0 \leq p_{1} \leq 1 \}$.
We have
$\subgR f(\veczero) \cap \subgR' g(\veczero) = \{ (1/2, 1/2)\}$
and 
$\subgR f(\veczero) \cap \subgR' g(\veczero) \cap \ZZ\sp{2} = \emptyset$,
which shows the failure of \eqref{comsubgrZ}.
\finbox
\end{example}

As the min-max formula 
\eqref{minmaxGenZR3} is not true
(in general)
when $f$ and $-g$ are integrally convex,
we are motivated to restrict $g$
to a subclass of integrally concave functions,
while allowing $f$ to be a general integrally convex function.
However, 
the possibility of 
$g$ being \Mnat-concave or \Lnat-concave
is denied by the above examples.
That is, we cannot hope for the combination of 
(integrally convex, \Mnat-convex)
nor
(integrally convex, \Lnat-convex)
for $(f,-g)$.
Furthermore, since a function in two variables
is \Mnat-convex if and only if it is multimodular \cite[Remark~2.2]{MM19multm},
the possibility of the combination of
(integrally convex, multimodular)
for $(f,-g)$
is also denied.
Thus we are motivated to consider the combination of
(integrally convex, separable convex).

In the following we address the
Fenchel-type min-max formula
for a pair of an integrally convex function  
and a separable concave function.
A function
$\Psi: \ZZ^{n} \to \RR \cup \{ -\infty \}$
in $x=(x_{1}, x_{2}, \ldots,x_{n}) \in \ZZ^{n}$
is called  
{\em separable concave}
if it can be represented as
\begin{equation}  \label{sepcavdef}
\Psi(x) = \psi_{1}(x_{1}) + \psi_{2}(x_{2}) + \cdots + \psi_{n}(x_{n})
\end{equation}
with univariate discrete concave functions
$\psi_{i}: \ZZ \to \RR \cup \{ -\infty \}$,
which means, by definition, that 
$\dom \psi_{i}$ is an interval of integers and
\begin{equation}  \label{univarcavedef}
\psi_{i}(k-1) + \psi_{i}(k+1) \leq 2 \psi_{i}(k)
\qquad (k \in \ZZ).
\end{equation}
The concave conjugate of $\Psi$ is denoted by $\Psi\sp{\circ}$, that is,
\begin{align}
\Psi\sp{\circ}(p) & = 
\inf \{ \langle p, x \rangle - \Psi(x) \mid x \in \ZZ\sp{n} \}
\qquad ( p\in \RR\sp{n}).
 \label{conjcavPsiZpR} 
\end{align}
This is a separable concave function represented as
\begin{equation} \label{Psiconj00}
\Psi\sp{\circ}(p)= 
 \psi_{1}\sp{\circ}(p_{1}) + \psi_{2}\sp{\circ}(p_{2}) + \cdots + \psi_{n}\sp{\circ}(p_{n}) ,
\end{equation}
where
\begin{equation} \label{psiconjdef00}
\psi_{i}\sp{\circ}(l)  = 
 \inf \{ k l  -  \psi_{i}(k)  \mid  k \in \ZZ \}
\qquad
(l \in \RR).
\end{equation}
It is often possible 
to obtain an explicit form of the (integral) conjugate function 
of an (integer-valued) separable convex (or concave) function;
see \cite{FM19partII,FM21boxTDI}.

The following is the Fenchel-type min-max formula
for a pair of an integrally convex function  
and a separable concave function.
The case of integer-valued functions, which is more interesting,
is given in \cite[Theorem~1.1]{MT22ICfenc},
while we include here the case of real-valued functions for completeness.

\begin{theorem} \label{THfencICsep}
Let $f: \ZZ\sp{n} \to \RR \cup \{ +\infty \}$
be an integrally convex function
and $\Psi: \ZZ\sp{n} \to \RR \cup \{ -\infty \}$
a separable concave function.
Assume that
$\dom f \cap \dom \Psi \neq \emptyset$ and 
$\inf \{ f(x) - \Psi(x) \mid  x \in \ZZ\sp{n}  \}$
is attained by some $x\sp{*}$.
Then 
\begin{equation} 
 \inf \{ f(x) - \Psi(x) \mid  x \in \ZZ\sp{n}  \} 
= \sup \{ \Psi\sp{\circ}(p) - f\sp{\bullet}(p) \mid  p \in \RR\sp{n} \}
\label{minmaxICsepR} 
\end{equation} 
and the supremum in \eqref{minmaxICsepR} is attained 
by some $p\sp{*} \in \RR\sp{n}$.
If, in addition, $f$ and $\Psi$ are integer-valued, then
\begin{equation} 
 \inf \{ f(x) - \Psi(x) \mid  x \in \ZZ\sp{n}  \} 
= \sup \{ \Psi\sp{\circ}(p) - f\sp{\bullet}(p) \mid  p \in \ZZ\sp{n} \}
\label{minmaxICsepZ} 
\end{equation}
and the supremum in \eqref{minmaxICsepZ} is attained 
by some $p\sp{*} \in \ZZ\sp{n}$.
\end{theorem}

\begin{proof}
The proof of the real-valued case \eqref{minmaxICsepR} 
consists in showing 
\eqref{comsubgrR}
for $(f,g)= (f,\Psi)$,
which is proved in Section \ref{SCfencRproof}.
The integral case \eqref{minmaxICsepZ} has been proved 
in \cite[Theorem~1.1]{MT22ICfenc}
by showing \eqref{comsubgrZ} for $(f,g)= (f,\Psi)$
as a consequence of box-integral subdifferentiability
described in Theorem~\ref{THsubgrICbox}.
\qedJIAM
\end{proof}

\begin{remark} \rm \label{RMminmaxAssmp1}
In the integer-valued case, 
the existence of $x\sp{*}$
attaining the infimum in
\eqref{minmaxICsepZ}
is guaranteed
if (and only if)
the set $\{ f(x) - \Psi(x) \mid  x \in \ZZ\sp{n}  \}$
of function values is bounded from below.
It is known \cite[Lemma~3.2]{MT22ICfenc} that 
if the supremum on the right-hand side 
of \eqref{minmaxICsepZ} is finite, then 
the infimum on the left-hand side is also finite.
\finbox
\end{remark}

Theorem~\ref{THfencICsep} implies
a min-max theorem for separable convex minimization
on a box-integer polyhedron. 
The case of integer-valued functions, which is more interesting,
is stated in \cite[Theorem~3.1]{MT22ICfenc}.
We define notation
$\mu_{P}(p) = \inf \{ \langle p, x \rangle \mid  x\in P \}$
for a polyhedron $P$.

\begin{theorem}\label{THboxint}
Let $P \ (\subseteq \RR\sp{n})$
be a nonempty box-integer polyhedron,
and $\Phi: \ZZ\sp{n} \to \RR \cup \{ +\infty \}$
a separable convex function.
Assume that
$\inf \{ \Phi (x) \mid  x \in P \cap \ZZ\sp{n} \}$ 
is (finite and) attained by some $x\sp{*}$.
Then 
\begin{equation} 
 \inf \{ \Phi (x) \mid  x \in P \cap \ZZ\sp{n} \} 
= \sup \{ \mu_{P}(p) - \Phi\sp{\bullet}(p) \mid  p\in \RR\sp{n}\} 
\label{minmaxBoxintSepR} 
\end{equation} 
and the supremum is attained by some $p\sp{*} \in \RR\sp{n}$.
If, in addition, 
$\Phi$ is integer-valued, then
\begin{equation} 
 \inf \{ \Phi (x) \mid  x \in P \cap \ZZ\sp{n} \} 
= \sup \{ \mu_{P}(p) - \Phi\sp{\bullet}(p) \mid  p\in \ZZ\sp{n}\} 
\label{minmaxBoxintSepZ} 
\end{equation} 
and the supremum is attained by some $p\sp{*} \in \ZZ\sp{n}$.
\end{theorem}
\begin{proof}
Denote the indicator function of $P \cap \ZZ\sp{n}$ by $\delta$,
which is an integer-valued integrally convex function
because $P$ is a box-integer polyhedron.
Then the statements follow from Theorem~\ref{THfencICsep} for $f=\delta$ and $\Psi = -\Phi$.
\qedJIAM
\end{proof}

This theorem generalizes a recent result of Frank--Murota \cite[Theorem~3.4]{FM21boxTDI},
which asserts the min-max formula \eqref{minmaxBoxintSepZ}
for integer-valued $\Phi$
when $P$ is an integral box-TDI polyhedron.

\subsection{Proof of \eqref{minmaxICsepR} for real-valued functions (Theorem \ref{THfencICsep})}
\label{SCfencRproof}

In this section we prove the min-max formula \eqref{minmaxICsepR} 
for real-valued functions in Theorem~\ref{THfencICsep}.
Let $x\sp{*}$ denote an element of 
$\dom f \cap \dom \Psi \ (\subseteq \ZZ\sp{n})$ 
that attains the infimum in \eqref{minmaxICsepR}.
According to the general framework described in Section \ref{SCfencheldual},
it suffices to show
$\subgR f(x\sp{*}) \cap \subgR' \Psi(x\sp{*})  \ne \emptyset$.

Since
$\overline{f - \Psi} = \overline{f} - \overline{\Psi}$
by Proposition~2.1 of \cite{MT22ICfenc},
we have
\[
\inf_{ x \in \ZZ\sp{n} } \{ f(x) - \Psi(x)  \} 
=\inf_{ x \in \RR\sp{n} } \{ (\overline{f - \Psi})(x)  \} 
=\inf_{ x \in \RR\sp{n} } \{ \overline{f}(x) - \overline{\Psi}(x)  \} ,
\]
which implies that 
$x\sp{*}$ is also a minimizer of
$\overline{f} - \overline{\Psi}$
over $\RR\sp{n}$.
Define
$\alpha :=  \inf \{ f(x) - \Psi(x) \mid  x \in \ZZ\sp{n}  \} = f(x\sp{*}) - \Psi(x\sp{*})$
and $f_{1}(x) := f(x) - \alpha$ for $x \in \ZZ\sp{n}$.
Then $f_{1}$ is an integrally convex function
satisfying
$\overline{f_{1}} \geq \overline{\Psi}$
and 
$f_{1}(x\sp{*}) = \Psi(x\sp{*})$.
The desired nonemptiness 
$\subgR f(x\sp{*}) \cap \subgR' \Psi(x\sp{*})  \ne \emptyset$
follows from Proposition~\ref{PRcomsubgrIC} below
applied to $(f,g)=(f_{1},\Psi)$.

\begin{proposition}  \label{PRcomsubgrIC}
Let
$f: \ZZ\sp{n} \to \RR \cup \{ +\infty \}$
and 
$g: \ZZ\sp{n} \to \RR \cup \{ -\infty \}$
be integrally convex and concave functions, respectively,
and
$x\sp{*} \in  \dom f \cap \dom g$. 
If 
$\overline{f} \geq \overline{g}$
on $\RR\sp{n}$ 
and 
$f(x\sp{*}) = g(x\sp{*})$,
then 
$\subgR f(x\sp{*}) \cap \subgR' g(x\sp{*})  \ne \emptyset$.
\end{proposition}
\begin{proof}
We may assume that
$x\sp{*} = \veczero$
and
$f(\veczero) = g(\veczero) = 0$.
By \eqref{subgIC} we have
\begin{align} 
 \subgR f(\veczero)
& = \{ p \in  \RR\sp{n} \mid    
  \langle d\sp{(i)}, p \rangle \leq f(d\sp{(i)}) 
 \ \ \mbox{for all $i \in I$} \} ,
\label{comsubgICprf1}
\\
 \subgR' g(\veczero)
& = \{ p \in  \RR\sp{n} \mid    
  \langle -\hat d\sp{(j)}, p \rangle \leq -g(\hat d\sp{(j)}) 
 \ \ \mbox{for all $j \in J$} \} 
\label{comsubgICprf2}
\end{align}
where
$\{ -1,0,+1 \}\sp{n} \cap \dom f$ is represented as
$\{ d\sp{(i)} \mid i \in I \}$
and
$\{ -1,0,+1 \}\sp{n} \cap \dom g$ as
$\{ \hat d\sp{(j)} \mid j \in J \}$. 
By the Farkas lemma (or linear programming duality)
\cite{Sch86},
there exists $p \in \subgR f(\veczero) \cap \subgR' g(\veczero)$
if and only if
\begin{equation} \label{comsubgICprf3}
\sum_{i \in I} u_{i} f(d\sp{(i)}) - \sum_{j \in J} v_{j} g(\hat d\sp{(j)}) \geq 0
\end{equation}
for any $u_{i} \geq 0$ and $v_{j} \geq 0$
satisfying 
$\sum_{i \in I} u_{i} d\sp{(i)} - \sum_{j \in J} v_{j} \hat d\sp{(j)} = 0$.
Let 
$\hat x := \sum_{i \in I} u_{i} d\sp{(i)} = \sum_{j \in J} v_{j} \hat d\sp{(j)}$,
$U := \sum_{i \in I} u_{i}$, and $V := \sum_{j \in J} v_{j}$.
By homogeneity we may assume $U + V \leq 1$.
Then $\hat x \in \dom \overline{f} \cap \dom \overline{g}$.
If $U > 0$, it follows from the convexity of $\overline{f}$
as well as  $\overline{f}(\veczero)=0$ and $U \leq 1$ that
\[
 \sum_{i \in I} u_{i} f(d\sp{(i)}) 
= U \sum_{i \in I} \frac{u_{i}}{U} \overline{f}(d\sp{(i)}) 
\geq 
 U  \overline{f}( \sum_{i \in I} \frac{u_{i}}{U} d\sp{(i)}) 
= U \overline{f}(\frac{1}{U}\hat x)
 \geq \overline{f}(\hat x).
\]
The resulting inequality
$ \sum_{i \in I} u_{i} f(d\sp{(i)})  \geq \overline{f}(\hat x)$
is also true when $U=0$.
Similarly, we obtain
$ \sum_{j \in J} v_{j} g(\hat d\sp{(j)})  \leq \overline{g}(\hat x)$,
whereas
$\overline{f}(\hat x) -  \overline{g}(\hat x) \geq 0$
by the assumption 
$\overline{f} \geq \overline{g}$.
Therefore, \eqref{comsubgICprf3} holds.
\qedJIAM
\end{proof}


\subsection{Connection to min-max theorems on bisubmodular functions}
\label{SCbisubbox}

Let $N = \{ 1,2,\ldots, n  \}$ and 
denote by $3\sp{N}$ the set of all pairs $(X,Y)$ of disjoint subsets $X, Y$ of $N$,
that is,
$3\sp{N} = \{ (X,Y)  \mid  X, Y \subseteq N, \  X \cap Y = \emptyset \}$.
A function $f: 3\sp{N} \to \RR$ is called {\em bisubmodular} if
\begin{align*} 
& f(X_{1}, Y_{1}) + f(X_{2}, Y_{2}) 
\\ & \geq
 f(X_{1} \cap X_{2}, Y_{1} \cap Y_{2}) +
 f((X_{1} \cup X_{2}) \setminus (Y_{1} \cup Y_{2}), 
  (Y_{1} \cup Y_{2}) \setminus  (X_{1} \cup X_{2}) )
\end{align*} 
holds for all 
$(X_{1}, Y_{1}), (X_{2}, Y_{2})  \in 3\sp{N}$. 
In the following we assume 
$f(\emptyset,\emptyset) =0$.
The associated {\em bisubmodular polyhedron} is defined by
\[
P(f) = \{  z \in \RR\sp{n} \mid  
  z(X) - z(Y) \leq f(X,Y) \ \ \mbox{for all\ } (X,Y) \in 3\sp{N} \},
\]
which, in turn, determines $f$ by
\begin{equation}
f(X,Y) = \max \{ z(X) - z(Y)  \mid z \in P(f) \} 
\qquad ((X,Y) \in 3\sp{N}).
\label{bisubfnfromXY}
\end{equation}
If $f$ is integer-valued, $P(f)$ is an integral polyhedron.	
The reader is referred to \cite[Section 3.5(b)]{Fuj05book} and \cite{Fuj14bisubmdc}
for bisubmodular functions and polyhedra.

In a study of $b$-matching degree-sequence polyhedra,
Cunningham--Green-Kr{\'o}tki \cite{CG91degseq}
obtained a min-max formula for the maximum component sum
$z(N) = \sum_{i \in N} z_{i}$ of $z \in P(f)$ 
upper-bounded by a given vector $w$.

\begin{theorem}[{\cite[Theorem~4.6]{CG91degseq}}] \label{THcungreen}
Let 
$f: 3\sp{N} \to \RR$ 
be a bisubmodular function 
with $f(\emptyset,\emptyset) =0$, and $w \in \RR\sp{n}$.
If there exists $z \in P(f)$ with $z \leq w$, 
then
\begin{align} 
&\max\{ z(N) \mid  z \in P(f), \  z \leq w  \}
\nonumber \\ & 
= \min\{ f(X,Y) + w(N \setminus X) + w(Y) 
  \mid  (X, Y) \in 3\sp{N} \}.
\label{minmaxCG}
\end{align} 
Moreover, if $f$ and $w$ are integer-valued, then there exists an integral vector $z$
that attains the maximum on the left-hand side of \eqref{minmaxCG}.
\finboxARX
\end{theorem}

The min-max formula \eqref{minmaxCG}
can be extended to a box constraint (with both upper and lower bounds on $z$).
This extension is given in \eqref{FPminmax} below.
Although this formula is not explicit in 
Fujishige--Patkar \cite{FP94},
it can be derived without difficulty from the results of \cite{FP94};
see Remark~\ref{RMminmaxFPder}.

\begin{theorem}[\cite{FP94}] \label{THfujpat}
Let 
$f: 3\sp{N} \to \RR$
be a bisubmodular function 
with $f(\emptyset,\emptyset) \allowbreak  =0$,
and $\alpha$ and  $\beta$ be real vectors with $\alpha \leq \beta$.
If there exists $z \in P(f)$ with $\alpha \leq z \leq \beta$,
then, for each $(A,B) \in 3\sp{N}$, we have
\begin{align}
&\max \{ z(A) - z(B)  \mid z \in P(f) , \alpha \leq z \leq \beta \} 
\nonumber \\ & 
= \!
\min \{ f(X,Y) \! + \! \beta(A \setminus X) \! + \! \beta(Y \setminus B)
    \! - \! \alpha(B \setminus Y) 
    \! - \! \alpha(X \setminus A) 
     \mid (X,Y) \in 3\sp{N} \}.
\label{FPminmax}
\end{align}
Moreover, if $f$, $\alpha$, and $\beta$ are integer-valued, 
then there exists an integral vector $z$
that attains the maximum on the left-hand side of \eqref{FPminmax}.
\finboxARX
\end{theorem}

Theorem~\ref{THfujpat} can be derived
from Theorem~\ref{THfencICsep}
as follows.
Let $\hat f$ denote the convex extension of 
the given bisubmodular function 
$f: 3\sp{N} \to \RR$,
as defined by Qi \cite[Eqn (5)]{Qi88}
as a generalization of the Lov{\'a}sz extension of a submodular function.
This function 
$\hat f: \RR\sp{n} \to \RR$ 
is a positively homogeneous convex function and 
it is an extension of $f$ in the sense that 
$\hat f (\unitvec{X}-\unitvec{Y}) = f(X,Y)$
for all $(X,Y) \in 3\sp{N}$.
It follows from the positive homogeneity of $\hat f$ 
and Lemma~11 of \cite{Qi88} that
\[
 \frac{1}{2}\left(\hat f(x) + \hat f(y)\right) \geq 
 \hat f\left(\frac{x+y}{2}\right) 
\qquad (x, y \in \ZZ\sp{n}).
\]
This implies, by Theorem~\ref{THfavtarProp33},
that the function $\hat f$ restricted to $\ZZ\sp{n}$
is an integrally convex function.
In the following we denote
the restriction of $\hat f$ to $\ZZ\sp{n}$
also by $\hat f$.

Fix $(A,B) \in 3\sp{N}$ and let
$C = N \setminus (A \cup B)$.
We define a separable concave function 
$\Psi(x) = \sum_{i \in N} \psi_{i}(x_{i})$ 
with $\psi_{i}: \ZZ \to \RR$ given as follows:
For $i \in A$,
\begin{equation}\label{sepconcaveA}
 \psi_{i}(k) =
\begin{cases} 
 \alpha_{i}(k-1)  &  (k \geq 1),  
\\ 
 \beta_{i}(k-1) & (k \leq 1);
\end{cases}
\end{equation}
For $i \in B$,
\begin{equation}\label{sepconcaveB}
 \psi_{i}(k) =
\begin{cases} 
 \alpha_{i}(k+1)  &  (k \geq -1),  
\\ 
 \beta_{i}(k+1) & (k \leq -1);
\end{cases}
\end{equation}
For $i \in C$,
\begin{equation}\label{sepconcaveC}
 \psi_{i}(k) =
\begin{cases} 
 \alpha_{i} k  &  (k \geq 0),  
\\ 
 \beta_{i} k & (k \leq 0) .
\end{cases} 
\end{equation}

We apply
Theorem~\ref{THfencICsep}
to the integrally convex function $\hat f$
and the separable concave function $\Psi$.
For these functions the min-max formula 
\eqref{minmaxICsepR} reads
\begin{align} 
 \min \{ \hat f(x) - \Psi(x) \mid  x \in \ZZ\sp{n}  \} 
= \max \{ \Psi\sp{\circ}(p) - \hat f\sp{\bullet}(p) \mid  p \in \RR\sp{n} \} , 
\label{minmaxBisub} 
\end{align}
where \eqref{minmaxICsepZ} gives integrality of $p$ in the integer-valued case.
In the following we show 
\begin{align}
&
\mbox{$\min$ in \eqref{minmaxBisub}} = \mbox{$\min$ in \eqref{FPminmax}} ,
\label{FPmin}
\\ & 
\mbox{$\max$ in \eqref{minmaxBisub}} = \mbox{$\max$ in \eqref{FPminmax}} 
\label{FPmax}
\end{align}
to obtain ``$\max = \min$'' in \eqref{FPminmax}.

The proof of \eqref{FPmin} consists of showing two equations 
\begin{align}
& \min \{ \hat f(x) - \Psi(x) \mid x \in \ZZ\sp{n}\}
=  \min \{ \hat f(x) - \Psi(x) \mid x \in \{-1, 0, +1 \}\sp{n}\} ,
\label{min01=minZ}
\\ &
 \min \{ \hat f(x) - \Psi(x) \mid x \in \{-1, 0, +1 \}\sp{n} \}
= \mbox{$\min$ in \eqref{FPminmax}} .
\label{RHS-FPminmax-fPsi}
\end{align}
We first show \eqref{RHS-FPminmax-fPsi}
while postponing the proof of \eqref{min01=minZ}.
On identifying a vector $x \in \{-1, 0, +1 \}\sp{n}$
with $(X,Y) \in 3\sp{N}$ by $x = \unitvec{X}-\unitvec{Y}$,
we have
$\hat f(x) = \hat f (\unitvec{X}-\unitvec{Y}) = f(X,Y)$
and
\[
 -\Psi(x) = \beta(A \setminus X) + \beta(Y \setminus B) - \alpha(B \setminus Y) 
    - \alpha(X \setminus A) ,
\]
which is easily verified from \eqref{sepconcaveA}--\eqref{sepconcaveC}.
Hence follows \eqref{RHS-FPminmax-fPsi}.

Next we turn to \eqref{FPmax} for the maximum
in \eqref{minmaxBisub}. 
The conjugate function 
$\hat f\sp{\bullet}$ is equal to 
the indicator function of $P(f)$.
That is,
$\hat f\sp{\bullet}(p)$
is equal to $0$ if $p \in P(f)$, and $+\infty$ otherwise.
The concave conjugate
$\Psi\sp{\circ}$ 
is given by
\[
 \Psi\sp{\circ}(p) =
\begin{cases} 
 p(A)-p(B)  &  (\alpha \leq p \leq \beta),  
\\ 
 -\infty & (\mbox{\rm otherwise})
\end{cases} 
\]
for $p \in \RR\sp{n}$.
Indeed, 
$\dom \psi\sp{\circ}_{i} = [\alpha_{i}, \beta_{i}]_{\RR}$
for all $i \in N$
and, 
for $l \in [\alpha_{i}, \beta_{i}]_{\RR}$,
we have
\[
 \psi\sp{\circ}_{i}(l) =
\begin{cases} 
 l  &  (i \in A),  
\\ 
 -l  &  (i \in B),  
\\ 
 0  &  (i \in C)  
\end{cases}
\]
from \eqref{sepconcaveA},  \eqref{sepconcaveB}, and \eqref{sepconcaveC}.
Therefore, 
we have
\begin{align}
 &
\max \{ \Psi\sp{\circ}(p) - \hat f\sp{\bullet}(p) \mid p \in \RR\sp{n} \} 
\nonumber \\
 &= \max\{ p(A) - p(B) \mid p \in P(f), \alpha \leq p \leq \beta\} 
 = \mbox{$\max$ in \eqref{FPminmax}} ,
\label{BsubminmaxICfnSpfn} 
\end{align}
where the variable $p$ corresponds to $z$ in \eqref{FPminmax}.

It remains to show \eqref{min01=minZ}.
The function 
$\hat f - \overline{\Psi}$
is a polyhedral convex function
and is bounded from below
since the value of \eqref{BsubminmaxICfnSpfn} is finite.
Therefore, $\hat f - \Psi$ has a minimizer.
Let $\hat x \in \ZZ\sp{n}$ be a minimizer of $\hat f - \Psi$
with $\| \hat x \|_{\infty}$ minimum.
To prove by contradiction, assume  
$\| \hat x \|_{\infty} \geq 2$.
Define
$ U = \{ i \in N \mid \hat x_{i} = \| \hat x \|_{\infty}\}$
and 
$V = \{ i \in N \mid \hat x_{i} = -\| \hat x \|_{\infty}\}$,
and let $d = \unitvec{U} - \unitvec{V}$.
By \eqref{sepconcaveA}--\eqref{sepconcaveC},
each $\psi_{i}$ is a linear (affine) function 
on each of the intervals $(-\infty,-1]$ and $[+1,+\infty)$.
Combining this with the fundamental property of 
the extension $\hat f$,
we see that there exists a vector 
$q \in \RR\sp{n}$ 
for which
\[
 (\hat f - \Psi)(\hat x \pm d) 
=  (\hat f - \Psi)(\hat x) \pm (f(U,V) - \langle q, d \rangle) 
\] 
holds, where the double-sign corresponds.
Since 
$\hat x$ is a minimizer of $\hat f - \Psi$,
we must have 
$f(U,V) - \langle q, d \rangle = 0$.
This implies, however, that
$\hat x - d$ is also a minimizer of $\hat f - \Psi$,
whereas we have $\| \hat x - d\|_{\infty} < \|\hat x\|_{\infty}$, a contradiction.
We have thus completed the derivation of 
Theorem~\ref{THfujpat} from Theorem~\ref{THfencICsep}.

\begin{remark} \rm \label{RMminmaxFPder}
The min-max formula \eqref{FPminmax}
can be derived from the results of \cite{FP94} as follows.
Given 
$\alpha, \beta \in \RR\sp{n}$ 
with $\alpha \leq \beta$,
we can consider a bisubmodular function
$w_{\alpha \beta}$ defined by
$w_{\alpha \beta}(X,Y) = \beta(X) - \alpha(Y)$
for disjoint subsets $X$ and $Y$.
The convolution of $f$ with $w = w_{\alpha \beta}$ is defined (and denoted) as 
\begin{align}
&
(f \circ w)(A,B) 
\nonumber \\ & = \!
\min \{ f(X,Y) + w(A \setminus X, B \setminus Y) + w(Y \setminus B,X \setminus A) 
      \mid (X,Y) \in 3\sp{N} \}
\nonumber \\ & = \!
\min \{ f(X,Y) \! + \! \beta(A \setminus X) \! + \! \beta(Y \setminus B) 
  \! - \! \alpha(B \setminus Y) \! - \! \alpha(X \setminus A) 
     \mid (X,Y) \in 3\sp{N} \}.
\label{fwdef}
\end{align}
This function is bisubmodular \cite[Theorem~3.2]{FP94}.
By \eqref{bisubfnfromXY} applied to $f \circ w$,
we obtain 
\begin{equation}
(f \circ w)(A,B) = \max \{ z(A) - z(B)  \mid z \in P(f \circ w) \} 
\qquad ((A,B) \in 3\sp{N}).
\label{convfw}
\end{equation}
On the other hand, Theorem~3.3 of \cite{FP94} shows
\begin{equation}
P(f \circ w) = P(f) \cap P(w) 
= \{  z \mid z \in P(f), \alpha \leq z \leq \beta \} .
\label{PfwPfPw}
\end{equation}
By substituting this expression 
into $P(f \circ w)$ on the right-hand side of \eqref{convfw}
we obtain
\begin{align}
(f \circ w)(A,B) 
&= 
\max \{ z(A) - z(B)  \mid z \in P(f \circ w) \} 
\nonumber \\
&=\max \{ z(A) - z(B)  \mid z \in P(f), \alpha \leq z \leq \beta \}.
\label{fwXYmax}
\end{align}
The combination of \eqref{fwdef} and \eqref{fwXYmax}
gives the desired equality \eqref{FPminmax}.

Finally we mention that the paper \cite{FP94}
considers a more general setting where 
$\alpha \in (\RR \cup \{ -\infty \})\sp{n}$,
$\beta \in (\RR \cup \{ +\infty \})\sp{n}$, and 
$f$ is a bisubmodular function 
defined on a subset $\mathcal{F}$ of $3\sp{N}$ such that 
\begin{align*} 
& (X_{1}, Y_{1}), (X_{2}, Y_{2})  \in \mathcal{F} 
\ \Longrightarrow \ 
 (X_{1} \cap X_{2}, Y_{1} \cap Y_{2}) \in \mathcal{F},
\\ &
 (X_{1}, Y_{1}), (X_{2}, Y_{2})  \in \mathcal{F} 
\ \Longrightarrow \ 
 ((X_{1} \cup X_{2}) \setminus (Y_{1} \cup Y_{2}), 
  (Y_{1} \cup Y_{2}) \setminus (X_{1} \cup X_{2}) )
\in \mathcal{F}.
\end{align*} 
The min-max formula \eqref{FPminmax} remains true in this general case.
\finbox 
\end{remark}




\noindent {\bf Acknowledgement}. 
The authors are thankful to Satoko Moriguchi and Fabio Tardella
for recent joint work on integrally convex functions,
and to Satoru Fujishige for a suggestive comment 
that led to Section \ref{SCbisubbox}.
This work was supported by JSPS/MEXT KAKENHI JP20K11697, 
JP20H00609, and JP21H04979.





\newpage




\newpage
\tableofcontents


\end{document}